%% file: rigspec.tex
\documentclass[11pt]{article}
\usepackage{latexsym,amsmath,amsthm,amssymb,amscd,amsfonts,enumitem,bm}
\usepackage{epsfig}
\usepackage{nicefrac}
\usepackage[all]{xy}
\usepackage{graphicx}
\usepackage[dvipsnames]{xcolor}

\newtheorem{lemma}{Lemma}[section]
\newtheorem{prop}[lemma]{Proposition}
\newtheorem{rem}[lemma]{Remark}
\newtheorem{rems}[lemma]{Remarks}
\newtheorem{theorem}[lemma]{Theorem}
\newtheorem{definition}[lemma]{Definition}
\newtheorem{claim}[lemma]{Claim}
\newtheorem{corollary}[lemma]{Corollary}
\newtheorem{example}[lemma]{Example}

\newtheorem{question}{Question}

\newtheorem{thm}{Theorem}
\newtheorem{proposition}[thm]{Proposition}
\newtheorem{thm-repeat}{Theorem}

\newtheorem{conj}[thm]{Conjecture}

\newtheorem*{defn*}{Definition}

\newtheorem*{thm*}{Theorem}
\newtheorem*{exs*}{Examples}
\newtheorem*{question*}{Question}

\newcommand{\R}{\mathbb{R}}
\newcommand{\C}{\mathbb{C}}

\renewcommand{\P}{\mathbb{P}}
\newcommand{\N}{\mathbb{N}}
\newcommand{\Z}{\mathbb{Z}}

\newcommand{\T}{\mathbb{T}}

\newcommand{\cc}{\mathbb{\mathcal C}}
\newcommand{\cb}{\mathbb{\mathcal B}}
\newcommand{\ca}{\mathcal A}
\newcommand{\cw}{\mathcal W}

\newcommand{\ce}{\mathcal E}
\newcommand{\cd}{\mathcal D}

\newcommand{\ci}{\mathcal I}
\newcommand{\cm}{\mathcal M}

\newcommand{\cf}{\mathcal{F}}

\newcommand{\cj}{\mathcal{J}}
\newcommand{\cp}{\mathcal{P}}

\newcommand{\id}{\textnormal{Id}\,}

\newcommand{\nbd}{neighbourhood }
\newcommand{\nbds}{neighbourhoods }
\newcommand{\fonction}[5]
{$$
\begin{array}{rcccl}
 #1 & : & #2 & \longrightarrow &#3 \\
    &   & #4 & \longmapsto &#5
\end{array}
$$}

\newcommand{\diff}{\textnormal{Diff}\,}

\newcommand{\reg}{\textnormal{\tiny Reg}}

\newcommand{\loc}{\textnormal{loc}}

\newcommand{\priv}{\backslash}
\newcommand{\lra}{\longrightarrow}
\newcommand{\hra}{\hookrightarrow}

\newcommand{\om}{\omega}
\newcommand{\eps}{\varepsilon}
\renewcommand{\phi}{\varphi}

\newcommand{\st}{\textnormal{st}}

\newcommand{\supp}{\text{Supp}\,}

\newcommand{\wdt}[1]{\widetilde{#1}}
\newcommand{\cqfd}{\hfill $\square$ \vspace{0.1cm}\\ }
\newcommand{\sbull}{{\tiny $\bullet$ }}
\newcommand{\ds}{\displaystyle}
\newcommand{\im}{\textnormal{Im}\,}

\renewcommand{\bar}[1]{\overline{#1}}

\newcommand{\nf}[2]{{\nicefrac{#1}{#2}}}

\renewcommand{\Im}{\textnormal{Im}\,}
\newcommand{\Om}{\Omega}

\newcommand{\its}{\item[\sbull]}

\newcommand{\cyl}{\textnormal{Cyl}}

\newcommand{\op}{\textnormal{Op}\,}

\newcommand{\leb}{{\cal L}eb}
\newcommand{\crit}{\textnormal{Crit}}

\newcommand{\met}{{\cal M}\textnormal{et}}

\newcommand{\cz}{\textnormal{CZ}}

\newcommand{\cliff}{\textnormal{Cliff}}

\def\eps{\varepsilon}

\makeatletter \@addtoreset {equation}{section}

\renewcommand\theequation
  {\ifnum \c@subsection>\z@ \arabic{section}.\arabic{subsection}.\arabic{equation}
  \else \arabic{section}.\arabic{equation} \fi}
\makeatother

\begin{document}

\title{\vspace*{-2cm}$\cc^0$-rigidity of Lagrangian submanifolds and punctured\\
holomorphic discs in the cotangent bundle}

\renewcommand{\thefootnote}{\arabic{footnote}}

\author{Cedric Membrez\footnote{Partially supported by the Swiss National Science
Foundation grant 155540 and the European Research Council Advanced grant 338809}
\, and Emmanuel Opshtein}

\date{\today}

\maketitle

\begin{abstract}
Our main result is the $\cc^0$-rigidity of the area spectrum and the Maslov class of
Lagrangian submanifolds. This relies on the existence of punctured pseudoholomorphic
discs in cotangent bundles with boundary on the zero section, whose boundaries
represent any integral homology class. We discuss further applications
of these punctured discs in symplectic geometry.
\end{abstract}


\section{Introduction}
The main theorem of this paper is a $\cc^0$-rigidity result for the area spectrum and Maslov class of
Lagrangian submanifolds. It can be stated in the language of $\cc^0$-symplectic geometry.
We recall the following definition due to Eliashberg-Gromov:
\begin{defn*}
A symplectic homeomorphism $h: (M, \omega) \to (M', \omega')$ is a homeomorphism which
is a $\cc^0$-limit of symplectic diffeomorphisms.
\end{defn*}
The Eliashberg-Gromov $\cc^0$-rigidity Theorem shows 
that a smooth symplectic homeomorphism is a symplectic diffeomorphism. This definition
raises the question of determining how close symplectic homeomorphisms are to their
smooth cousins. In \cite{buop} it was suggested to study this question from the perspective
of the action of symplectic homeomorphisms on submanifolds.
\begin{question}\label{q:rigsubm}
Let $N\subset (M,\om)$ be a submanifold and assume that its image by a symplectic
homeomorphism is a smooth submanifold $N'\subset (M',\om')$. Is $N'$ symplectomorphic
to $N$? Which classical symplectic invariants must coincide for $N$ and $N'$?
\end{question}
The question is far from having an easy and definitive answer as demonstrated by the
collection of results showing $\cc^0$-rigidity (e.g. \cite{lasi,opshtein,hulese,hulese2})
or in contrast $\cc^0$-flexibilty (e.g. \cite{buop,buhuse}). In the large the picture
suggested by these works is that $\cc^0$-rigidity prevails for coisotropic submanifolds
(although several of their invariants have not yet undergone investigation), while it
fails completely for most others.
For instance in \cite{buop} a symplectic homeomorphism of $\R^6$ is constructed that maps
an open symplectic disk to another smooth symplectic disk with half the symplectic area.
On the other hand, from \cite{lasi} we know that if
$L\subset M$ is a closed Lagrangian submanifold and $h(L)$ is smooth for a
 symplectic homeomorphism $h$, then $h(L)$ is Lagrangian. Our main
result is that in this situation the area spectrum and the Maslov class of $L$ and $h(L)$ coincide.
More precisely, given a Lagrangian submanifold $L\subset (M,\om)$, define its
area homomorphism
\fonction{\ca_\om^L}{H_2(M,L)}{\R}{\sigma}{\int_\Sigma \om, \;\text{ where }\;
[\Sigma]=\sigma,}
and its Maslov index $\mu_L:\pi_2(M,L)\to \Z$ (see \S \ref{sec:defmaslov} for the definitions of Maslov indices).
\begin{thm}\label{thm:rigspec}
Let $h: (M, \omega) \to (M', \omega')$ be a symplectic homeomorphism that sends a
closed Lagrangian submanifold $L$ to a smooth (hence Lagrangian) submanifold $L'$.
Then \vspace{-,2cm}
\begin{itemize}
\item[a)] $\ca_{\om'}^{L'}=h_*\ca_{\om}^{L}$,\vspace{-,2cm}
\item[b)] $\mu_{L'}=h_*\mu_L$.
\end{itemize}
\end{thm}
It is not hard to see that theorem \ref{thm:rigspec}.a) provides a positive answer to 
question \ref{q:rigsubm} for Lagrangian submanifolds, when the Nearby Lagrangian 
conjecture holds for $T^*L$ (see proposition \ref{prop:nlctoriglag}). This conjecture is 
known to hold in  $T^*\T^2$ \cite{digoiv}, so symplectic homeomorphisms act on 
$2$-dimensional tori in the same way as symplectic diffeomorphisms. For instance, 
a symplectic homeomorphism of $\C\P^2$ cannot take the Chekanov torus to the 
Clifford torus.

The $\cc^0$-rigidity of the area homomorphism  of Lagrangian tori has already been 
proved in \cite{buop}, but even follows from a theorem by Benci-Sikorav (see \cite{sikorav}). 
In its general form, a proof of theorem \ref{thm:rigspec} can be obtained as a consequence 
of deep results by Abouzaid on Lagrangian submanifolds in cotangent bundles \cite{abouzaid}. 
We briefly discuss this approach. Given a closed manifold $L$ endowed with a Riemannian 
metric $g$, we define $\ell_g^{\min}(\beta)$ to be the length of the minimizing geodesic in a 
class $\beta\in H_1(L;\Z)$. In the cotangent bundle $\pi: T^*L \to L$  we also define
$$
\cw(L,g,r) := \{\,\Vert p\Vert_g < r \,\} := \{\,(q,p) \in T^*L \;|\;
\Vert p\Vert_{g_q} < r \,\} \,\subset\, T^*L,
$$
where $\Vert \cdot\Vert_{g_q}$ is the natural dual norm on $T_q^*L$.
\begin{thm}\label{thm:closeclose++}
Let $\iota: L' \hra (T^*L, d\lambda)$ be a Lagrangian embedding in the cotangent
bundle. Assume $\pi\circ\iota : L' \to L$ induces an isomorphism in homology and
$L'\subset \cw(L,g,\eps)$ for some $\eps > 0$ and choice of metric $g$ on $L$. 
Then for all $\beta' \in H_1(L';\Z)$ we have \vspace{-,2cm}
\begin{itemize}
\item[a)] $\left| \iota^*\lambda(\beta') \right|\leq \eps \ell_g^{\min}(\pi_*\circ\iota_*\beta')$
 \;\;(\cite{paposi,abouzaid,amohol}),\vspace{-,2cm}
\item[b)] $\mu_{\iota(L')}(\iota_*\beta')=0$ \;\;(\cite[Appendix E by Abouzaid]{kragh}).
\end{itemize}
\end{thm}

By the Weinstein neighborhood theorem, this local result implies that if $L,L'$ 
are two Lagrangian submanifolds of a symplectic manifold $M$ satisfying the assumptions of the theorem,
the area homomorphisms of $L$ and $L'$ are $\eps$-close
(see \cite[Lemma 5.1]{buop}) and their Maslov classes coincide, which in turn implies 
theorem \ref{thm:rigspec}. In \cite[Theorem 1.10]{paposi} theorem
\ref{thm:closeclose++}.a) is proved for Lagrangian submanifolds which are
Lagrangian isotopic to the zero section in $T^*L$. It is deduced from the
existence of a {\it graph selector} obtained via generating functions. 
Theorem \ref{thm:closeclose++}.a) would therefore follow from the Nearby Lagrangian conjecture and this result. 
In general, using a result of Abouzaid \cite{abouzaid} that describes the wrapped Floer homology
of $L'$ with a fiber of the cotangent bundle, a graph selector for
$L'\subset T^*L$ is constructed in \cite{amohol} whenever $\pi:L'\to L$
induces an isomorphism in homology. This graph selector can therefore be
used exactly as in \cite[Theorem 1.10]{paposi} to prove theorem
\ref{thm:closeclose++}.a) and hence theorem \ref{thm:rigspec}.a).
This approach puts the $\cc^0$-rigidity of the area homomorphism in the framework of
Abouzaid's work on wrapped Floer homology and Fukaya categories.

In the present paper we choose to explain another  approach, that
lies more within the classical framework of pseudoholomorphic techniques
in symplectic geometry. It does not permit us to prove theorem
\ref{thm:closeclose++} though, but only a weaker version, which is still sufficient
for our purpose of proving theorem \ref{thm:rigspec}.
\begin{thm}\label{thm:closeclose}
Let $(L, g)$ and $(L', g')$ be two closed Riemannian manifolds and
$\iota: L' \hra (T^*L, d\lambda)$ a Lagrangian embedding such that
$\pi_* \circ \iota_*: H_1(L') \to H_1(L)$ is an isomorphism. Assume
that $\iota$ extends to a symplectic embedding $\ci$ of a neighborhood
of $L'$ such that
\[
L \subset \ci(\cw(L',g',\eps')) \subset \cw(L,g,\eps) \subset T^*L
\]
for some $\eps, \eps' > 0$. Then for all $\beta' \in H_1(L';\Z)$ we have\vspace{-,2cm}
\begin{itemize}
\item[a)] $|\iota^*\lambda(\beta')| \leq \eps \ell_g^{\min}(\pi_*\circ \iota_* \beta')$,\vspace{-,2cm}
\item[b)] $\mu_{\iota(L')}(\iota_*\beta')=0$.
\end{itemize}
\end{thm}

We deduce theorem \ref{thm:rigspec} from theorem \ref{thm:closeclose} in 
section \ref{sec:rigspec}.
We obtain theorem \ref{thm:closeclose} and hence theorem \ref{thm:rigspec}
from a technical result about the existence of punctured holomorphic discs
with boundary on the zero section of a cotangent bundle, whose boundaries
represent any given non-zero homology class.
More precisely, we prove the following statement (see section \ref{sec:existence}
for the relevant definitions).
\begin{thm}\label{thm:existence-intro} Let $(L,g)$ be a closed Riemannian manifold
and $\beta \in H_1(L;\Z)$. Assume that $g$ is generic, in the sense that it
has a unique minimizing geodesic $\gamma(\beta)$ in class $\beta$. Then,
for every almost complex structure $J$ on $T^*L$ compatible with $d\lambda$
and $g$-cylindrical at infinity, there exists a $J$-holomorphic map
$u:(D\priv \{0\},\partial D)\to (T^*L, L)$ asymptotic to a lift of
$\gamma(\beta)$ at $0$.
\end{thm}
This theorem might be interesting in its own right and we present further
direct applications in section \ref{sec:applications}. In particular we define
a Poisson bracket invariant for Lagrangian embeddings and compute this
invariant in a special case. Theorem \ref{thm:existence-intro} raises several 
questions such as the existence of other holomorphic curves in cotangent 
bundles, uniruledness, etc., which could lead to further applications.

More interest in the present approach might come from the existence of a 
relevant assumption for questions about Lagrangian rigidity in cotangent bundles.
Theorem \ref{thm:closeclose++} is a deep result that concerns any Lagrangian 
submanifold in a cotangent bundle, but is not easy to prove. On the other hand, 
for those submanifolds $L'\subset T^*L$ that contain the zero section in some 
Weinstein neighbourhood, theorem \ref{thm:closeclose} provides the same conclusion 
as theorem \ref{thm:closeclose++}, but with a relatively easy proof. Hence we 
would like to propose the following weakened version of the Nearby Lagrangian 
conjecture.
\begin{conj}\label{conj:nl-}
Let $K\subset T^*L$ be an exact Lagrangian submanifold. Assume that some 
Weinstein \nbd of $K$ contains the zero section $L$. Then $K$ is Hamiltonian 
isotopic to $L$.
\end{conj}
This conjecture implies the strongest possible $\cc^0$-rigidity for Lagrangian 
submanifolds. 
\begin{proposition}\label{prop:nlctoriglag}
Let $h:M\to M'$ be a symplectic homeomorphism that takes a Lagrangian 
submanifold $L$ to a smooth (hence Lagrangian) submanifold $L'$. 
Assume that conjecture \ref{conj:nl-} holds for $T^*L$. 
Then there exists a symplectic 
diffeomorphism $f:M\to M'$ that takes $L$ to $L'$. 
\end{proposition}
We wish to make a final remark in this introduction. The problem of
$\cc^0$-flexibility/rigidity of submanifolds was introduced in
\cite{opshtein,hulese,buop}. At least for the second author of the
present paper, the formalization of these questions in the framework
of $\cc^0$-symplectic geometry came with the hope of starting a study
of a new type of geometry.  Ideally, the  $\cc^0$-rigid properties of
$\cc^0$-rigid submanifolds (e.g. the coisotropic ones) would provide interesting 
invariants. For instance, the fact that the characteristic foliation of a hypersurface 
is $\cc^0$-rigid \cite{opshtein} led to conjecturing that topological 
hypersurfaces in a symplectic
manifold are covered by ``characteristic sets", which would be invariant
under symplectic homeomorphisms. In \cite{hulese}, the concept of a
$\cc^0$-Lagrangian submanifold was introduced: these are topological
$n$-dimensional submanifolds of symplectic manifolds which can be
locally straightened  via symplectic
homeomorphisms to $\R^n\subset \C^n$. In this perspective theorem
\ref{thm:rigspec} raises the question whether $\cc^0$-Lagrangian
submanifolds can be equipped with an area homomorphism invariant
by $\cc^0$-symplectic homeomorphisms.
This is however not the case,
starting from dimension $6$. Let indeed $L := \partial D \times
\partial D \times \partial D \subset \C^3$, where $D$ is the closed
Euclidean disc of radius $1$ centered at $0$ in $\C$. Then $D' :=
D \times \{1\}\times \{1\}$ is a disc of area $\pi$ attached to this
Lagrangian submanifold. By \cite{buop} there exists a symplectic
homeomorphism $h$ of $\C^3$ such that $h(D') = \frac 12 D' :=
D(0,\tfrac 12)\times\{1\}\times \{1\}$. Then $h(L)$ is a
$\cc^0$-Lagrangian submanifold of $\C^3$, but the area of $h(D')$
does not coincide with the area of $D'$. Even worse, since the
homeomorphism $h$ can have support localized in an arbitrary \nbd
of $D'$, the area of smooth discs with boundary on $h(L)$ do not
depend only on the relative homology class of these discs.

\noindent\textbf{Organisation of the paper.} In \S \ref{sec:discussion} we
present yet another short proof of theorem \ref{thm:rigspec} for Lagrangian 
tori based on holomorphic discs, which illustrates our motivation for
establishing theorem \ref{thm:existence-intro}. \S \ref{sec:existence}
is the technical core of the paper and provides the proof of theorem
\ref{thm:existence-intro}. Although this result might not be a surprise
to experts in the field, we have not been able to find details in the
literature and we give a detailed account here. The Fredholm theory for
punctured holomorphic discs in symplectic cobordisms with boundary on
a Lagrangian submanifold is provided in appendix \ref{sec:fredholm} for
the sake of completeness. In \S \ref{sec:rigspec} we establish our
main applications, theorems \ref{thm:rigspec} and \ref{thm:closeclose}.
In \S \ref{sec:applications} we discuss some further, rather
straightforward, applications of theorem \ref{thm:existence-intro},
for instance 
calculations of Poisson bracket invariants of the zero section in
Weinstein neighborhoods.

\noindent\textbf{Aknowledgements.} E.O. wishes to thank Nicolas Vichery and
Frol Zapolsky for their suggestion of using the work of Abouzaid to prove
theorem \ref{thm:rigspec}. He also wishes to very warmly thank Charles
Boubel for many discussions on Riemannian geometry. Part of this work
was realized during a visit of E.O. to the Technion and Tel Aviv University
in fall 2014. He thanks Misha Entov and Lev Buhovski for the invitation
and for many stimulating discussions, in particular on $pb_4$-type
invariants, which ultimately resulted in the present proof of theorem
\ref{thm:rigspec}. C.M. would like to thank Paul Biran and Leonid
Polterovich for enlightening discussions and stimulating questions related
to the work. Both authors would like to thank Chris Wendl very warmly for
help on issues of transversality for punctured holomorphic disks.

\noindent\textbf{Notation.} We work in the smooth category. Unless explicitly
stated, all our manifolds and their structures are smooth. We adopt the
following notation in this paper. It will be regularly recalled in the
course of the paper. Let $(L,g)$ be a closed Riemannian manifold.
\begin{itemize}
\its $L \subset (T^*L, d\lambda)$ denotes the zero section of the cotangent
bundle with its standard Liouville form $\lambda$ and $L$ is a Lagrangian
submanifold.
\its $T_r^*L = \cw(L,g,r) := \{\,(q,p)\in T^*L\,|\, \Vert p\Vert_g<r \,\}$ and
$\cw_g := \cw(L,g,1) = \{\,(q,p)\in T^*L\, |\, \Vert p\Vert_g < 1 \,\}$ is the
unit (co)disk bundle.
\its $M := \partial \cw_g$ is the contact boundary of $\cw_g$ with the contact form
$\alpha := \lambda_{|M}$.
\its $\gamma: [0,\ell] \to L$ is a geodesic with unit speed with respect to $g$
and $\ell_g(\gamma) = \ell$ is the length of $\gamma$.
\its $\ell_g^{\min}(\beta)$ is the minimal length of a closed curve in a class $\beta
\in H_1(L;\Z)$. We call a closed geodesic with this length a minimizing geodesic in
class $\beta$.
\its $\tilde \gamma: [0,\ell] \to M$ is the lift of a unit speed geodesic $\gamma$
to $M$ given by $\tilde \gamma(t) = (\gamma(t), g(\dot\gamma(t), \cdot))$. Note that
\[
\int_{\tilde \gamma} \alpha = \ell_g(\gamma) = \ell.
\]
\its $D = \{ z \in \C \,|\, |z| \leq 1 \}$ is the closed unit disk and $\partial D =
\{ z \in \C \,|\, |z| = 1 \}$. $D(z,R)$ denotes the closed disk of radius $R$ centered
at $z$.
\its $\mathrm{int}\,\Om$ denotes the interior of a set $\Om$.
\its $\mathrm{Op}(X, Y)$ or $\mathrm{Op}(X)$ is an open neighborhood of $X \subset Y$.
\end{itemize}


\section{A source of motivation for theorem \ref{thm:existence-intro}}\label{sec:discussion}

This section is independent from the rest of the paper. It aims at explaining at least
one source of motivation for establishing theorem \ref{thm:existence-intro}
(existence of punctured holomorphic disks). We start the discussion with a proof of
theorem \ref{thm:rigspec}.a) for tori, which we think illustrates our purpose well. As
explained in \cite[lemma 5.1]{buop} it follows directly from the following theorem.

\begin{theorem}
For $ \eps >0$ let $\iota:L\hra \big(S^1\times(-\eps,\eps)\big)^n \subset
(T^*\T^n, d\lambda)$ be a Lagrangian embedding such that $\pi\circ \iota:L\to \T^n$
induces an isomorphism in homology. Let $e_j:=(\pi\circ \iota)^*\big([ 0_{\T^{j-1}}
\times S^1 \times 0_{\T^{n-j}}]\big) \in H_1(L)$. Then
$$
\left| \iota^*\lambda(e_j) \right| \leq \eps.
$$
\end{theorem}
\noindent {\it Proof:} We prove that $\left|\iota^*\lambda (e_1) \right|
\,\leq \,\eps$. Without loss of generality we can assume that $\iota^*\lambda
(e_1) \geq 0$ and $\eps<1$. Put $\iota^*\lambda (e_j) = \eps \Delta_j$ with $\Delta_j
\in \R$ and $\Delta_1 \in [0,\infty)$. Denote by $\lfloor\Delta_1\rfloor \in \N$ the
integer part and by $\text{Frac}(\Delta_1):=\Delta_1-\lfloor\Delta_1\rfloor \in [0,1)$ its fractional part.
Define
\begin{align*}
\kappa & := 1 + \eps \lfloor\Delta_1\rfloor, \\
a_1 & := 1-\eps\text{Frac}(\Delta_1) > 0, \\
a_j & := N\kappa-\eps\Delta_j,
\end{align*}
with $N$ chosen large enough such that $a_j > a_1 > 0$ and $a_j + \eps \Delta_j \geq 2 \kappa$
for $j = 2, \ldots, n$. Consider now
the standard symplectic embedding $\Phi: \big((S^1\times(-\eps,\eps))^n , d\lambda \big)
\hra (\C^n,\om_\st)$ with $\Phi(0_{\T^n}) = S^1(a_1) \times \dots \times S^1(a_n)
\subset \C^n$ ($S^1(a)\subset \C$ being the circle that encloses a disk of area $a$). Then
$\cw:= \im \Phi = \{\bigcup\, S^1(t_1)\times \ldots \times S^1(t_n) \;|\; |t_i-a_i| < \eps \,\}$.
Obviously,
$$
e_d(\cw)\leq a_1+\eps\leq 1+\eps,
$$
and equality in the second inequality can hold only if $a_1 = 1$, i.e. $\Delta_1 \in \N$
($e_d$ denotes the displacement energy in $(\C^n,\om_\st)$). On the other hand, the
discs with boundaries on $\Phi\circ \iota(L)$ representing the classes $e_1, \dots, e_n$
have symplectic areas $a_1+\eps\Delta_1 = 1 + \eps \lfloor\Delta_1\rfloor = \kappa$, $a_2 + \eps \Delta_2
= N\kappa, \dots, a_n + \eps \Delta_n = N\kappa$. Since $\pi \circ \iota$ is an isomorphism
in homology, we therefore see that every disc with boundary on $\Phi\circ \iota(L)$ and positive area has area
equal to a multiple of $\kappa$. By a theorem of Chekanov \cite{chekanov}, we infer that
$$
1 + \eps \lfloor\Delta_1\rfloor \;=\; \kappa \;\leq\; e_d(\Phi\circ \iota (L)) \;\leq\;
e_d(\cw) \;\leq\; 1 + \eps.
$$
We conclude that $\eps \lfloor\Delta_1\rfloor \, \leq \, \eps$, so $\Delta_1 \,\leq\, 1$, 
with equality only if $\Delta_1\in \N$. Thus $\iota^*\lambda (e_1) \leq \eps$. Applying the same 
procedure to the classes $e_2, \ldots, e_n$ we obtain the statement of the theorem.
\cqfd

What makes this proof easy (although it ultimately relies on a deep theorem by Chekanov
\cite{chekanov}), is that we have a very large sample of Lagrangian embeddings of tori in
$\C^n$ 
with holomorphic discs of arbitrary relative areas, so that we can choose which is 
responsible for the displacement energy of the Lagrangian submanifold.
This situation is unfortunately
very specific to tori: there are very few examples of Lagrangian embeddings of a given
manifold into a symplectic manifold, even without speaking of the freedom to choose the area
of holomorphic discs with boundary on this submanifold. On the other hand, a 
manifold always embeds into its cotangent bundle
as the zero section, which is Lagrangian. In this setting there are no symplectic discs with
boundary on the zero section, but punctured holomorphic discs may replace the compact discs
with a similar benefit.


\section{Punctured holomorphic curves in cotangent bundles}\label{sec:existence}

Let $L$ be a closed manifold and $(T^*L\overset{\pi}\to L,\lambda)$
be its cotangent bundle equipped with the Liouville form $\lambda=pdq$.
The cotangent bundle is endowed with an $\R^+_*$-action given by
$\tau\cdot(q,p):=(q,\tau p)$. Let $g$ be a  Riemannian metric on $L$
and $\beta \in H_1(L{;\Z})$. If $g$ is generic, $\beta$ has exactly
one connected representative of minimal length (hence a geodesic),
whose arc-length parametrization is henceforth denoted by $\gamma(\beta)$
(or even $\gamma$ when there is no risk of confusion). For general $g$
there may be several connected minimal representatives and $\gamma(\beta)$ is then
one of them. The metric $g$ induces an isomorphism $\sharp: TL \to T^*L,
v \mapsto v^{\sharp} = g(v, \cdot)$, hence an inner product (still denoted $g$)
and a norm $\Vert \cdot\Vert_g$ on the fibers of $T^*L$. We write for the
unit disk bundle and its boundary	
$$
\cw_g := \cw(L,g,1) := \{\, (q,p) \;|\; q\in L,p\in T_q^*L, \Vert p\Vert_g < 1\,\}
\;\text{ and }\; M:=\partial \cw_g.
$$
Then $\alpha:=\lambda_{|M}$ is a contact form on $M$, whose Reeb vector
field is denoted by $R$. It is well-known that $R$ generates the cogeodesic
flow on $M$: its trajectories are the lifts $\tilde\gamma(t) :=
(\gamma(t),\dot \gamma(t)^\sharp)$ of the unit speed geodesics $\gamma$.
Moreover, $(T^*L, d\lambda)$ is a symplectic cobordism with one positive
end given by the identification
$$
( [1, \infty) \times M,r\alpha) \overset \sim\lra (T^*L \priv \cw_g,\lambda),
\qquad \big(r,(q,p)\big)  \longmapsto  (q,rp).
$$
Note that this map extends to $\big((0, \infty) \times M,r\alpha\big) \simeq
(T^*L \priv L,\lambda)$. The function $r = \Vert p \Vert_g \in [0,\infty)$ is
called the radial coordinate in $T^*L$.  The image of the vector field
$r\frac \partial{\partial r}$ under this identification is Liouville.

An almost complex structure $J$ on $T^*L$ is compatible with $\om:=d\lambda$
if $\om(\cdot,J\cdot)$ is a positive definite quadratic form. An almost complex
structure $J$ on $ {(0,\infty)}\times M$ is compatible with $\alpha$ if $J$
preserves $\ker \alpha_{| \{r\}\times M}$ and is compatible with $d\alpha=\om$
on this subbundle, sends $r \frac\partial{\partial r}$ to the Reeb vector
field $R$ of $(\{r\}\times M,\alpha)$ and is invariant by the $\R^+_*$-action
on $(0, \infty) \times M$. Hence a metric $g$ on $L$ defines on $T^*L$ a
{\it fiberwise convex hypersurface} $M$, which in turn determines a $1$-form
$\alpha$, and {\it in fine} a class of almost complex structures on ${(0,\infty)}
\times M\simeq T^*L\priv L$, namely those which are compatible with $\alpha$.
Let  $\gamma$ be a closed geodesic with unit speed and length $\ell := \ell_g(\gamma)$
and let $\tilde\gamma$ be its lift to $M$. Then for any almost complex structure
$J$ compatible with $\alpha$ it can be shown that
\begin{equation}\label{eq:geod-cyl}
v_{\gamma, g}: \R \times \R / \ell \Z \longrightarrow (0, \infty) \times M,
\quad (s, t) \longmapsto (e^s, \tilde\gamma(t)),
\end{equation}
is $J$-holomorphic for the complex structure $j \partial_s = \partial_t$ on
$\R \times \R / \ell \Z$.

\begin{definition}
An almost complex structure $J$ on $T^*L$ compatible with $d\lambda$ is
\emph{cylindrical at infinity} (with respect to $g$) if it is compatible
with $\alpha$ outside a compact \nbd of the zero section. We denote by
$\cj^\infty_{\cyl,g}$ the set of all such almost complex structures,
where the $\infty$ of the notation aims at reminding that these
structures are constrained only near infinity. A $J$-holomorphic curve $u$
defined on $U\priv\{z_0\}\subset \C$ (where $U$ is an open subset and
$z_0\in U$) is said to be asymptotic to $\tilde\gamma$ at $z_0$
if there exists a biholomorphic
identification of a punctured neighborhood of $z_0$ with $[0, \infty)
\times \R / \ell \Z$ that provides coordinates $(s,t)$ on $U\priv\{z_0\}$ such that
$d(u(s,t),v_{\gamma{,g}}(s,t))\underset{s\to \infty}{\lra} 0$.
Here, $d$ stands for any distance which is $\R^+_*$-invariant on the
cylindrical end.
\end{definition}

The aim of this section is to prove the following result. We say that
a metric $g$ has discrete length spectrum in the class $\beta$ if
the lengths of the geodesics representing this class form a discrete set.
\begin{theorem}\label{thm:existence}
Let $(L,g)$ be a closed Riemannian manifold and $\beta \in H_1(L;\Z)$ a
homology class in which $g$ has discrete length spectrum.
Then for  all  $J \in \cj^\infty_{\cyl,g}$ there exists a map $u_J :
(D\priv \{0\},\partial D) \to (T^*L,L)$ which solves the following problem:
\begin{equation}\label{eq:pj}
\left\{
\begin{array}{l}
du_J\circ j = J(u_J)\circ du_J, \\
{u_{J*}[\partial D]}= \beta, \\
u_J \text{ is asymptotic to  $\tilde\gamma(\beta)$  at $0$.}
\end{array}\right. \tag{$\cp(J,\beta)$}
\end{equation}
In other terms the set ${\widehat\cm(J,\beta)}$ of solutions of $\cp(J,\beta)$
is non-empty for all $J \in \cj^\infty_{\cyl,g}$.
\end{theorem}

\begin{rems}\label{rk:existence}
\begin{itemize}
\item[(i)] $S^1$ acts on $\widehat \cm(J,\beta)$ by holomorphic source
reparametrizations (of $D\priv \{0\}$). We denote in the sequel
$\cm(J,\beta):=\widehat \cm(J,\beta)_{/S^1}$, the space of
unparametrized solutions of  $\cp(J,\beta)$.
\item[(ii)] As already mentioned, when a metric has several minimizing
geodesics in the class $\beta$, $\gamma(\beta)$ stands for some
minimizing geodesic.
\item[(iii)] If the length spectrum is not discrete, the proof of theorem
\ref{thm:existence} will provide for all $\eps>0$ a map $u_{J,\eps}$
asymptotic to a lift of a geodesic in the class $\beta$ of length at most
$\ell_g(\gamma(\beta))+\eps$.
\end{itemize}
\end{rems}

The scheme of proof of theorem \ref{thm:existence} is mostly standard and 
we explain it in the present paragraph, under the  assumption that 
$\gamma(\beta)$ is a primitive geodesic, i.e. not a multiple-cover.
Let $g$ be a metric with discrete length spectrum in the class $\beta$, 
and assume that $\gamma(\beta)$ is primitive.
We first define an almost complex structure $J_h \in \cj^\infty_{\cyl,h}$
for an arbitrary metric $h$ on $L$ (\S \ref{sec:jpart}) which has a
unique $J_{h}$-holomorphic punctured disc $u_{\gamma, h} \in \widehat
\cm(J_h,\beta)$ modulo reparametrization (\S \ref{sec:cmuniqueness}).
We also prove that
 all elements of $\cm(J,\beta)$ for $J\in \cj_{\cyl,h}^\infty$
are somewhere injective (\S \ref{sec:somewhereinj}).
For a suitable perturbation $g_\eps$ of our metric $g$ we prove
the surjectivity of the linearization of the $\bar \partial_{J_{g_\eps}}$-operator at
$u_{\gamma,g_\eps}$ (\S \ref{sec:transversality} and \S \ref{sec:goodmetric}).
In \S \ref{sec:compactnessI} we show compactness: if a sequence of almost 
complex structures $J_n  \in \cj_{\cyl,g_\eps}^\infty$ coincide with $J_{g_\eps}$ 
outside a compact set and converge to $J\in \cj_{\cyl,g_\eps}^\infty$
in the $\cc^\infty$-topology, then a sequence $u_n$ of solutions of $\cp(J_n,\beta)$
has a subsequence that converges to a solution of $\cp(J,\beta)$.
Classical arguments then show that if $\{J_t\}_{t\in[0,1]}$ is a smooth
generic path of almost complex structures from $J_{g_\eps}$ to a generic
$J_n \in \cj_{\cyl, g_\eps}^\infty$
that coincide with $J_{g_\eps}$ outside a compact set, the set
$\cup \{t\}\times \cm(J_t,\beta)$ is a one-dimensional compact cobordism
between the point $\{0\}\times \cm(J_{g_\eps},\beta)$ and $\{1\}\times
\cm(J_n,\beta)$, hence the latter space is non-empty. Applying again compactness
to a generic sequence $J_n$ of approximations of $J$ that coincide with
$J_{g_\eps}$ outside a compact set, we find a solution of
$\cp(J,\beta)$ for all $J\in \cj^\infty_{\cyl, g_\eps}$ that coincide
with $J_{g_\eps}$ at infinity. We finally use an SFT-compactness argument
to obtain the theorem for the metric $g$ itself (\S \ref{sec:proofexistence}).


\subsection{A particular almost complex structure}\label{sec:jpart}
Let $g$ be an arbitrary Riemannian metric on $L$. In this section we describe a
particular almost complex structure on $T^*L$ that is determined by the metric $g$.

The Levi-Civita connection $\nabla$ of $g$ induces a natural connection $\nabla^*$ on
$T^*L$ by the identity
$$
d(\sigma(X)) = \nabla^*\sigma\,(X) + \sigma(\nabla X)
$$
for any 1-form $\sigma$ and vector field $X$ on $L$.
This connection provides a splitting $T(T^*L)=H\oplus F$, where $H$ is the horizontal
distribution of $\nabla^*$ and $F$ is the distribution of the tangent spaces to
the fibers. Note that we have a canonical identification $F_{(q,p)}\simeq T^*_qL$.
We call $\pi:T^*L\to L$ the natural projection and, by abuse of notation, $\pi:H_{(q,p)}
\to T_qL$  the isomorphism  $d\pi_{|H}$, at least when there is no chance of confusion.
In order to define our almost complex structure we first need to explain some relations
between the Riemannian geometry on $L$ and the contact structure on $M$.
Recall that
$R$ denotes the Reeb vector field on $M$, also seen as a vector field on $T^*L\priv
L\simeq (0,\infty)\times M$, and that $r\frac\partial{\partial r}=p$ is the infinitesimal
generator of the $\R^+_*$-action on $T^*L\priv L$. Also notice that $g$ induces an
inner-product on $T_q^*L$, still denoted $g$, by $g_q(u^\sharp,v^\sharp) := g_q(u,v)$
for $u, v \in T_qL$. By definition $M=\{ (q,p) \,|\, g_q(p,p) = 1 \}$.
\begin{lemma}\label{le:riemancontact}  On $T^*L\priv L \simeq (0, \infty)\times M$
we have
\begin{enumerate}
\item $R(q,p)\in H_{(q,p)}$ and $g_q(\pi(R(q,p)),\cdot)=\frac p{\Vert p \Vert_g}
= \frac\partial{\partial r}$ (in other terms, $\pi(R)^\sharp=\frac\partial{\partial r}$).
\item $H$ is a subspace of $T(\{r\}\times M)$ for every $r \in {(0,\infty)}$
and is Lagrangian with respect to $d\lambda$.
\item The map $\sharp\circ \pi$ is an isomorphism between $\ker \alpha\cap H_{(q,p)}$
and $\ker \alpha\cap F_{(q,p)}$. Moreover,
\begin{align*}
\ker \alpha\cap H_{(q,p)} & = \{ \, V\in H_{(q,p)}\; |\; p(\pi(V))=0 \,\}\\
\ker \alpha\cap F_{(q,p)} & = \{ \, V\in F_{(q,p)}\;|\; g_q(V,p)=0 \,\}.
\end{align*}
\item $\ker \alpha=(\ker \alpha\cap H_{(q,p)})\oplus (\ker \alpha\cap F_{(q,p)})$.
\end{enumerate}
\end{lemma}

\noindent{\it Proof:} This is all well-known, but we provide some indications.
\begin{enumerate}
\item Since $R$ generates the cogeodesic flow of speed $1$ on
$M$ (hence on each $\{r\}\times M$ since $R$ does not depend on the $\R^+_*$-coordinate),
$R$ is horizontal, $g_q(\pi(R(q,p)),\cdot) = p$ when $\Vert p \Vert_g = 1$
and equals $\nf p{ \Vert p\Vert_g}$ in general.
\item $H\subset T(\{r\}\times M)$ because parallel transport preserves the metric.
It is also Lagrangian since the Levi-Civita connection has no torsion \cite{grigoryants}.
\item The fact that $\ker \alpha \cap H_{(q,p)} = \{\,V\in H_{(q,p)}\, |\, p(\pi(V))=0\,\}$
and $\ker \alpha\cap F_{(q,p)} = F_{(q,p)} \cap$ $T_{(q,p)}M=\{\,V\in F_{(q,p)}\,|\,
g_q(V,p) =0 \,\}$ are immediate by definition.
Now let $V \in \ker \alpha \cap H_{(q,p)}$. Since $\sharp$ is an isometry by definition
of $g_q$ on $T_q^*L$ and $p(\pi(V))=0$, we have the chain of equalities
$$
g_q(\pi(V)^\sharp,p) = \Vert p\Vert_g \,g_q(\pi(V)^\sharp,\pi(R)^\sharp)
= \Vert p\Vert_g \,g_q(\pi(V),\pi(R)) = p(\pi(V)) = 0,
$$
which means that $\pi(V)^\sharp\in \ker \alpha\cap F_{(q,p)}$. By dimension
considerations we see that $\sharp\circ \pi$ is indeed an isomorphism.
\item This is obvious by dimension considerations, since by (3), we
have $\dim \ker \alpha\cap F_{(q,p)}=n-1=\dim \ker \alpha\cap H_{(q,p)}$. \cqfd
\end{enumerate}

We now choose a smooth non-decreasing function $\chi: (0,\infty) \to (1,\infty)$
with $\chi(r)= 1$ near $0$ and $\chi(r)= r$ near $\infty$. Lemma~\ref{le:riemancontact}
implies that
$$
T_{(q,p)}T^*L =\left\langle \frac\partial{\partial r}, R(q,p) \right\rangle
\oplus \ker \alpha\cap H_{(q,p)}\oplus \ker \alpha\cap F_{(q,p)} \hspace{,5cm}
\forall\, q\in L, p\neq 0.
$$
We define the almost complex structure $J_g$ on $T^*L\priv L$ by
$$
\left\{
\begin{array}{l}
\ds J_g(R)=-\chi(r)\frac\partial {\partial r}\\
\begin{array}{lccc}
J_g: & \ker \alpha\cap H_{(q,p)}& \lra &\ker \alpha\cap F_{(q,p)}\\
 & V & \longmapsto & -g_q(\pi(V),\cdot)
\end{array}\\
J_g^2=-\id.
\end{array}
\right.
$$
By lemma \ref{le:riemancontact}, $J_g$ indeed exchanges $\ker \alpha\cap
H_{(q,p)}$ and $\ker \alpha\cap F_{(q,p)}$: if $V \in \ker \alpha\cap H_{(q,p)}$,
$$
-g_q(J_gV,p) = g_q(\pi(V)^\sharp,p)= \Vert p\Vert_g\, g_q(\pi(V)^\sharp,\pi(R)^\sharp) = \Vert p\Vert_g
\, g_q(\pi(V),\pi(R)) = p(\pi(V))=0.
$$
Since $\chi=1$ near $0$, the first point of lemma~\ref{le:riemancontact} shows that
$J_g$ has the following alternative definition near the zero section: for $V\in H$,
$J_gV=-g(\pi(V),\cdot)\in F$. This obviously shows that $J_g$ extends smoothly to
the zero section, because $H_{(q,p)} \lra T_qL$ as $p\to 0$.
On the zero section we have a canonical isomorphism $T(T^*L)_{|L} \simeq TL
\oplus T^*L$ and $J_g=-\sharp$.

\begin{lemma}
$J_g\in \cj_{\cyl,g}^\infty$.
\end{lemma}
\noindent{\it Proof:}
By definition,  $J_gR(q,p)=-r\frac \partial{\partial r}$ when $\Vert p\Vert_g\gg 1$
and $J_g$ preserves $\ker \alpha$. Since {$H_{(q,p)}$ is Lagrangian, showing
that $J_{g|\ker \alpha}(q,p)$ is compatible with $d\lambda(q,p)$ amounts to proving
that $d\lambda(V,J_g V) > 0$ for $V \in \ker \alpha\cap H_{(q,p)}$}. Moreover, since
$d\lambda$ is invariant by the lifts of diffeomorphisms of $L$ to $T^*L$, a
straightforward computation shows that we can assume $g_{q}=\sum dq_i^2$. Then,
for $V\in \ker \alpha\cap H_{(q,p)}$,
\begin{align*}
d\lambda(V,J_gV)
&= \sum dp_i\wedge dq_i (V,J_gV)\\
& {=-\sum dp_i(J_gV)dq_i(\pi(V))}& {\text{(because $J_gV \in F_{(q,p)}$, so $dq_i(J_gV)=0$)}}\\
& =\sum dp_i(g(\pi(V),\cdot))dq_i(\pi(V))\\
& =\sum dq_i(\pi(V))^2>0& \text{ (because $g_q=\sum dq_i^2$).}
\end{align*}
Finally, to show that $J_g$ is dilation-invariant at infinity,
it is enough to work on $\ker \alpha \cap H$ (because $H=(\ker \alpha \cap H)\oplus
\langle R \rangle$). If $V\in \ker \alpha \cap H$, $J_gV\in F\cap TM$, so the
dilations on $(0,\infty)\times M$ act by the identity on both $V$ and $J_gV$. \cqfd \newline
We will also need the following properties of the almost complex structure $J_g$.
\begin{prop}\label{prop:propJg}
The automorphism of $T^*L$ defined by  $\sigma(q,p):=(q,-p)$ is $J_g$-anti\-ho\-lo\-mor\-phic.
Moreover, the sub-levels $\{\|p\|_g\leq r\}$ are $J_g$-pseudoconvex domains for $r>0$.
\end{prop}
\noindent{\it Proof:}
We consider canonical coordinates $(q,p) \in T^*L$ determined by a local choice of coordinates
$q = (q_1, \ldots, q_n)$ on $L$. One can show that the horizontal distribution $H_{(q,p)}$
of the connection $\nabla^*$ is spanned by the vectors
\begin{equation}\label{eq:hor_lift_vectors}
\frac{\partial}{\partial q_i}^H = \frac{\partial}{\partial q_i} + \sum\limits_{k,l}
\Gamma^k_{il}(q)\, p_k \, \frac{\partial}{\partial p_l}, \quad i = 1, \ldots, n,
\end{equation}
where $\Gamma^k_{il}$ are the Christoffel symbols of the metric $g$. Since $d\sigma
\tfrac{\partial}{\partial p_l} =  -\tfrac{\partial}{\partial p_l}$, we see that
$\sigma$ preserves the lifts $\tfrac{\partial}{\partial q_i}^H$ and hence the
horizontal distribution. A brief calculation then reveals that $\sigma$ is
$J_g$-antiholomorphic.

In $T^*L \priv L$ one easily sees that $J_g$ preserves the contact distribution
$\ker \alpha$ and is compatible with $d\alpha$, hence the sets $\{\|p\|_g\leq r\}$
are $J_g$-pseudoconvex for $r > 0$.
\cqfd


\subsection{Computation of $\cm(J_g,\beta)$}\label{sec:cmuniqueness}

The almost complex structure $J_g$ sends $R$ to a multiple of the radial vector field
$\tfrac{\partial}{\partial r}$, hence for any closed unit speed geodesic
$\gamma: \R / \ell \Z \to L$ the image of the map $v_{\gamma,g}$ of \eqref{eq:geod-cyl}
is $J_g$-holomorphic (its tangent planes are $J_g$-invariant).
We now need a holomorphic parametrization of this set. Recall that
$$
J_g R = -\chi(r)\frac\partial{\partial r},
$$
where $\chi$ is an increasing, weakly convex function with $\chi(r)=1$ for $r$ close to
$0$ and $\chi(r)=r$ outside a compact set. Then, \label{page:defG}
$$
G(u):=\int_{0}^u\frac{dr}{\chi(r)}
$$
is a well-defined strictly increasing function $G: [0,\infty) \to [0, \infty)$ with $G(0)=0$.
Its inverse defines a function $G^{-1}:
[0,\infty) \to [0, \infty)$ that satisfies the differential equation $h'=\chi\circ h$.
A straightforward computation shows that
\begin{equation}\label{eq:defv}
u_{\gamma,g}: (0, \infty)\times \R / \ell \Z \longrightarrow (0, \infty)\times M, \quad
(s,t) \longmapsto (G^{-1}(s), \tilde\gamma(t))
\end{equation}
 is a $J_g$-holomorphic parametrization of
$\im {v_{\gamma,g}}$. One can see that this map extends as a map from the half-cylinder
$[0, \infty)\times \R / \ell\Z$ to $T^*L$, which we also call $u_{\gamma,g}$. The
half-cylinder $([0, \infty)\times \R / \ell\Z, j)$ is conformally equivalent to the
closed punctured disk $(D\priv\{0\}, i)$ via the map $(s,t) \mapsto \exp(-2\pi(s+it)/ \ell)$.
In general we will use the half-cylinder and closed punctured disk interchangeably.
Hence we can consider $u_{\gamma,g}: (D\priv\{0\}, \partial D) \to (T^*L, L)$. It is
asymptotic to $\tilde\gamma$ at the puncture and we have $u_{\gamma,g}(e^{-2\pi i t / \ell})
= \gamma(t) \subset L \subset T^*L$. In other words, $u_{\gamma,g}\in \widehat{\cm}(J_g,\beta)$.
The aim of this section is to prove that $\cm(J_g,\beta)$ consists of the unique
element $\{[u_{\gamma,g}]\}$. We require some preliminary lemmas.

\begin{lemma}\label{le:goodintersectionJg} Let $u:(D\priv\{0\},\partial D)\to (T^*L,L)$
be a $J_g$-holomorphic map asymptotic to $\tilde\gamma$ at $0$. Then
\begin{itemize}
\item[(i)] $u$ is smooth up to $\partial D$,
\item[(ii)] $\crit (u)$ is a finite set of $D\priv\{0\}$,
\item[(iii)] $u^{-1}(L)\cap \mathrm{int}\,D\priv\{0\}$ is countable and can only accumulate at
critical points of $u$. 
\end{itemize}
\end{lemma}
\noindent{\it Proof:} Point (i) is a general and classical fact that relies on the
ellipticity of the $\bar \partial_{J_g}$-operator on the space of curves with boundary
on totally real submanifolds.
Points (ii) and (iii) are specific to our situation. They rely on the special form
of $J_g$ that makes the involution $\sigma(q,p):=(q,-p)$ anti-holomorphic, see
proposition \ref{prop:propJg}.

Let $u:(D\priv\{0\},\partial D)\to (T^*L,L)$  be a $J_g$-holomorphic map asymptotic
to $\tilde\gamma$ at $0$. The map $v:\C\priv \mathrm{int}\,D \to T^*L$ defined by 
$v(z):=\sigma\circ u(\nf 1{\bar z})$ is $J_g$-holomorphic and coincides with $u$ on
$\partial D$ (because $u(\partial D)\subset L=\text{Fix}\, (\sigma)$). Thus, the map
\fonction{w}{\C\priv\{0\}}{T^*L}{z}{\left\{\begin{array}{l} u(z)
\;\text{ if } z\in  D\priv\{0\} \\ v(z) \;\text{ if } z \in \C\priv D \end{array} \right.}
is holomorphic on $\C\priv \partial D$, continuous on $\partial D$, and since $u$
and $v$ are smooth up to $\partial D$, $w$ is $J_g$-holomorphic on $\partial D$ as well.
Thus, $w: \C\priv\{0\}\to T^*L$ is $J_g$-holomorphic. Since $w$ is
asymptotic to $\tilde\gamma$ at $0$, it has no critical points near $0$ and
\cite[Lemma 2.4.1]{mcsa} guarantees that $w$ has only finitely many critical points in
$D \priv \{0\}$, where it coincides with $u$. This proves (ii).

We prove (iii) by contradiction. First assume there exists a sequence $z_n\in
\mathrm{int}\, D\priv\{0\}$ that converges to a point $z \in\mathrm{int}\, D\priv\{0\}$,
which is not a critical point of $u$, and such that $u(z_n) \in L$ for all $n$
(hence $u(z) \in L$).
We define
$$
\Om:=\{\, \xi\in \mathrm{int}\,D\priv\{0\} \;|\; \exists\,\xi'\in \C \priv D :
u(\xi)=v(\xi')\,\}.
$$
In particular, $\Om\supset u^{-1}(L)\cap \mathrm{int}\,D\priv\{0\} \ni z_n$.
By \cite[Lemma 2.4.3]{mcsa}, $\Om$ contains a \nbd of each of its accumulation points in
$\mathrm{int}\,D\priv\{0\}$ that are not critical points of $u$.
It therefore contains a \nbd of $z$
by assumption. Notice that the asymptotic behaviour of $u$ at $0$ implies that $\Om$ does
not meet some \nbd of the puncture $0$. Let now $c:[0,1]\to D$ be a continuous curve between
$c(0)=z$ and $c(1)=0$ that avoids the (finitely many) critical points of $u$.
Let ${t_*} :=\sup\{\,t\;| \; c([0,t])\subset \Om\,\}$. Since $z\in {\mathrm{int}\,\Om}$,
and $c(\op(1))\cap \Om = \emptyset$, we have $0<{t_*}<1$. Let now $t_n < {t_*}$ be any
sequence that converges to $t_*$. By assumption $c(t_n)\in \Om$, so there exists $\xi_n'\in
{\C \priv D}$ such that $v(\xi_n')=u(c(t_n))$. Obviously $\xi_n'$ cannot accumulate at $\infty$
because $v(\xi_n') = u(c(t_n))\to u(c(t_*))\in T^*L$. It can neither accumulate at a point
$\xi'\in {\C \priv D}$ because then, by \cite[Lemma 2.4.3]{mcsa}, $c(t_*)$ would be an
accumulation point of $\Om$, so would be contained in ${\mathrm{int}\,\Om}$. Thus $c(t)$ would belong
to $\Om$ for $t<{t_*}+\eps$, $\eps>0$, which contradicts our definition of ${t_*}$. We therefore
conclude that, after extraction, $\xi_n'\to \xi'\in \partial {D}$, so $u(c(t_*))=\lim u(c(t_n))
=\lim v(\xi_n')=v(\xi')\in v(\partial D)\subset L$. This shows that any curve that joins
$z$ to $0$ in ${D} \priv \crit (u)$ intersects $u^{-1}(L)$. Since $\crit(u)$ is finite,
we conclude that the connected component $U$ of $z$ in ${D}\priv u^{-1}(L)$ does not contain $0$.
Thus $u_{|U}:(U,\partial U)\to (T^*L,L)$ is $J_g$-holomorphic.
If $\partial U$ is smooth, we can conclude by Stokes Theorem.
In general, we can invoke the pseudoconvexity of the hypersurfaces $\{\|p\|_g=r\}$ for $r > 0$
(see proposition \ref{prop:propJg})
to conclude that $u_{|U}$ takes values in $L$. Since $U$ is open and $L$ is totally real,
we conclude that $u_{|U}$ is constant, which implies that $u$ is constant. But this
contradicts the fact that $u$ is asymptotic to $\tilde\gamma$ at $0$.

This contradiction shows that $u^{-1}(L)$ can only accumulate {at} $\crit(u)\cup \partial {D}$.
On the other hand, since $L$ is totally real, any non-critical point of $u$ in $\partial {D}$ is
a point where the intersection of $u$ with $L$ is clean. These points can therefore not
be accumulated by points of $u^{-1}(L)\cap {\mathrm{int}\, D\priv\{0\}}$. This completes
the proof of (iii).\cqfd

In the next lemma, ${D}^+ := {D}\cap\{\im z\geq 0\}$. Recall that $\alpha$ is the
one-form obtained by ${\alpha} := \lambda_{|M}$ on $M$ and is extended by invariance
under the ${\R_*^+}$-action on $T^*L\priv L$. It is not defined on $L$.
\begin{lemma}\label{le:limalpha}
Let $u:(D^+,[-1,1])\to (T^*L,L)$ be a $J_g$-holomorphic map. Assume that $\eps_n\to 0$
verifies $u(D^+\cap \{\im z=\eps_n\})\cap L=\emptyset$. Then,
\begin{equation}\label{eq:limalpha}
\lim_{n\to \infty} \int_{D^+\cap \{\im z=\eps_n\}}u^*\alpha=-\ell_g(u_{|[-1,1]}),
\end{equation}
where $D^+\cap\{\im z=\eps_n\}$ is oriented from left to right.
\end{lemma}
Notice that both sides of equality \eqref{eq:limalpha} are well-defined. The assumption
guarantees that $u^*\alpha$ is well-defined on $D^+\cap \{\im z=\eps_n\}$. Moreover,
by ellipticity of $\bar \partial_{J_g}$, $u$ is smooth up to $[-1,1]$, so $u([-1,1])$
has finite length.

\noindent{\it Proof:} Since both sides of the equality (\ref{eq:limalpha}) are
additive with respect to a partition of $\op([-1,1],D^+)$, we can freely assume that $u$ takes
values in a small chart of $T^*L$ localized on $L$, which amounts to assuming that
$L=\R^n$. Let $\gamma:=u_{|[-1,1]}$  and $\gamma_\eps:=u_{|D^+\cap \{\Im z=\eps\}}$.
Taking into account that for the coordinates $x + iy \in D^+$,
$$
\partial_y u(t,0) = J_g \partial_x u(t,0) = J_g \dot \gamma (t)
=-g_{\gamma(t)}(\dot \gamma(t),\cdot)\in T^*_{\gamma(t)}\R^n,
$$
we have
\begin{align*}
& \dot \gamma_\eps(t)=\dot \gamma (t)+O(\eps),\\
& \gamma_\eps(t)=\big(\gamma(t),-\eps g_{\gamma(t)}(\dot \gamma(t),\cdot)\big)+O(\eps^2).
 \end{align*}
Thus  $p(\gamma_\eps(t)):= p_{\gamma_\eps(t)}
= -\eps g_{\gamma(t)}(\dot \gamma(t),\cdot)+O(\eps^2)$ and $\|p(\gamma_\eps(t))\|_g=\eps\|\dot \gamma(t)\|_g+O(\eps^2)$. 
If $\dot \gamma(t)\neq 0$, this implies
$$
\alpha(\gamma_\eps(t)) := \alpha_{\gamma_\eps(t)} = \frac p{\|p\|_g}(\gamma_\eps(t)) =
-g_{\gamma(t)}\left(\frac {\dot \gamma(t)}{\|\dot \gamma(t)\|_g},\;\cdot\; \right)+O(\eps).
$$
Notice now that, using the reflection argument as in lemma \ref{le:goodintersectionJg},
one immediately sees that the critical set of the restriction of $u$ to any compact set
of $(-1,1)$ is a finite set, so the Lebesgue measure $\leb(\crit(\gamma))=0$. Thus, one
can cover $\crit(\gamma)$ by an open set $J_\delta\subset (-1,1)$ of total length $\delta$.
We can also assume that $J_\delta$ contains some \nbd of $\{-1,1\}$, so that for some
$\eps(\delta)>0$, $\gamma_\eps(t)$ is well-defined for $t\in (-1,1)\priv J_\delta$ for
all $\eps < \eps(\delta)$.
Then, for $\eps_n<\eps(\delta)$, we have
\begin{align*}
\ds\int_{(-1,1)\priv J_\delta+i\eps_n} u^*\alpha \ds & = \int_{(-1,1)\priv J_\delta}
{\alpha_{\gamma_{\eps_n}(t)} (\dot \gamma_{\eps_n}(t))} dt \\
& = -\int_{(-1,1)\priv J_\delta} {g_{\gamma(t)}}\left(\frac{\dot\gamma(t)}{{\|\dot\gamma(t)\|_g}},
\dot\gamma(t)+O(\eps_n) \right) + O(\eps_n)dt\\
& =-\int_{(1,-1)\priv J_\delta}\|\dot\gamma(t)\|_g + O(\eps_n)dt.
\end{align*}
Notice also that our $O(\eps_n)$ depends on $u$ and $\delta$, but is uniform in $t$, so
$$
\lim_{n\to \infty}\int_{(-1,1)\priv J_\delta+i\eps_n} u^*\alpha =-\ell_g\big(\gamma((-1,1)
\priv J_\delta)\big).
$$
Letting now $\delta$ go to $0$, the right hand side of this equality obviously converges
to $-\ell_g(\gamma)$ because $\gamma$ is smooth. On the other hand,
$$
\left|\int_{J_\delta+i\eps_n}u^*\alpha\right|
\leq  \int_{J_\delta}\left| \frac p{{\|p\|_g}}{\bigg|_{\gamma_{\eps_n}(t)}}(\dot\gamma_{\eps_n}(t))\right|dt
\leq \int_{J_\delta} {\| \dot \gamma_{\eps_n}(t)\|_g}\,dt
\leq C\delta,
$$
where $C$ is an upper bound for ${\|du\|_g}$. This shows equality \eqref{eq:limalpha}. \cqfd

\begin{lemma}\label{le:uniquenessjg}
If $\gamma$ is a geodesic of minimal length in its homology class $\beta$, then
$u_{\gamma,g}$ is the only $J_g$-holomorphic punctured disc with boundary on $L$ and
asymptotic to {$\tilde\gamma$} at 0 up to $S^1$-reparametrization. In
other terms,
$$
\cm(J_g,\beta)= \{[u_{\gamma,g}]\}.
$$
\end{lemma}

\noindent{\it Proof:} Let $u \in \widehat{\cm}(J_g,\beta)$. By lemma \ref{le:goodintersectionJg} we have
$u^{-1}(L)=\ci\cup \partial {D}$, where $\ci \subset {\mathrm{int}\,D\priv\{0\}}$ can only accumulate at 
critical points of $u$. On $\Om := {D\priv (u^{-1}(L) \cup\{0\})}$ we can write $u=(a,\tilde u)\in {(0, \infty)}
\times M$. Calling $\pi:TM\to \ker \alpha$ the projection along the Reeb vector field and putting
$V=\langle V,R\rangle R+\pi(V)$ for $V\in TM$, the equations for a $J_g$-holomorphic map can be
written
\begin{align*}
& d a = \chi( a)\langle d\tilde u\circ j, R\rangle, \\
& J_g(u) \pi (d\tilde u) = \pi(d\tilde u\circ j).
\end{align*}
As a result, if ${s + it}$ are local holomorphic coordinates near some point in $\Om$,
$$
u^*d\alpha\big(\tfrac{\partial}{\partial s},\tfrac{\partial}{\partial t}\big)
= \tilde u^*d\alpha\big(\tfrac{\partial}{\partial s},\tfrac{\partial}{\partial t}\big)
= d\alpha\left(\frac{\partial \tilde u}{\partial s},\frac{\partial \tilde u}{\partial t}\right)
= d\alpha\left(\pi\frac{\partial \tilde u}{\partial s},\pi \frac{\partial \tilde u}{\partial t}\right)
= d\alpha\left(\pi\frac{\partial \tilde u}{\partial s}, J_g\pi\frac{\partial \tilde u}{\partial s}\right)
\geqslant 0,
$$
and equality holds if and only if $\pi\frac{\partial \tilde u}{\partial s} =\pi\frac{\partial
\tilde u}{ \partial t}=0$.

Fix now $\eps\ll 1$ such that $\partial D(0,1-\eps)$ avoids $\ci$. 
Since $\ci$ can only accumulate at critical points of $u$ and $\crit(u)$ is finite,
it is possible to find a finite number of disjoint disks $D(\zeta_j, \eps_j) \Subset 
D(0,1-\eps)$ such that $\ci \cap D(0,1-\eps) \subset \bigcup_j D(\zeta_j, \eps_j)$.
Here the $\eps_j$ are chosen small enough and the number of disjoint disks $k$
depends on these choices.
Let also $\gamma_\eps^{j} := \partial D(\zeta_{j},\eps_{j})$, $\gamma_\eps := \partial D(0,\eps)$,
$\gamma_{1-\eps}:=\partial D(0,1-\eps)$, all those circles being oriented as the boundary of
$\partial \Om_\eps:= \partial\big(D(0, 1-\eps) \backslash (D(0,\eps) \cup \bigcup_j D(\zeta_j, \eps_j) ) \big)$. 
Then,
$$
{0 \leq \int_{\Om_\eps}u^*d\alpha}
= \int_{\partial \Om_\eps} u^*\alpha
= \sum_{j=1}^k \int_{\gamma_\eps^j}  u^*\alpha+\int_{\gamma_\eps} u^*\alpha
+\int_{\gamma_{1-\eps}}u^*\alpha.
$$
Since $u$ is asymptotic to ${\tilde\gamma}$ at $0$, a straightforward computation shows that
$$
\lim_{\eps\to 0}\int_{\gamma_\eps}u^*\alpha=\ell_g(\gamma).
$$
On the other hand, since $\alpha=\nf p{{\| p \|_g}}$ is a bounded
$1$-form on $T^*L\priv L$, its integral over the small loops $u\circ \gamma_\eps^j$
tend to $0$ when $\eps_j\to 0$ (recall that $u$ is smooth near the $\zeta_j$).
Since $u$ is smooth on int\,$D\priv \{0\}$ and $\ci$ does not approach $0$, 
we can  assume that
\[
\left|\sum_{j = 1}^k \int_{\gamma_\eps^{j}} { u^*\alpha }\right| < \eps
\]
by decreasing $\eps_j$ and increasing $k$ if necessary.
Finally, taking orientation into account, lemma~\ref{le:limalpha} shows that
$$
\lim_{\eps\to 0} \int_{\gamma_{1-\eps}}u^*\alpha=-\ell_g(u(\partial {D})).
$$
Putting all these estimates together {we get}
$$
0\leq \int_\Om u^*d\alpha=\ell_g(\gamma) - {\ell_g}(u(\partial {D})).
$$
However the reverse inequality holds as well, because the projection of $u$ to $L$
provides a singular chain between $\gamma$ and $u(\partial {D})$ and $\gamma$
is the geodesic of minimal length in its homology class. We therefore conclude
that $u^*d\alpha=\tilde u^*d\alpha=0$ on $\Om = {D}\priv(\ci \cup \{0\})$,
so $\pi\circ d\tilde u=0$ on $\Om$, which means that $du(z)\in \langle R,
\frac \partial{\partial r}\rangle$ $\forall z\in \Om$. Since $u$ is asymptotic
to ${\tilde\gamma}$ at $0$, we see that $\im u=\im u_{\gamma,g}$. But then
$u=u_{\gamma,g}$ modulo source reparametrization. \cqfd


\subsection{Transversality} \label{sec:transversality}
The object of this paragraph is to study the surjectivity of the linearized Cauchy-Riemann
operator at $u_{\gamma,g}$. We recall that
\begin{itemize}
\its $(L,g)$ determines $M=\{g=1\}, \alpha, R$,
\its $T_{(q,p)}(T^*L) = \left\langle \frac \partial{\partial r}, R(q,p) \right\rangle
\oplus \ker \alpha \cap F_{(q,p)} \oplus \ker \alpha \cap H_{(q,p)} \;\; \forall\,
q \in L, p \neq 0$,
\its $J_gR=-\chi(r)g_q(\pi(R),\cdot)=-\chi(r)\frac\partial{\partial r}$, where $\chi(r)=1$
near $r=0$ and $\chi(r)=r$ near $\infty$,
\its $J_gV=-g_q(\pi(V),\cdot)\in \ker \alpha\cap F_{(q,p)}$  for $V\in \ker \alpha\cap H_{(q,p)}$,
\its $\gamma$ is a minimizing geodesic in the class $\beta$ whose length $\ell_g(\gamma)$
will be denoted $\ell$, 
\its $u_{\gamma,g}(s,t) = (f(s), \tilde \gamma(t)) \in (0, \infty)\times M\subset T^*L$
and $f(s):=G^{-1}(s)$ with $G(r)=\int_0^r \frac{dr}{\chi(r)}$. In fact $f$ solves the
differential equation $f' = \chi\circ f$.
\end{itemize}

The functional analytic setup is as follows. Since this approach is fairly standard,
we refer the reader to the appendix (page \pageref{p:bkpd}) for the precise definitions
and recall the main points here. We define a Banach manifold of maps $\cb^{k,p,\delta}$ that contains $u_{\gamma,g}$.
We consider a Banach space bundle $\ce^{k-1,p,\delta} \to \cb^{k,p,\delta}$ whose fibers
are spaces of complex antilinear bundle maps. The non-linear Cauchy Riemann operator
$\bar \partial_{J_g}$ defines a smooth section of this bundle by
\[
\bar \partial_{J_g}(u) = du + J_g(u) \circ du \circ j
\]
and we have $\bar \partial_{J_g}(u_{\gamma,g}) = 0$. The linearization of $\bar \partial_{J_g}$
at $u_{\gamma,g}$ is given by
\fonction{ \mathbf{D}_{u_{\gamma,g}} }{ T_{u_{\gamma,g}}\cb^{k,p,\delta} }{ \ce^{k-1,p,
\delta}_{u_{\gamma,g}}}{\xi}{ \nabla \xi + J_g \circ \nabla \xi \circ j + (\nabla_{\xi}J_g)
\circ d u_{\gamma,g} \circ j,}
where $\nabla$ is any symmetric connection on $T^*L$. Using local coordinates in a neighborhood
of our geodesic our operator takes the form
\[
\mathbf{D}_{u_{\gamma,g}}(\xi) = \bar \partial_{J_g}\xi + (dJ_g(\xi))du_{\gamma,g} \circ j 
:= d\xi + J_g(u)\circ d\xi \circ j+(dJ_g(\xi))du_{\gamma,g} \circ j.
\]
In the appendix we show that this operator is Fredholm and has index $1$ if $\gamma$ is
minimal (see corollary \ref{cor:index}).
By elliptic regularity the kernel of $\mathbf{D}_{u_{\gamma,g}}$, which will be our only concern,
does not depend on $k,p$ and $\delta$ provided $kp>2$ and $\delta > 0$ small enough.

The idea is clear: the linearized operator $\mathbf{D}_{u_{\gamma,g}}$ is Fredholm and
has index $1$. Establishing surjectivity therefore amounts to proving that its kernel has dimension $1$.
Introducing appropriate coordinates we manage to compute the operator explicitly. Under
certain conditions on the metric we then compute its kernel and show that it has the right
dimension. In the next section we will show that the metrics that satisfy these conditions
are $\cc^0$-dense in the set of Riemannian metrics on $L$.

We first need to introduce good coordinates, called Fermi coordinates, in a \nbd of our
closed geodesic $\gamma$. Unfortunately, under no assumption on the parallel transport along
$\gamma$, these coordinates are multivalued (equivalently, they only give local coordinates
on some cover of $L$), so we need to introduce a new set of notation. Let $(\dot\gamma(0),
v_1,\dots, v_{n-1})$ be an orthonormal basis of $T_{\gamma(0)}L$, and let $(\dot \gamma(t),
V_1(t),\dots,V_{n-1}(t))$ be the orthonormal basis obtained by parallel transport of
$(\dot \gamma(0), v_1,\dots,v_{n-1})$ along $\gamma_{|[0,t]}$. We define
\fonction{\phi}{\R\times (-1,1)^{n-1}}{L}{(x_n,x')}{\ds \exp_{\gamma(x_n)}\left(\sum_{i=1}^{n-1}
x_iV_i(x_n)\right) \qquad (x'=(x_1,\dots,x_{n-1})).}
By restricting this map to $\R\times (-\delta,\delta)^{n-1}$ for $\delta > 0$ small enough
we get an immersion. This can alternatively be achieved on $\R\times (-1,1)^{n-1}$ by
rescaling the metric. Such a rescaling obviously does not affect the generality of our
argument, so we assume that $\phi$ itself is an immersion. When the parallel transport
$O$ along $\gamma_{|[0,\ell]}$ is the identity, $\phi$ induces an embedding
$S^1\times (-1,1)^{n-1}\hra L$, which provides coordinates near $\gamma$. In general
however, this immersion only provides multivalued coordinates (it satisfies
$\phi(x_n+\ell,x')=\phi(x_n,Ox')$). The map $\phi$ naturally lifts to a map
$$
\Phi:T^*(\R\times (-1,1)^{n-1})\looparrowright T^*L
$$
such that $\Phi(x_n+\ell,y_n,x',y')=\Phi(x_n,y_n,Ox',{}^tO^{-1}y')=\Phi(x_n,y_n,Ox',Oy')$.
Here $(y_n, y') \in \R^n$ are the coordinates in the fiber.
We define $\hat g,\hat J_g,\hat \alpha,\dots$ to be the pull-backs of their corresponding
objects in $T^*L$ by $\Phi$. {By construction, $\hat g_{ij}=\delta_{ij}+O(\|x'\|^2)$,
where $\|\cdot\|$ is any norm on $\R^{n-1}$ (equivalently, $\hat g_{ij}(x_n,0)=\delta_{ij}$ and
$\frac{\partial \hat g_{ij}}{\partial x_l}(x_n,0)=0$).}
Notice that by functoriality of the Liouville form, we get $\hat
\alpha=\alpha_{\hat g}$, so $\hat R=R_{\hat g}$, \dots, and finally $\hat J_g=J_{\hat g}$
(provided we use $\hat \chi=\Phi^*\chi$ in the definition of $J_{\hat g}$).
Also, $\gamma$ pulls-back to $\hat \gamma(t)=(t,0)\in \R\times (-1,1)^{n-1}$, $\tilde
\gamma$ to $\tilde{\hat\gamma}$, and $u_{\gamma,g}$ to the map
\fonction{\hat u_{\gamma,g}=u_{\hat \gamma,\hat g}}{[0, \infty) \times \R}{T^*\big(\R\times (-1,1)^{n-1}\big)
\simeq \R^2_{(x_n,y_n)}\times (-1,1)^{n-1}_{x'}\times \R^{n-1}_{y'}}{(s,t)}{(t,f(s),0,0)}
This map is $J_{\hat g}$-holomorphic for the standard holomorphic structure $j\partial_s
=\partial_t$
on $[0, \infty) \times \R$
(recall that $J_{\hat g}(x_n, y_n, 0,0)\frac\partial{\partial x_n}=-\chi(y_n)\frac \partial{\partial y_n}$).
Finally, $\xi$ lifts to a vector field $\hat \xi$ along $u_{\hat \gamma,\hat g}$
\fonction{\hat \xi}{[0,\infty)\times \R}{T_{\hat u_{\gamma, g}}\big(T^*(\R\times (-1,1)^{n-1})\big)
\simeq  \R^2\times \R^{2n}}{(s,t)}{(t,f(s),\hat z(s,t))}
(the tangent bundle to $T^*(\R\times (-1,1)^{n-1})$ is trivial). The vector space $\R^{2n}$
naturally splits into $\R^2\times \R^{2(n-1)}$ tangent to $T^*\R$ and to $T^*(-1,1)^{n-1}$,
respectively.  Both factors $\R^2$ and $\R^{2(n-1)}$ further split as $\R\times \R$ and $\R^{n-1}
\times \R^{n-1}$, where the second factors are tangents to the fibers of the cotangent bundles,
while the first factors are non-canonical "horizontal spaces"
{(not to be confused with the subbundle $H$)},
which are tangent to $L$ at the zero-section. We thus write
$$
\begin{array}{ll}
\hat z(s,t) & = (\hat z_n(s,t),\hat z'(s,t))\\
 & = (\hat a_n(s,t),\hat b_n(s,t),\hat a'(s,t),\hat b'(s,t))\in \R^2\times \R^{2(n-1)}\\
 & = (\hat a_n(s,t),\hat b_n(s,t),\hat a_1(s,t), \dots, \hat a_{n-1}(s,t), \hat b_1(s,t),\dots, ,\hat b_{n-1}(s,t)).
\end{array}
$$
Then, $\hat z(s,t)$ has the following properties (see the appendix \ref{app:func_an_setup}):
\begin{enumerate}
\item  $\hat z\in W^{k,p}_\loc( [0,\infty) \times\R,\R^{2n})$,
\item $\hat z(0,t)$ is tangent to the zero section ({\it i.e.} $\hat b_i(0,t)=0$ $\forall\,i$)
\item $\Vert \hat z'(s,t)\Vert\leq Ce^{-\delta s}$ for some constant $C$. In particular, $\hat a'$ and $\hat b'$ tend to $0$
uniformly when $s$ tends to $+\infty$,
\item $\hat a_n(s,\cdot)$ and $\hat b_n(s,\cdot)$ tend to constants uniformly when $s$ tends
to $+\infty$
\item Finally, $\hat z$ satisfies a certain pseudo-periodicity with regard to the parallel
transport map $O$. Recall that $\Phi(x_n+\ell,y_n,x',y')=\Phi(x_n,y_n,Ox',Oy')$, thus we have
$$
\hat z_n(s,t+\ell) = \hat z_n(s,t),\quad \hat z'(s,t+\ell) = O\hat z'(s,t) := (O\hat a'(s,t),
O\hat b'(s,t)).
$$
\end{enumerate}
We call $W$ the space of maps
$\hat z:[0,\infty)\times \R\to \R^{2n}$ that satisfy
the four properties listed above. For $\hat z\in W$ and
$\hat \xi=(t,f(s),\hat z(s,t))$, we then define
$$
\hat D\hat \xi:=\bar \partial_{\hat J_g}\hat \xi + (d \hat J_g(\hat \xi))d\hat u_{\gamma,g}\circ j.
$$
The naturality of the lift $\Phi$ of $\phi$ readily implies that for $\Phi_* \hat \xi = \xi$
$$
\Phi_* \hat D \hat \xi = \mathbf{D}_{u_{\gamma,g}}(\xi),
$$
so if $\xi$ lies in the kernel of $\mathbf{D}_{u_{\gamma,g}}$, $\hat D\hat \xi=0$ as well.
Thus, $\mathbf{D}_{u_{\gamma,g}}$ is surjective as soon as $\dim  \ker \hat D=1$.
We can get a more explicit expression for $\hat D$ by noticing that
\begin{itemize}
\its $\hat D\hat \xi$ is a $(0,1)$-operator, so it is determined by its
action on $\frac\partial{\partial s}$. We will therefore identify $\hat D\hat \xi$ with
$\hat D\hat\xi(\frac\partial{\partial s})$. Moreover,
$d u_{\hat \gamma,\hat g}\circ j \frac\partial{\partial s} =
\frac\partial{\partial x_n}$
\its $J_{\hat g}$ is a tensor field on $T^*(\R\times (-1,1)^{n-1})\simeq \R^2\times
(-1,1)^{n-1}\times \R^{n-1}$, whose tangent bundle is trivial. It can therefore be seen
as a map $J_{\hat g}:\R^2\times (-1,1)^{n-1}\times \R^{n-1}\to M_{2n}(\R)$. Via
this identification,  $dJ_{\hat g}(\hat \xi)
= dJ_{\hat g}(\hat z)$ and
$\bar \partial_{J_{\hat g}}\hat \xi\,\frac \partial{\partial s}= \frac{\partial \hat z}
{\partial s}+J_{\hat g} \frac{\partial \hat z}{\partial t}$.
\end{itemize}
Thus, putting $\hat J(s,t):=J_{\hat g}(u_{\hat \gamma,\hat g}(s,t))
= J_{\hat g}(t,f(s),0,0)$, we get
$$
\begin{array}{ll}
\hat D\hat \xi\,(\frac\partial{\partial s}) & =\frac{\partial \hat z}{\partial s}+\hat
J(s,t)\frac{\partial \hat z}{\partial t}
+dJ_{\hat g}(\hat z)\frac\partial{
\partial x_n}\\
& =\frac{\partial \hat z}{\partial s}+\hat J(s,t)\frac{\partial \hat z}{\partial t}
+d\big(J_{\hat g}\frac\partial{\partial x_n}\big)\hat z. \\
\end{array}
$$
Let us also finally simplify our notation. In the remainder of this section no further
reference to $\hat g,\hat J_g,\hat u_{\gamma,g},\hat \xi$ will be made. We only pay
attention to our operator
$\hat D$ that lives on the pull-back space. In order to keep the notation as light as
possible, we therefore suppress all hats $\hat {}$ from our letters, remembering that all
objects correspond to their pull-backs by $\Phi$. In other terms, in the rest of this
section, $\gamma(t)=(t,0)\in \R\times (-1,1)^{n-1}$, $u_{\gamma,g}=
(t,f(s),0,0)) \in T^*\R \times T^*(-1,1)^{n-1}$, $g=\Phi^*g$,
$J_g=J_{\hat g}$ \dots We only call $\hat L:=\R \times (-1,1)^n$ in order to keep the
reader aware that we work in our cover.

Summarizing our discussion we obtain the following.

\begin{prop}\label{prop:DtoDhat} Our initial linearized Cauchy-Riemann operator is
surjective as soon as the kernel of the operator
\fonction{D}{W}{W^{k-1,p}([0,\infty)\times \R,\R^{2n})}{z}{\ds \frac{\partial z}{\partial s}
+ J_g(s,t)\frac{\partial z}{\partial t} + d\left(J_g\frac\partial{\partial x_n}
\right) z}
has dimension $1$.
\end{prop}

We now wish to explicitly compute our operator $D$ in the coordinates $(x_n,y_n,x',y')$
that we have introduced in the \nbd of $\im u_{\gamma,g}$. Recall that $g$ is flat of
order $1$ near $\gamma$ in these coordinates ($g_{ij}=\delta_{ij}+O(\Vert x'\Vert^2)$).

\begin{lemma}\label{le:compJ}
On $T^*\hat L$, endowed with the coordinates $(x_n, y_n, x_1, \ldots, x_{n-1}, y_1,
\ldots, y_{n-1})$ defined above, we have at $(x_n,y_n,0,0)$, $y_n\geq 0$,
\begin{itemize}
\item  $\ds J_g\frac\partial{\partial x_i} = - \frac\partial{\partial y_i}$ $\;\forall\, i<n\;\;$
and $\;\;\ds J_g\frac\partial{\partial x_n}=-\chi(y_n)\frac\partial{\partial y_n}$,
\item $\ds \frac{\partial J_g}{\partial y_i}\left(\frac\partial{\partial x_n}\right)=
{\frac{1-\chi(y_n)}{y_n}\frac\partial{\partial y_i}}$ $\;\forall\, i < n \;\;$
and $\;\; \ds \frac{\partial J_g}{\partial y_n}\left(\frac\partial{\partial x_n}\right) =
-\chi'(y_n)\frac\partial{\partial y_n}$,
\item {$\ds \frac{\partial}{\partial x_i}\left(J_g\frac{\partial}{\partial x_n}\right)
=-\frac{y_n}2 \sum_{l\neq n} \frac{\partial^2g_{nn}}{\partial x_i\partial x_l}
\frac{\partial}{\partial x_l}$} $\;\forall\, i < n \;\;$ and $\;\; \ds \frac{\partial}{
\partial x_n}\left(J_g\frac{\partial}{\partial x_n}\right) = 0.$
\end{itemize}

\end{lemma}
\noindent{\it Proof:} Until now, we have not made a distinction between the metric
$g$ on the manifold and the induced metric on the cotangent bundle. In this proof alone,
we need to make a clear distinction, so the metric in the cotangent bundle will be
denoted $g^\sharp$. By definition of the Fermi coordinates the vector fields
$(\nf\partial{\partial x_i})$ are parallel along the geodesic $\gamma$, so the vectors
$(\nf\partial{\partial x_i})$ belong to the horizontal space at every point of the lift
of $\gamma$, which are precisely the points $(x_n,y_n,0,0)$. Since $\alpha =
\textnormal{sign}(y_n)dx_n$ at these points, we thus see that $\frac\partial{\partial x_i}
\in \ker \alpha\cap H$ for $i\neq n$, so
$$
J_g\frac\partial{\partial x_i}(x_n,y_n,0,0)=-\frac\partial{\partial y_i} \hspace{,5cm}
\forall\, i\neq n.
$$
Since $R$ generates the cogeodesic flow, it is also clear that $R(x_n,y_n,0,0)=\nf\partial{
\partial x_n}$, so
$$
J_g\frac\partial{\partial x_n}(x_n,y_n,0,0)=-\chi(y_n)\frac\partial{\partial y_n}.
$$
This proves the first assertion of the lemma and shows as well
$$
\frac\partial{\partial x_n}\left(J_g\frac\partial{\partial x_n}\right)(x_n,y_n,0,0)=0,
\quad\frac\partial{\partial y_n}\left(J_g\frac\partial{\partial x_n}\right)(x_n,y_n,0,0)
=-\chi'(y_n)\frac\partial{\partial y_n}.
$$
In order to compute the other derivatives, we need to compute $J_g\frac \partial{\partial x_n}
(x_n,y_n,x',y')$ at the first order, which requires to decompose $\nf \partial{\partial x_n}$
along $\langle R, \frac\partial{\partial r}\rangle\oplus  \ker \alpha \cap H
\oplus \ker \alpha \cap F$.\\\newline
{\bf Preliminary step: computation of $R$.} $R$ is colinear to the Hamiltonian vector field
associated to the function $H(x,y) := \|y\|^2_{g^\sharp_x} = \sum g_{ij}^\sharp(x)y_iy_j$,
because both generate the cogeodesic flow on the level sets of $H$, which coincide with the
hypersurfaces $\{r\}\times M$.
This shows that for $r > 0$, $R$ is colinear to
$$
\sum_l\Bigg( \Big(\sum_{i,j}\frac{\partial g_{ij}^\sharp}{\partial x_l}y_iy_j\Big)\frac\partial{
\partial y_l}-2\Big(\sum_jg_{lj}^\sharp y_j\Big)\frac{\partial}{\partial x_l} \Bigg).
$$
We specialize to two situations. When $x'=0$, we have $\frac {\partial g_{ij}^\sharp}{\partial
x_l}=0$ and $g_{lj}^\sharp =\delta_{lj}$, so for $y_n > 0$
$$
R(x_n,y_n,0,y')\propto R':=\frac\partial{\partial x_n}+\sum_{{l\neq n}}
\frac{y_l}{y_n}\frac\partial{\partial x_l}.
$$
Taking again into account that $g_{ij}(x_n,0)=\delta_{ij}$, we thus get
$$
\frac{\partial}{\partial r}(x_n,y_n,0,y')\propto \pi(R')^\sharp=g(\pi(R'),\cdot)=\frac\partial{
\partial y_n}+\sum_{l\neq n}\frac{y_l}{y_n}\frac\partial{\partial y_l}.
$$
When $y'=0$, we get
$$
R(x_n,y_n,x',0)\propto \sum_l \frac{\partial g_{nn}^\sharp }{\partial x_l}y_n^2\frac\partial{
\partial y_l} - {2}\sum_l g_{ln}^\sharp y_n\frac\partial{\partial x_l}
$$
Taking into account that $g_{ln}^\sharp (x_n,x')=\delta_{ln}+O(\|x'\|^2)$, we therefore get
$$
R(x_n,y_n,x',0)\propto R'=\frac\partial{\partial x_n}-{\frac{y_n}2}\sum_l \frac{\partial g_{nn}^\sharp }{
\partial x_l}\frac\partial{\partial y_l}+O(\|x'\|^2),
$$
and
$$
\frac{\partial}{\partial r}(x_n,y_n,x',0)\propto \pi(R')^\sharp=g\left(\frac{\partial}{
\partial x_n},\cdot\right)+O(\|x'\|^2)=\frac{\partial}{\partial y_n}+O(\|x'\|^2).
$$ \newline
{\bf Computation of $\frac{\partial}{\partial y_i}\big(J_g\frac{\partial}{\partial x_n}\big)$ for
$i < n$.}
Using the preliminary step we see that
$$
J_g\frac{\partial}{\partial x_n}(x_n,y_n,0,y')=J_g \left(R'-\sum_{l\neq n}\frac{y_l}{y_n}\frac{\partial}{
\partial x_l}\right).
$$
Taking into account that $y_l\in O(\|y'\|)$ and $J_g\frac{\partial}{\partial x_l} = -\frac{\partial}{
\partial y_l} $ for $l\neq n$ (using equation \eqref{eq:hor_lift_vectors} and the fact
that $x' = 0$), we thus
get
$$
J_g\frac{\partial}{\partial x_n}(x_n,y_n,0,y') = -\chi(\|y\|_{g^\sharp})\pi(R')^\sharp+\sum_{l\neq n}\frac{y_l}{
y_n}\frac{\partial}{\partial y_l} .
$$
Since $y_n>0$, $\chi(\|y\|_{g^\sharp})=\chi(y_n)+O(\|y'\|^2)$, so
$$
J_g\frac{\partial}{\partial x_n}(x_n,y_n,0,y')=-\chi(y_n)\frac{\partial}{\partial y_n}+\sum_{l
\neq n}\frac{y_l}{y_n}(1-\chi(y_n))\frac{\partial}{\partial y_l}+O(\Vert y'\Vert^2).
$$
Differentiating with respect to $y_l$, we get the announced formula for $\frac{\partial}{\partial
y_l}\big(J_g\frac{\partial}{\partial x_n}\big)$. \\ \newline
{\bf Computation of $\frac{\partial}{\partial x_i}\big(J_g\frac{\partial}{\partial x_n}\big)$ for
$i < n$.}
Using the preliminary step, we can write
$$
\frac{\partial}{\partial x_n}(x_n,y_n,x',0)=R' + {\frac{y_n}2}\sum_l \frac{\partial g_{nn}^\sharp }{
\partial x_l}\frac{\partial}{\partial y_l}+O(\|x'\|^2).
$$
Noticing that $\frac{\partial g_{nn}^\sharp }{\partial x_l} \in O(\|x'\|)$ for all $l$,
$J_g\frac{\partial}{\partial y_l}=\frac{\partial}{\partial x_l}+O(\|x'\|)$ for all $l\neq n$,
and $J_g\frac{\partial}{\partial y_n}=\frac 1{\chi(y_n)}\frac{\partial}{\partial x_n}+O(\|x'\|)$,
we conclude that
$$
\begin{array}{ll}
J_g\frac{\partial}{\partial x_n}(x_n,y_n,x',0)&=\ds-\chi(y_n)\pi(R')^\sharp+\frac{y_n}{{2}
\chi(y_n)}\frac{\partial g_{nn}^\sharp }{\partial x_n}\frac{\partial}{\partial x_n}+
\frac{y_n}{2}\sum_l \frac{\partial g_{nn}^\sharp }{\partial x_l}\frac{\partial}{\partial x_l}+O(\|x'\|^2), \\
&\ds=-\chi(y_n)\frac{\partial}{\partial y_n}+\frac{y_n}{{2}\chi(y_n)}\frac{\partial g_{nn}^\sharp }{
\partial x_n}\frac{\partial}{\partial x_n}+ \frac{y_n}{2}\sum_{l\neq n} \frac{\partial g_{nn}^\sharp }{
\partial x_l}\frac{\partial}{\partial x_l}+O(\|x'\|^2).
\end{array}
$$
Since $\sum_{j = 1}^n g_{nj}^\sharp (x_n, x') g_{jn} (x_n, x') = 1$ and each factor for $j\neq n$ is
$O(\|x'\|^2)$, we see that $g_{nn}^\sharp =\frac 1{g_{nn}} + O(\Vert x'\Vert^4)$. Using
$g_{nn}(x_n,x')=1+O(\|x'\|^2)$ we obtain $\frac{\partial
g_{nn}^\sharp}{\partial x_l} = -\frac 1{g_{nn}^2}\frac{\partial g_{nn}}{\partial x_l} + O(\|x'\|^3)
= - \frac{\partial g_{nn}}{\partial x_l}+O(\|x'\|^2)$. Differentiating with respect to $x_j$
we get
$$
\frac\partial{\partial x_j}\left(J_g\frac\partial{\partial x_n}\right)(x_n,y_n,0,0)
= -\frac{y_n}{2\chi(y_n)}\frac{\partial^2 g_{nn}}{\partial x_n\partial x_j}
\frac\partial{\partial x_n}
-\frac{y_n}2\sum_{l\neq n}\frac{\partial^2 g_{nn}}{\partial x_j\partial x_l}
\frac\partial{\partial x_l}.
$$
Notice finally that $\frac{\partial g_{nn}}{\partial x_j}(x_n,0)\equiv 0$, so
$\frac{\partial^2 g_{nn}}{\partial x_n\partial x_j}=0$. We thus get
$$
\frac\partial{\partial x_j}\left(J_g\frac\partial{\partial x_n}\right)(x_n,y_n,0,0)
=-\frac{y_n}2\sum_{l\neq n}\frac{\partial^2 g_{nn}}{\partial x_j \partial x_l}
\frac\partial{\partial x_l},
$$
which is the third point of the lemma. \cqfd

We are now in position to compute the kernel of our operator
$D:W\to W^{k-1,p}([0, \infty)\times \R,\R^{2n})$.
Recall that $z=\sum a_i\frac \partial{\partial x_i}+b_i\frac\partial {
\partial y_i}$ and that
$$
Dz = \ds \frac{\partial z}{\partial s} + J_g(s,t)\frac{\partial z}{\partial t}
+ d\left(J_g\frac\partial{\partial x_n}\right) z.
$$
Then,
\begin{align*}
\ds \frac{\partial z}{\partial s} + J_g(s,t)\frac{\partial z}{\partial t}
& = \ds \sum_{l} \frac{\partial a_l}{\partial s}\frac{\partial }{\partial x_l}
+\frac{\partial b_l}{\partial s}\frac{\partial }{\partial y_l}
+\sum_{l\neq n}-\frac{\partial a_l}{\partial t}\frac{\partial }{\partial y_l}
+\frac{\partial b_l}{\partial t}\frac{\partial }{\partial x_l} \\
& \phantom{==} \ds -\chi\circ f\frac{\partial a_n}{\partial t}\frac{\partial}{\partial y_n}
+\frac 1{\chi\circ f}\frac{\partial b_n}{\partial t}\frac{\partial}{\partial x_n} \\
&\ds = \sum_{l\neq n} \left(\frac{\partial a_l}{\partial s}+\frac{\partial b_l}{\partial t}\right)
\frac{\partial }{\partial x_l} + \sum_{l\neq n} \left(\frac{\partial b_l}{\partial s}
-\frac{\partial a_l}{\partial t}\right)\frac{\partial }{\partial y_l} \\
& \phantom{==} \ds+ \left(\frac{\partial a_n}{\partial s}+\frac 1{\chi\circ f}
\frac{\partial b_n}{\partial t} \right)\frac{\partial}{\partial x_n} +
\left(\frac{\partial b_n}{\partial s}- \chi\circ f\frac{\partial a_n}{\partial t} \right)
\frac{\partial}{\partial y_n}.
\end{align*}
On the other hand,
\begin{align*}
d\left(J_g\frac\partial{\partial x_n}\right) z
& = \ds \sum a_i\frac{\partial J_g}{\partial x_i} \left(\frac{\partial}{\partial x_n}\right)
+\sum b_i\frac{\partial J_g}{\partial y_i}\left(\frac{\partial}{\partial x_n}\right) \\
& = \ds -\sum_{l\neq n}  \left(\sum_{i\neq n}a_i(s,t){\frac{f(s)}2}\frac{\partial^2g_{nn}}{
\partial x_i\partial x_l}(t)\right) \frac\partial{\partial x_l} \\
& \phantom{==} + \frac{1-\chi\circ f(s)}{f(s)}\sum_{l\neq n}b_l(s,t) \frac\partial{\partial y_l}
- b_n(s,t) \chi'\circ f(s)\frac\partial {\partial y_n},
\end{align*}
(we put $\frac{\partial^2g_{nn}}{\partial x_i\partial x_j}(t) :=
\frac{\partial^2 g_{nn}}{\partial x_i
\partial x_j}(t,0)$).
Taking into account the boundary conditions and the pseudo-periodicity, we finally obtain
the following proposition.

\begin{prop}
$z=\sum a_i\frac{\partial }{\partial x_i}+b_i\frac{\partial }{\partial y_i}\in \ker D$ if
and only if
\begin{align}\label{eq:system1}
\begin{split}
\ds \frac{\partial a_l}{\partial s} + \frac{\partial b_l}{\partial t}
- \frac{f(s)}{2} \sum_{i\neq n}a_i\frac{\partial^2 g_{nn}}{\partial x_i\partial x_l}(t)
& = 0 \qquad \forall\,l\neq n \\
\ds \frac{\partial b_l}{\partial s}-\frac{\partial a_l}{\partial t}
+ \frac {1-\chi\circ f(s)}{f(s)} b_l & = 0 \qquad \forall\,l \neq n \\
\ds \frac{\partial b_n}{\partial s}-\chi\circ f(s)\frac{\partial a_n}{\partial t}
-\chi'\circ f(s)b_n & = 0 \\
\ds \frac{\partial a_n}{\partial s}+\frac 1{\chi \circ f(s)}\frac{\partial b_n}{\partial t} & = 0
\end{split}
\end{align}
and satisfies the boundary and pseudo-periodicity conditions
\begin{align*}
& b_j(0,t)\equiv 0\;\;\forall\,j ,\quad a_j(\infty,t)
= b_j(\infty,t)=0 \;\;\forall\,j < n, \\
& a_n(s,\cdot) \text{ and } b_n(s,\cdot) \text{ converge uniformly to constants }
\nu,\mu \text{ when }s\to \infty,\\
& a'(s,t+\ell) = Oa'(s,t),\quad b'(s,t+\ell) = Ob'(s,t),\\
& a_n(s,t+\ell) = a_n(s,t),\quad b_n(s,t+\ell) = b_n(s,t).
\end{align*}
\end{prop}

It will be convenient to notice that this computation did not really involve the precise
formula for our Fermi coordinates, but only depended on some of its properties, namely that
$\phi:\R\times(-1,1)^{n-1}\to L$ is a covering of a \nbd of $\gamma$ (so that we can
lift everything), and that {$g_{ij}=\delta_{ij}+O(\|x'\|^2)$} along $\hat\gamma=\R\times \{0\}$.
The pseudo-periodicity will soon turn out to be important. This justifies the following.
\begin{definition}\label{def:genfermi} A map $\phi:\R\times (-1,1)^{n-1}\looparrowright L$
provides {\bf generalized Fermi coordinates} near $\gamma$ if $\phi(t,0)=\gamma(t)$,
$(\phi^*g)_{ij}(x_n,x')=\delta_{ij}+O(\Vert x'\Vert^2)$, and if $\phi(x_n+\ell,x')=
\phi(x_n,Ox')$ for some matrix $O\in O_{n-1}^+(\R)$ (then $O$ must correspond to the
parallel transport along $\gamma_{|[0,\ell]}$).
\end{definition}

In general, we do not know how to solve the system of equations \eqref{eq:system1}.
We now impose some assumptions on the metric that allow us to explicitly compute the
solutions and check that the kernel is indeed $1$-dimensional.

\begin{prop}\label{prop:transversality} Let $g$ be a Riemannian metric on $L$ with
a closed geodesic $\gamma$. We assume that the metric has the following expansion in
some generalized Fermi coordinates:
$$
\hat g(x_n,x')=(1+{k}\Vert x'\Vert^2)\delta_{ij}+o(\Vert x'\Vert^2), \quad k>0.
$$
 Then the operator $\mathbf{D}_{u_{\gamma,g}}$ is surjective.
\end{prop}
\noindent{\it Proof:} Notice that the coefficients of our partial differential equations
depend only on the second derivatives of our metric, so they coincide for two metrics which
are tangent of order $2$.  We can therefore easily compute these
{coefficients under the assumptions of proposition \ref{prop:transversality} and we find
$$
\frac{\partial^2 g_{nn}}{\partial x_i\partial x_j}(t)=\left\{
\begin{array}{ll}
{2}k & \text{ if } j=i\neq n,\\
0 & \text{ else.}
\end{array}
\right.
$$
The solutions  $z=(a,b)$ of \eqref{eq:system1} thus satisfy
\begin{align}\label{eq:system2}
\begin{split}
\ds \frac{\partial a_l}{\partial s} + \frac{\partial b_l}{\partial t} - k f(s)a_l  & = 0 \qquad \forall\,l\neq n \\
\ds \frac{\partial b_l}{\partial s}-\frac{\partial a_l}{\partial t}
+ \frac {1-\chi\circ f(s)}{f(s)} b_l & = 0 \qquad \forall\,l \neq n \\
\ds \frac{\partial b_n}{\partial s}-\chi\circ f(s) \frac{\partial a_n}{\partial t}
-\chi'\circ f(s)b_n & = 0 \\
\ds \frac{\partial a_n}{\partial s}+\frac 1{\chi \circ f(s)}\frac{\partial b_n}{\partial t} & = 0
\end{split}
\end{align}
We need several changes  of variables. First we define $\tilde a_n := a_n$,
$\tilde b_n:=\frac{b_n}{\chi\circ f}$.
Let $\rho(s):=\exp\left({\ds \int_0^s}\frac {1 - \chi\circ f(s')}{f(s')}ds'\right):
[0, \infty) \to \R $ be a solution of the differential equation
$$
\rho'(s)=\rho(s)\frac {1- \chi\circ f(s)}{f(s)}, \quad \rho(0) = 1,
$$
and put $\tilde a_l:=\rho(s)a_l(s,t)$ and $\tilde b_l:=\rho(s)b_l(s,t)$ for $l = 1, \ldots, n-1$.

Since $f(s)$ tends to $\infty$ when $s$ goes to $\infty$ and $\chi(r)=r$ at infinity,
we see that $\rho$ and $\nf 1{\chi\circ f}$ tend to $0$ at $\infty$. As a result,
the $\tilde a_l,\tilde b_l,\tilde a_n$ satisfy the same boundary conditions at $s=0,\infty$
as $a_l,b_l,a_n$ for $l\neq n$, but now $\tilde b_n(0,t)\equiv 0$ {\it and} $\tilde b_n(
\infty,t)=0$. Putting also $g(s) := -\frac {\rho'}{\rho}(s) - kf(s)$,
a straightforward computation shows that these functions satisfy now
\begin{eqnarray}\label{eq:systilde}
\begin{array}{l}
\ds  \frac{\partial \tilde a_l}{\partial s}+\frac{\partial \tilde b_l}{\partial t}+g(s) \tilde a_l=0\\
\ds \frac{\partial \tilde b_l}{\partial s}-\frac{\partial \tilde a_l}{\partial t}=0
\end{array}
\hspace{,2cm}\forall\,l\neq n, \hspace{1cm}&
\begin{array}{l}
\ds \frac{\partial \tilde a_n}{\partial s}+\frac{\partial \tilde b_n}{\partial t}=0 \\
\ds \frac{\partial \tilde b_n}{\partial s}-\frac{\partial \tilde a_n}{\partial t}=0\\
\end{array}
\end{eqnarray}
Let us first focus on our system for $l\neq n$. The functions $\tilde b_l$ satisfy
\[
\Delta \tilde b_l+g(s)\frac{\partial \tilde b_l}{\partial s} = 0,
\]
together with the pseudo-periodicity and boundary conditions
\[
\tilde b_l(0)=0,\; \tilde b_l(s)\underset{s \to \infty}\lra 0,\qquad \tilde b'(s,t+\ell)=O\tilde b'(s,t).
\]

Since $\Delta h^2=2h\Delta h+2\Vert \nabla h\Vert^2$, we get $\Delta \tilde b_l^2+g(s)
\frac{\partial \tilde b_l^2}{\partial s}=2\Vert \nabla \tilde b_l\Vert^2\geq 0$.
Summing these $l$ equations we get
\begin{equation}\label{eq:b'}
\Delta \Vert \tilde b'\Vert^2+g(s)\frac{\partial \Vert \tilde b'\Vert^2}{\partial s} \geq 0.
\end{equation}
Define now $h_\eps(s,t):=\Vert \tilde b' \Vert^2 + \eps e^{- \alpha s-\frac\beta 2 t^2}$. An
immediate computation gives
$$
\Delta h_\eps+g(s)\frac{\partial h_\eps}{\partial s} \geq \eps(\alpha^2 -\beta+\beta^2t^2
- g(s)\alpha)e^{-\alpha s-\frac\beta 2t^2}.
$$
We recall that $g(s) = -\frac {\rho'}{\rho}(s) - kf(s) = -\frac{1 - \chi\circ f}{f}(s) - kf(s)$,
that $f(s)\underset{s\to \infty}\lra +\infty$, and that $\chi(r)=r$ when $r\gg 1$. We therefore
see that $-g$ is bounded from below on $[0,\infty)$,
so the right hand side of this last inequality is positive if $\alpha$ is large enough, which
we assume henceforth. The function $h_\eps$ is positive by definition, and does not have any
maximum on $[0,\infty)\times \R$ (because its Laplacian is positive at each critical point). Let
$(s_n,t_n)$ be a maximizing sequence for $h_\eps$, that is $h_\eps(s_n,t_n)\to \sup h_\eps$.
Obviously, $s_n$ remains bounded because $h_\eps$ converges uniformly to $0$ as $s$ goes to
$\infty$. We also claim that thanks to the quasi-peridodicity  ($\tilde b'(s,t+\ell) =
O\tilde b'(s,t)$), we can assume that $t_n\in[0,\ell]$.  Indeed, if Frac$(\frac {t_n}\ell)$
denotes the fractional part of $\frac{t_n}\ell$ (whose sign is the same as $t_n$'s), we have
$$
h_\eps(s_n,t_n) = \Vert \tilde b'(s_n,t_n)\Vert^2+e^{-\alpha s_n^2-\frac\beta 2t_n^2}
 = \Vert \tilde b'(s_n,\ell\text{Frac}(\tfrac{t_n}\ell))\Vert^2 + e^{-\alpha s_n^2-\frac\beta 2t_n^2}
\leq h_\eps(s_n,\ell\text{Frac}(\tfrac{t_n}\ell)).
$$
Since $(s_n,t_n)$ cannot accumulate at an interior point, we see that $s_n\to 0$, so
$\sup h_\eps=\eps$. This implies that $h_\eps\leq \eps$, so $\Vert \tilde b'\Vert^2\leq \eps$.
Since this holds for all $\eps>0$, we conclude that $\tilde b'\equiv 0$, so $b'=0$. Thus,
using the second equation of the system \eqref{eq:system2}, we see that $a'=a'(s)$, and
from the first equation, that
$$
\frac{d a'(s)}{ds} - k f(s)a'(s)=0.
$$
This last equation gives
$$
a'(s) = a'(0)e^{k\int_0^sf(u)du}.
$$
Letting $s\to \infty$, and taking into account that $f(s) > 0$ and $a'(\infty)=0$, we thus
get $a'(0)=0$, and therefore $a'=0$.

The equations for $(\tilde a_n,\tilde b_n)$ are the standard Cauchy-Riemann equations, so
$\tilde h(z):=\tilde a_n(z)+i\tilde b_n(z)$ is a holomorphic function. By the condition
$\tilde b_n(0,t)=0$, Schwarz reflection provides an extension $h:\C\to \C$, and in view of
the boundary conditions and periodicity, this extension is bounded. By Liouville's theorem,
this function $h$ is constant, so $\tilde b_n\equiv 0$ and $\tilde a_n$ is a constant.
We therefore get that $\ker D$ is one-dimensional, being parameterized by
the sole value of the constant $\tilde a_n=a_n$. By proposition \ref{prop:DtoDhat}, $D$ is surjective. 
\cqfd

\subsection{A good Riemannian metric} \label{sec:goodmetric}
The aim of this paragraph is to construct a metric close to the reference metric
$g$, which satisfies the assumptions of proposition \ref{prop:transversality}.
A metric $g'$ is said to be $\eps$-close to a metric $g$ if $\forall\, q\in L$
and $\forall\,u,v\in T_qL$, $| g_q(u,v)-g_q'(u,v) | \leq \eps \| u \|_g \| v \|_g$.
This defines a topology on the space of metrics which we call the $\cc^0$-topology.

\begin{prop}\label{prop:goodmetric}
Let $g$ be a Riemannian metric on a manifold $L$ and $\beta\in H_1(L)$.
 For
all $\eps>0$ there exists a Riemannian metric $g_{\eps,\beta}$ with the
following properties:
\begin{enumerate}
\item $g_{\eps,\beta}$ is $\eps$-close to $g$ in the $\cc^0$-topology,
\item $g_{\eps,\beta}$ has a unique minimizing geodesic $\gamma$ in the
class $\beta$. Moreover, for some  generalized Fermi coordinates
$\phi:\R\times (-1,1)^{n-1}\looparrowright L$ near $\gamma$ and some $k > 0$,
$\phi^*g_{\eps,\beta}=(1 + k\Vert x'\Vert^2)\delta_{ij}+o(\Vert x'\Vert^2)$.
In other terms, $\phi^*g_{\eps,\beta}$ is tangent at order $2$ to a metric
with constant scalar curvature $-k$.
\end{enumerate}
As a result of proposition \ref{prop:transversality} and lemma \ref{le:uniquenessjg},
we see that there exists a unique solution to $\cp(J_{g_{\eps,\beta}},\beta)$
at which the differential of the operator $\bar \partial_{J_{g_{\eps,\beta}}}$
is surjective. 
\end{prop}
\noindent {\it Proof:}
Since the Riemannian metrics with a unique minimizing geodesic in class
$\beta$ are $\cc^l$-generic (for any $l$), we can assume that this
uniqueness property holds for $g$. Then the minimizing geodesic $\gamma(\beta)$ 
is the $k$-cover of some primitive geodesic $\gamma'$ in a class $\beta'$ such that 
$k\beta'=\beta$ (we say that a geodesic is primitive if it is not a multiple-cover of 
another one). It is then immediate to check that $\gamma(\beta')=\gamma'$. In 
dimension $2$, a primitive geodesic that minimizes the length in a homology class 
is always simple (i.e. injective). In dimension higher than $2$,  the property of a 
primitive geodesic being simple is $\cc^l$-generic (in the metric), so a further 
perturbation of $g$ ensures that $\gamma'$ is simple. Summarizing this preparatory 
discussion, we can slightly perturb $g$ so as to ensure that $\gamma := \gamma(\beta) 
= \gamma(k\beta')$ for some $k\geq 1$ is the unique $\beta$-minimizing geodesic, 
while $\gamma':=\gamma(\beta')$ is a simple closed geodesic. We denote $\ell'$ its length
and let $\phi:\R\times(-1,1)^{n-1} \looparrowright L$ be Fermi coordinates near $\gamma'$.
Then,
$$
\phi^*g=\delta_{ij}+h_{ij}(x_n,x'), \qquad |h_{ij}(x_n,x')|\leq C(x_n)\Vert x'\Vert^2.
$$
The pseudo-periodicity of $\phi$ has some immediate consequences.
\begin{itemize}
\its $\phi_*\Vert x'\Vert$ is a well-defined function locally near {$\gamma'$} on $L$.
With slight abuse of notation, we therefore understand $\Vert x'\Vert$ as a function
defined on both $\R\times (-1,1)^{n-1}$ and $L$ (near {$\gamma'$}).
\its Similarly, $\phi_*\delta_{ij}$ is a well-defined tensor locally near {$\gamma'$} on $L$,
because the ambiguity in multi-valuedness of $\phi$ is given by an orthogonal matrix
$O$ in the fiber direction, whose derivative is again $O$, which precisely preserves
the metric $\delta_{ij}$. By construction of the Fermi coordinates, $g=\phi_*\delta_{ij}+O(\Vert x'\Vert^2)$.
\its $h_{ij}(x_n+{\ell'},x')=h_{ij}(x_n,Ox')$, so our bound on $h_{ij}$ is in fact uniform
in $x_n$: $|h_{ij}(x_n,x')|$ $\leq C\Vert x'\Vert^2$  for some constant $C$ large enough.
\end{itemize}
Let now $\rho:[0,\infty) \to [0,1]$ be a smooth non-increasing function that equals $1$ near $0$,
with $\supp \rho=[0,1]$, and define $\rho_\eps(t):=\rho(\nf t\eps)$. The previous
remarks show that
$$
g_\eps:=\rho_\eps(\Vert x'\Vert^2) (1+nC\Vert x'\Vert^2)\phi_*\delta_{ij}+(1-
\rho_\eps(\Vert x'\Vert^2))(1+\eps)g
$$
is a well-defined metric on $L$. We claim that it has the required properties
provided $\eps$ is chosen small enough. Indeed, since both $g$ and $(1+nC\Vert
x'\Vert^2)\phi_*\delta_{ij}$ are tangent to $\phi_*\delta_{ij}$, the difference
$\Vert g-g_\eps\Vert_{\cc^0}$ is of order $\eps$ for $\eps$ small enough, so point
(1) of proposition \ref{prop:goodmetric} holds. Next, we prove that $\gamma$
is the unique $\beta$-minimizing geodesic for $g_\eps$.
Notice that
\begin{align*}
\phi^*g_\eps-\phi^*g
& = \rho_\eps(\|x'\|^2)\left((1+nC\|x'\|^2)\delta_{ij}-(\delta_{ij}+h_{ij}(x_n,x'))\right)
 + (1-\rho_\eps(\|x'\|^2))\eps g\\
& \geq \rho_\eps(\|x'\|^2)(nC\|x'\|^2\delta_{ij}-h_{ij}(x_n,x')),
\end{align*}
and the estimate $|h_{ij}(x_n,x')|\leq C\|x'\|^2$ holds. Thus, $\phi^*g_\eps\geq \phi^*g$
for $\eps$ small enough, with equality if and only if $x'=0$.
Thus, $g_\eps\geq g$ in the \nbd of $\im\gamma'=\im \gamma$ endowed with our Fermi coordinates,
with equality exactly on {$\im \gamma$}. On the other hand, if $\eps$ is small enough,
$g_\eps=(1+\eps)g$ outside this neighbourhood. Thus we see that $g_\eps\geq g$ on $L$,
with equality exactly on {$\im \gamma$}. Then $\ell_{g_\eps}(\gamma)=\ell_g(\gamma)=\ell$,
while if {$\gamma_2$} is any other closed connected curve in the class $\beta$,
$\ell_{g_\eps}({\gamma_2})>\ell_g({\gamma_2})\geq \ell$. We conclude that $\gamma$ is
indeed the unique $g_\eps$-minimizing geodesic  in the class $\beta$.
Moreover, in the multi-valued coordinates $\phi$ near $\gamma$, we have $\phi^*g_\eps
=(1+nC\Vert x'\Vert^2)\delta_{ij}$.
Finally, notice that although $\phi$ does not provide Fermi
coordinates for $g_\eps$ near {$\gamma'$}, it does provide generalized Fermi coordinates
because $\phi(x_n+{\ell'},x')=\phi(x_n,Ox')$ and $\phi^*g_\eps$ is indeed tangent to
$\delta_{ij}$ at order $2$ (see definition \ref{def:genfermi}). Point (2)
therefore holds for $g_\eps$. \cqfd

\begin{rem}\label{rk:goodmetric} We have proved more than stated. In fact,
we see that we can even choose a minimizing geodesic $\gamma$ for $g$ and
construct a \emph{deformation} of $g$ that achieves the conditions of
proposition \ref{prop:goodmetric}, for which $\gamma$ remains a geodesic
for the parameter of the deformation. This point is not really important
in our argument. It simply allows for an easy and explicit computation of
the index of the operator $\mathbf{D}_{u_{\gamma,g}}$ (corollary \ref{cor:index}).
\end{rem}

\subsection{Somewhere injectivity}\label{sec:somewhereinj}
We recall the context first. $(L,g)$ is a Riemannian manifold, $J\in \cj_{\cyl,g}^\infty$
is an almost complex structure on $T^*L$, $\beta \in H_1(L)$ is a class with a unique
minimal representative, which has a primitive geodesic (i.e. not a multiple cover).
This last assumption is essential in this section.
Let $\cm(J,\beta)$ be the set of $J$-holomorphic
maps from the punctured disc to $T^*L$, asymptotic at $0$ to a lift of the minimal geodesic in the class
$\beta$ and with boundary on $L$ (see the precise definition on
p. \pageref{thm:existence}). Recall that a map $u: D\priv \{0\}\to T^*L$ is said
to be {\it somewhere injective in a region $K\subset T^*L$} if there exists a point
$z\in D\priv \{0\}$ with $u(z)\in \mathrm{int}\,K$ and $u^{-1}(u(z))=\{z\}$. We call
such a point an injectivity point.
The aim of this section is to show that all elements of the moduli space
$\cm(J,\beta)$ have a somewhere injective point in a region where we will be
free to vary our almost complex structure.

\begin{lemma}[Somewhere injectivity]\label{le:injpoints}
For any $0<r<R< \infty$ every element $u\in \cm(J,\beta)$ has an injectivity point
in $T^*_RL\priv T^*_rL$.
\end{lemma}
\noindent {\it Proof:} Let $u:(D\priv \{0\},\partial D)\to (T^*L,L)$ be a
$J$-holomorphic map asymptotic to the lift of the geodesic $\gamma(\beta)$, which is 
primitive by assumption. Notice that this lift is then automatically injective.
Let $0<\eps<r$ be a regular value of $\|u\|_g$, so that the subset $\{\,\|u\|_g = \eps\,\}
\subset D\priv\{0\}$ is a finite union of closed circles. Let $\Om$ be the connected
component of a \nbd of $0$ in $\{\,\|u\|_g > \eps\,\}$ and $\Om_1,\dots,
\Om_n$ the other components (there are finitely many). Let $v:=u_{|\Om}$ and
$v_i:=u_{|\Om_i}$. Note that since $u(\partial D)\subset L$, none of these
components touches $\partial D$.\\ \newline
\textbf{Somewhere injectivity of $v$.} $v$ is a proper $J$-holomorphic map
 $\Om\to T^*L\priv T^*_\eps L$. We define the following subsets of $\Omega$,
\begin{align*}
\ci(v)& := \{\, z \, | \, v^{-1}(v(z))=\{z\}\, \}, \\
\cc(v)& := \{\, z \, | \, dv(z)=0 \, \}, \\
\cd(v)& := \{\, z \, | \, \exists\,z'\neq z, v(z)=v(z') \text{ and } v(\op(z))\neq v(\op(z')) \, \}.
\end{align*}
Another description of $\cd(v)$ is that the image of the restriction
of $v$ to {\it any} two \nbds of $z$ and $z'$ never coincides (this set corresponds
to the points where distinct branches of $v$ intersect).
It is well-known that the Micaleff-White theorem implies that the sets $\cc(v)$ and
$\cd(v)$ are discrete (in fact $\cc(v)$ is at least finite in our situation)
\cite{miwh,mcsa2}. Moreover, we have the following standard fact.
\begin{claim}
$\ci(v)\priv \cc(v)$ is open and $\ci(v)\cup \cd(v)$ is closed.
\end{claim}
Indeed, since $\cc(v)$ is discrete, from a sequence of points $z_n$ in the
complement of $\ci(v)\priv \cc(v)$  that converges to a point $z\in \Om$, we
can either extract a constant subsequence in $\cc(v)$ or a subsequence in ${}^c\ci(v)$.
In the first case, $z\in \cc(v)\subset {}^c\big(\ci(v)\priv \cc(v)\big)$, so we
assume henceforth the latter and we replace $z_n$ by its subsequence, so
$z_n\notin \ci(v)$ $\forall n$. Thus, there exist points $z_n'\neq z_n$ in $\Om$
such that $v(z_n)=v(z_n')$. Since $v$ is proper, continuous, and $v(z_n)$ converges,
we can extract further so that $z_n'\to z'\in \Om$, with $v(z)=v(z')$.  If $z\neq z'$,
these points do not belong to $\ci(v)$, and if $z=z'$, $v$ is injective in no \nbd of
$z$, so $z\in \cc(v)$.

To see that $\ci(v)\cup \cd(v)$ is closed, it is enough to
prove that $\overline{\ci(v)}\subset \ci(v)\cup \cd(v)$ (because $\cd(v)$ is discrete).
Consider a sequence $z_n\in \ci(v)$ that converges to a point $z\in \Om\priv \ci(v)$.
Then, there exists $z'\neq z$ such that $v(z')=v(z)$, and since $z_n\in \ci(v)$,
$v(\op(z'))$ cannot coincide with $v(\op(z))$. {This finishes the proof of our claim.}

Since $\cc(v),\cd(v)$ are discrete their union is countable. So there exists a
smooth embedded path $\gamma\subset D$ emanating from the origin,
joining $\partial D$, that avoids $\cc(v)$ and $\cd(v)$ and such that
$\gamma \cap \Om$ is connected. It is then clear that
$\ci(v)\cap \gamma\cap \Om$ is both open and closed in $\gamma\cap \Om$.
Now since $v$ is asymptotic at $0$ to
the lift of a primitive geodesic, which is injective as we noticed already, a \nbd 
of $0$ in $\Om$ is contained in $\ci(v)$. We therefore conclude
that $\gamma\cap \Om\subset \ci(v)$ and by the intermediate value theorem
there exists an element $z\in \ci(v)$ with $v(z)=u(z)\in T^*_RL\priv T^*_rL$
(recall that  $\|v(z)\|_g$ goes to $\infty$ when $z$ tends to $0$ and goes
to $\eps<r$ when $z$ tends to $\partial \Om$). In fact, there exists even a
connected open sub-arc $\gamma' \subset \gamma\cap\Om\subset \ci(v)$, whose
image by $v$ lies in $T^*_R L \priv T^*_rL$. \\ \newline
\textbf{Somewhere injectivity of $u$.} We now conclude our proof by showing
that some point of $\gamma'$ is a point of injectivity for $u$ itself. We argue
by contradiction. Assume that this is not the case, then for every $z \in
\gamma'$ there exists a $j \in \{1, \ldots, n\}$ and a $\xi_j \in \Om_j$ such
that $v(z) = v_j(\xi_j)$. In particular for all $j$ the subsets $\{\,z \in \gamma'\,|\, \exists
\,\xi_j \in \Om_j: v(z) = v_j(\xi_j)\,\}$ are closed in $\gamma'$. Indeed,  if $z_n \in \gamma' \to z \in \gamma'$ is 
such that $\exists\, \xi_n \in \Om_j$ with $v(z_n) = v_j(\xi_n)$, then by properness of $v_j$, $\xi_n$ has a
subsequence converging to $\xi_* \in \Om_j$ and we have $v(z) = v_j(\xi_*)$.  
The union of these subsets covers $\gamma'$.
Since $\gamma'$ is open, this is only possible if each
subset is either empty or equals $\gamma'$.
Thus, one of the maps $v_i$ (say $v_1$) must satisfy $\im v_1\supset
v(\gamma')$. 
Let $\Om'\subset \Om$ be the connected component of $\gamma'$ in the set of points
$z\in \Om$ such that $v(z)=v_1(z')$ for some $z'\in \Om_1$. 
Then, by the same argument as above, $\Om'$ is closed in $\Om$ by properness of $v_1$.
Notice that since $\cc(v)$ is
finite, $\Om'\priv \cc(v)$ is also closed in $\Om\priv\cc (v)$. We claim that
$\Om' \priv \cc(v)$ is also open in $\Om \priv \cc(v)$. Indeed,
if $z\in \Om'\priv \cc(v)$,
there is by definition  a sequence of distinct points $z_n\in \Om'$ that converges
to $z$ (because $\Om'$ is the connected component of an open arc). 
Notice that the Micaleff-White theorem shows that the preimage of any point
$v(z_n)$ is finite, so we can assume that the $v(z_n)$ are distinct as well. By
definition of $\Om'$, there are distinct  points $z_n'\in \Om_1$ such that
$v_1(z_n')=v(z_n)\to v(z)$. Now since $z\in \Om$, we have $\|v(z)\|_g>\eps$, and since
$v_1$ is proper into $\{\|p\|_g > \eps\}$, the sequence $z_n'$ is compactly contained in
$\Om_1$. Thus, we can assume that $z_n'$ converges to $z'\in \Om_1$. Then, $v(z)=
v_1(z')$, and using
\cite[Lemma 2.4.3]{mcsa} ($z\notin \cc(v)$), we get that $z\in \textrm{int}\,
\Om' \priv \cc(v)$.
Thus $\Om'\priv \cc(v)$ is open and closed in $\Om\priv \cc(v)$, which is connected,
so $\Om'\priv \cc(v)=\Om\priv \cc(v)$. Since $\Om'$ is also closed in $\Om$, we get that $\Om'=\Om$.
But this is impossible, since $\| v_1 \|_g$ is bounded while $\| v \|_g$ is not. \cqfd


\subsection{Compactness results for punctured holomorphic disks} \label{sec:compactnessI}
We state compactness results for sequences of punctured holomorphic disks
in cotangent bundles. In this section we will use results from symplectic
field theory and various constructions therein, such as the notion of holomorphic
buildings and the splitting of symplectic manifolds along a contact
hypersurface (stretching the neck). We give details where
possible, but for the sake of the exposition we refer the reader to
\cite{boelhowize} and \cite{abbas} for the precise definitions and notions
in symplectic field theory. In order to make this section more readable
we adapt the notation to the setting of punctured disks in cotangent
bundles, even though the results may hold in more general settings.

Our setting is as follows. For a closed Riemannian manifold $(L, g)$
we denote by $M \subset T^*L$ the unit cotangent bundle and by $\cw_g$ the
open unit disk bundle, i.e. $\partial \cw_g = M$. Then $(M, \alpha :=
\lambda|_M)$ is a contact manifold and $(T^*L, d\lambda)$ has a
cylindrical end $E$ given by $(E := T^*L \backslash \cw_g, d\lambda)
\simeq ([1,\infty) \times M, d(r \alpha))$. As before we denote by
$L \subset T^*L$ the zero-section and note that the above identification
extends to $(T^*L\backslash L ,d\lambda)\simeq ((0, \infty) \times M,
d(r \alpha))$.

\subsubsection{Energy of holomorphic curves in cotangent bundles}\label{sec:energy}
Let $J \in \mathcal{J}^{\infty}_{\mathrm{cyl},g}$ be an almost complex structure
on $T^*L$ that is cylindrical on the whole end $E$. Let $(S,j)$ be a compact
Riemann surface (possibly with boundary) and let $Z \subset S$ be a finite subset
of interior punctures. For a $J$-holomorphic map $F:(S \backslash Z, j) \to
(T^*L, J)$ we write $F = (a, f): U \subset S \backslash Z \to [1, \infty) \times
M$ on the subset $U := F^{-1}(E)$. In $(T^*L, \omega := d\lambda)$ we define the
$\omega$-energy of $F$ as
\begin{equation}\label{def:sympl_area_curve}
\mathcal{E}_{\omega}(F) = \int\limits_{F^{-1}(\cw_g)} F^*\omega
+ \int\limits_{F^{-1}(E)} f^*{d\alpha}.
\end{equation}
Let
$\Lambda := \left\{\, \varphi \in C_{c}^{\infty}([1,\infty))
\,\big|\, \varphi \geq 0, \int_1^{\infty} \varphi(s) ds = 1 \,\right\}$. We
define the $\alpha$-energy to be
\[
\mathcal{E}_{\alpha}(F) = \sup\limits_{\varphi \in \Lambda}
\; \int\limits_{F^{-1}(E)} \varphi \circ a \, da \wedge f^*\alpha
\]
The total energy of $F$ is then defined as the sum
\[
\mathcal{E}(F) = \mathcal{E}_{\omega}(F) + \mathcal{E}_{\alpha}(F).
\]

In order to apply the SFT compactness theorem we must find a uniform upper
bound on the energy of a sequence of punctured holomorphic curves. Compared with
the lemmas of these kind present in \cite{boelhowize}, for instance proposition 6.3,
the following estimation allows us to consider a sequence of almost complex
structures that are all cylindrical, but not fixed at infinity.

\begin{lemma}\label{lem:energy_bound}
Let $J$ be an almost complex structure on $T^*L$ that is cylindrical on the
end $E$. Let $\gamma$ be a closed unit speed geodesic on $L$ and $\tilde \gamma$ its lift
to $M$. For every $J$-holomorphic curve $u : (D \backslash \{ 0 \},
\partial D) \to (T^*L, L)$ asymptotic to $\tilde\gamma$ at $0$ we have
\[
\mathcal{E}(u) \leq {3}\alpha(\tilde \gamma) = 3\ell_g(\gamma).
\]
\end{lemma}
\noindent{\it Proof:}
Let $u:(D\priv\{0\}, \partial D) \to (T^*L,L)$ be a $J$-curve asymptotic to
$\tilde\gamma$ at $0$. Recall that $T^*L = \cw_g\cup E$, where $E \simeq
([1,\infty) \times M,d(r\alpha))$ and $\cw_g \priv L\simeq ((0, 1)
\times M, d(r\alpha))$, that $J$ is $d\lambda$-compatible on $T^*L$ and
compatible with $\alpha$ on $E$.
These compatibilities imply that $d\alpha(\cdot,J\cdot)\geq 0$ and
$dr \wedge \alpha(\cdot, J\cdot)\geq 0$ on $E_{-\delta}:=[1 -\delta,\infty)
\times M$ for some $\delta>0$ small enough. We will exploit this
non-negativity below, under the wording that {\it $\alpha$ is tamed by $J$}.

On $u^{-1}(E_{-\delta})$ we put $u = (a,\tilde u)\in [1-\delta,\infty)\times M$.
The total energy of $u$ is then given by $\mathcal{E}(u)=\sup\{\,\mathcal{E}_\phi(u)
\, | \, \phi\in \Lambda \,\}$, where
$$
\mathcal{E}_\phi(u) = \int\limits_{u^{-1}(\cw_g)}u^*\om + \int\limits_{u^{-1}(E)}
\tilde u^*d\alpha + \phi\circ a \,da\wedge \tilde u^*\alpha.
$$
Let us fix $\phi\in \Lambda$ (thus $\phi: [1, \infty) \to [0, \infty)$ has
compact support and $\int_1^\infty \phi=1$) and
consider a smooth function $\tilde \phi: (0, \infty) \to [0, \infty)$ such
that $\tilde \phi(t) = 1$ for $t \in (0, 1-\delta]$, $\tilde \phi(t)\geq 1$
for $t \leq 1$, $\tilde \phi(t)\geq \phi(t)$ for $t > 1$,
$\tilde \phi = 0$ near $\infty$, and $\int_{0}^{\infty} \tilde \phi \leq 3$
(this is possible because $\int_{0}^1 1 dt = \int_{1}^{\infty}\phi(t)dt
= 1$).
Define then $h(t) := \int_{0}^t \tilde \phi(s)ds$ and $\om_h :=
d(h\alpha)$. Since $h(t) = t$ for $t < 1-\delta$, $\om_h$ coincides with
$\om=d\lambda$ on $(0, 1-\delta]\times M$ and in particular extends to
the zero section. We claim that
\begin{align}
 0 & \leq u^*\om_h, \label{eq:omhpositive}\\
 \mathcal{E}_\phi(u) & \leq \int u^*\om_h. \label{eq:controlenergy}
\end{align}
Let us first see how to conclude. By \eqref{eq:omhpositive} we have
$$
\int_{D\priv \{0\}}u^*\om_h = \underset{\eps\to 0}\lim\int_{D\priv D_\eps} u^*\om_h
 = \underset{\eps\to 0}\lim\int_{D\priv D_\eps} u^*d(h\alpha)
 = - \underset{\eps\to 0}\lim\int_{\partial D_\eps}h\circ a\,\tilde u^*\alpha.
$$
Since $h$ is constant near $\infty$, $h(\infty) \leq 3$ and $u$ is asymptotic
to $\tilde\gamma$ at $0$. 
Note that the orientation of $\partial D_{\eps}$ corresponds to $-\tilde \gamma:=\tilde \gamma(-t)$ in the limit as $\eps\to 0$ (see p.\pageref{eq:defv}).
Now using \eqref{eq:controlenergy} we get
$$
\mathcal{E}_\phi(u)\leq \int u^*\om_h=- h(\infty) \int_{-\tilde \gamma} \alpha
=h(\infty)\alpha(\tilde \gamma)\leq 3\ell_g(\gamma).
$$
Since this holds for all $\phi$, we have our desired bound on $\mathcal{E}(u)$.

It remains to establish our claim. For \eqref{eq:omhpositive}, notice that
$\om_h=d\lambda$ on $(0,1-\delta]\times M$, so $u^*\om_h\geq 0$ on
$u^{-1}\big(L\cup (0,1-\delta)\times M\big)$. On $[1-\delta, \infty)\times M$,
$\alpha$ being tamed by $J$, it is easy to see that $\tilde u^*d\alpha\geq 0$
and $da\wedge \tilde u^*\alpha\geq 0$. Thus, since $h'\geq 0$, we get
$$
u^*\om_h=u^*d(h\alpha)= h(a)\tilde u^*d\alpha+h'(a)da\wedge \tilde u^*\alpha\geq 0.
$$
We now prove \eqref{eq:controlenergy}.
\begin{align*}
\mathcal{E}_\phi(u)
&\ds = \int_{u^{-1}(\cw_g)}u^*\om +
\int_{u^{-1}(E)} (\tilde u^*d\alpha+\phi\circ a \ds \, da\wedge\tilde u^*\alpha) \\
&\ds = \int_{ u^{-1}(\cw_g \priv E_{-\delta})} u^*\om + \int_{u^{-1}(\cw_g \cap E_{-\delta})} u^*d(r\alpha) +
\int_{ u^{-1}(E)}\big(\tilde u^*d\alpha+\phi\circ a \ds \,da\wedge\tilde u^*\alpha\big) \\
&\ds = \int_{ u^{-1}(\cw_g \priv E_{-\delta})} u^*\om_h +
\int_{u^{-1}(\cw_g \cap E_{-\delta})} \big(a\,\tilde u^*d\alpha + da \wedge \tilde u^*\alpha\big) \\
&\quad + \int_{u^{-1}(E)}\big(\tilde u^*d\alpha+\phi\circ a \, da\wedge\tilde u^*\alpha\big)
\end{align*}
As we already noticed, since $u$ is $J$-holomorphic for a $J$ which tames $\alpha$
on $E_{-\delta}$, we have $\tilde u^*d\alpha\geq 0$ and $da\wedge \tilde u^*\alpha\geq 0$,
so the inequalities $h(t)\geq t$ and $h'(t)\geq 1$  for $0< t <1$, $h(t)\geq 1$ for $t \geq 1$,
$h'(t)\geq \phi$ on $t \geq 1$ imply that
\begin{align*}
\mathcal{E}_\phi(u) & \ds \leq \int_{u^{-1}(\cw_g \priv E_{-\delta})} u^*\om_h
+ \int_{u^{-1}(\cw_g \cap E_{-\delta})} \big(h(a)\tilde u^*d\alpha+h'(a)\, da\wedge \tilde u^*\alpha\big) \\
& \quad + \int_{u^{-1}(E)}\big( h(a)\tilde u^*d\alpha+h'(a) \, da\wedge\tilde u^*\alpha\big)\\
& \ds \leq \int_{u^{-1}(\cw_g \priv E_{-\delta})} u^*\om_h
+ \int_{u^{-1}(E_{-\delta})}u^*d(h\alpha) \;\;=\;\; \int u^*\om_h.
\end{align*}
This concludes the proof of \eqref{eq:controlenergy} and thus of our lemma. \cqfd

If $J$ is only compatible with $\alpha$ on $[K,\infty)\times M$ for a constant
$K > 1$, a rescaling provides the estimate $\ce(u)\leq 3K\ell_g(\gamma)$. Thus lemma
\ref{lem:energy_bound} implies the following statement.
\begin{corollary}\label{cor:energy_bound}
Let $J_n$ be a sequence of almost complex structures in $\cj_{\cyl,g}^\infty$
which are compatible with $\alpha$ outside a fixed compact set. Then there
exists a  constant $C > 0$ that depends only on the fixed compact set, such that every $J_n$-holomorphic curve
$u_n: (D\priv\{0\},\partial D) \to (T^*L,L)$ asymptotic to $\tilde\gamma$
at $0$ satisfies the energy bound $\ce(u_n)\leq C$.
\end{corollary}


\subsubsection{Compactness for punctured holomorphic disks I}\label{sec:compactness1}

\begin{theorem}\label{thm:sft_compactness}
Let $(L, g)$ be a closed Riemannian manifold which has a unique minimizing geodesic
$\gamma$ in the class $\beta \in H_1(L;\Z)$. Let $J_n \in J_{\cyl,g}^\infty$ be a sequence of
almost complex structures on $T^*L$ that are fixed outside a compact subset and
converge to an almost complex structure $J \in J_{\cyl,g}^\infty$ in the
$C^{\infty}$-topology. Assume $u_n: (D \backslash \{ 0 \}, \partial D) \to
(T^*L, L)$ is a sequence of $J_n$-holomorphic curves asymptotic to
$\tilde\gamma$ at 0. Then, after restricting to a subsequence, $u_n$ converges in
$C^{\infty}_{\mathrm{loc}}$ to a $J$-holomorphic map $u: (D \backslash \{ 0 \},
\partial D) \to (T^*L, L)$ asymptotic to $\tilde\gamma$ at $0$.
\end{theorem}

\noindent{\it Proof:}
Since all $u_n$ are asymptotic to $\tilde\gamma$ at $0$, by corollary
\ref{cor:energy_bound} the energy $\mathcal{E}(u_n)$ is uniformly bounded.
We now apply the SFT compactness theorem for punctured holomorphic curves
with boundary in symplectic cobordisms \cite{boelhowize,abbas}. It implies
that there exists a $N \in \N$ such that our sequence $u_n: (D \backslash \{0\},
\partial D) \to (T^*L, L)$ converges to a stable holomorphic building of
height $N$. This is given by the following data in our situation.

\begin{enumerate}
\item[(i)] $v_0: (S_0 \backslash Z_0, \partial S_0, j_0) \to (T^*L, L, J)$ is
a proper $J$-holomorphic map from a compact Riemann surface
with boundary $S_0$ with a finite set of punctures
$Z_0 \subset S_0 \backslash \partial S_0$
to the almost complex manifold $(T^*L, J)$ such that $v_0(\partial S_0) \subset L$
and $v_0$ has finite energy. Note that $S_0$ may have multiple components.
\item[(ii)] For $k = 1, \ldots, N$ we have holomorphic maps $v_k:
(S_k \backslash Z_k, j_k) \longrightarrow (\R \times M, \widehat{J})$
from closed Riemann surfaces $S_k$ with a finite set of punctures $Z_k \subset
S_k$ to the symplectization $\R \times M$. Here the almost complex structure
$\widehat{J}$ is $\R$-invariant and equals $J$ where $J$ is translation invariant in
the end $[1, \infty) \times M$. Furthermore we have decoration maps
$\Phi_k$ for $k = 1, \ldots, N$ that glue all of the negative punctures of
$Z_k$ to the positive punctures of $Z_{k-1}$. More precisely, let
$\overline{S}_k$ denote the oriented blow-up of $S_k$ at the points
$z \in Z_k$. The conformal action $j_k$ defines a circle action at
every boundary circle of $\overline{S}_k$ and a decoration $\Phi_k$
is a choice of map from the negative boundary circles of
$\overline{S}_k$ to the positive boundary circles of
$\overline{S}_{k-1}$ {that anti-commutes with the circle action
($\Phi_k(e^{it}\xi)=e^{-it}\Phi_k(\xi)$).}
\item[(iii)] Denote by $\overline{S} := \overline{S}_0 \cup_{\Phi_1}
\overline{S}_1 \cup \ldots \cup_{\Phi_N} \overline{S}_N$
the piecewise smooth surface obtained by gluing together all blow ups
$\overline{S}_k$ of $S_k$ at their punctures via the decoration maps
$\Phi_k$. Then $\overline{S}$ has no nodal points and is homeomorphic
to $\widehat{D_0} := \R_{\geq 0} \times S^1 \cup \{ \infty \} \times S^1$,
which is homeomorphic to the oriented blow-up of $D$ at $0$.
This holds because $T^*L$ and $\R\times M$ are exact symplectic manifolds,
so there can be neither bubbling of spheres, nor of discs with boundary on $L$
(in a more general context with possible bubbling, this point holds only if
we add nodal points to our discussion). Furthermore there exists a
diffeomorphism $G: T^*L = \overline{\cw}_g \cup_M E \to \cw_g$ such that
the compositions $G \circ v_0$ and $\pi \circ v_k$ for all $k$ fit
together into a continuous map $\overline{S} \to \overline{\cw}_g$ that
maps $\{\infty\} \times S^1$ to $\tilde\gamma$, where $\pi: \R \times M \to M$
is the projection.
\end{enumerate}
\begin{figure}[h!]
\begin{center}
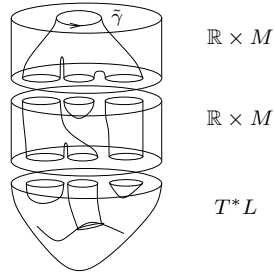
\end{center}
\caption{A holomorphic building appearing as a limit of $u_n$}
\label{fig:sft}
\end{figure}

We now analyze the limit holomorphic building. We start by considering
the top floor of the building, i.e. the holomorphic curve given by
$v_N: (S_N \backslash Z_N, j_N) \longrightarrow (\R \times M, \widehat{J})$.
By point (iii) the piecewise smooth surface $\overline{S}$ is
homeomorphic to $\widehat{D_0}$, hence only one component
of $S_N$ has a positive puncture in $Z_N$. This implies that there are no
other components of $S_N$, since the maximum principle asserts that
every punctured holomorphic curve in $\mathbb{R}\times M$ must have at
least one positive puncture.
The genus of $\overline{S}$ is 0, thus we have $S_N = \mathbb{C}P^1$,
$Z_N = \{ z_0, z_1, \ldots, z_l \}$ is a finite set of punctures and
$v_N: (\mathbb{C}P^1 \backslash \{ z_0, z_1, \ldots, z_l \}, j_N) \to
(\mathbb{R}\times M, \widehat{J})$ is
a $\widehat{J}$-holomorphic curve asymptotic to the Reeb orbit
$\tilde\gamma$ at the positive puncture, say
$z_0$, and asymptotic to Reeb orbits ${\tilde \gamma_1, \ldots,
\tilde \gamma_l}$ at the negative punctures.

Considering now the map given by the composition of
$\R\times M \twoheadrightarrow M \subset T^*L$
with the natural projection $\pi:T^*L\to L$, the image of
$v_N(S_N\priv Z_N)$ provides
a cobordism between $\gamma = \pi(\tilde\gamma)$ and $\cup_{i=1}^l\gamma_i :=
\cup_{i=1}^l \pi(\tilde \gamma_i)$, so $\sum [\gamma_i] =[\gamma]=\beta$. On the
other hand, since $v_N$ is $\widehat J$-holomorphic, we have
$$
0\leq \ca(v_N) := \int_{S_N\priv Z_N}v_N^*d\alpha 
= \int_{\partial (S_N\priv Z_N)}v_N^*\alpha
= \int_{{\tilde\gamma}}\alpha-\sum_{i=1}^l \int_{{\tilde\gamma_i}} \alpha
= \ell_g(\gamma) - \sum_{i = 1}^l \ell_g(\gamma_i).
$$
Notice now that, $\overline S$ being homeomorphic to a cylinder, the
building provides a capping of all but one of the $\tilde\gamma_i$.
Projecting these cappings to $L$, we deduce that only one of the $\gamma_i$
is homologically non-trivial. Thus, one of the $\gamma_i$ represents the
class $\beta$, while all the others vanish in homology. Since $\gamma$ is
the unique minimal geodesic in class $\beta$, we infer that
$\{\gamma_1,\dots,\gamma_l\}=\{\gamma\}$, that $\ca(v_N)=0$, so that
$v_N$ is a trivial cylinder above the Reeb orbit $\tilde \gamma$. Iterating
this argument for the floors below, we see that the stable holomorphic
building consists only of the holomorphic map $v_0:S_0\priv Z_0\to T^*L$
with a unique puncture ($\# Z_0=1$) at which $v_0$ is asymptotic to
$\tilde\gamma$. By point (iii) above, $S_0$ is a disc, so
$$
v_0: (D \backslash \{ 0 \}, \partial D) \to (T^*L, L)
$$
is $J$-holomorphic and asymptotic to $\tilde\gamma$ at $0$.\cqfd


\subsubsection{The splitting construction}\label{sec:framesplit}

In the following we will use a well-known splitting construction, also
called stretching the neck in the literature \cite{boelhowize}. In this
section we review this construction and explain our viewpoint on the
different objects introduced by the theory. We insist on our definition
of the completion of a symplectic cobordism, which might slightly
differ from the usual one.
\paragraph{Symplectic cobordisms.}
A symplectic cobordism is a compact symplectic manifold $(X,\om)$ with
boundaries $M^\pm$ of positive (resp. negative) contact type:
in a \nbd of $M^\pm$, $\om=d\lambda_\pm$ is exact, and the corresponding
Liouville vector field, defined by $\om(Y_\pm,\cdot)=\lambda_\pm$, points
outside $X$ for positive boundaries and inside $X$ for negative ones. Then $\alpha_\pm :=
\lambda_{\pm|M^\pm}$ is a contact form on $M^\pm$. One of the boundaries may be empty but not both.
It is called an {\it exact symplectic cobordism} if $\om$ has a global
primitive $\lambda$ on $X$ and $\lambda_\pm=\lambda$. This should not
be confused with the case where $(X,\om)$ is an exact symplectic manifold
and $\lambda_{\pm} \neq \lambda$.
We illustrate these points using the main examples that will appear
in the sequel. Let $(L,g)$ be a closed Riemannian manifold.
\begin{itemize}
\its $(\overline{\cw(L,g,r)},d\lambda)$ is an exact symplectic cobordism with
$M^-=\emptyset$, $M^+=\partial \cw(L,g,r)$.
\its For $r > r'$, $\overline{\cw(L,g,r)}\priv \cw(L,g,r')$ is an exact symplectic
cobordism with $M^+=\partial \cw(L,g,r)$ and $M^-=\partial \cw(L,g,r')$.
\its If $L'\subset T^*L$ is a Lagrangian submanifold and $\cw'\subset
\cw(L,g,r)$ is a Weinstein \nbd of $L'$, $\overline{\cw(L,g,r)}\priv \cw'$ is
exact as a symplectic manifold since it lies in $T^*L$, but it is an
exact symplectic cobordism only if $L'$ is an exact Lagrangian submanifold
of $T^*L$.
\end{itemize}
\paragraph{Cylindrical ends and completions.}
In a \nbd of a  boundary component of a symplectic cobordism (say positive),
the transverse Liouville vector field provides a coordinate $r$ and a symplectic
identification between this \nbd and $(M^+\times (1-\eps,1],d(r\alpha_+))$.
One can then glue {\it a cylindrical end} $M^+\times [1,\infty)$ to $X$
and obtain a {\it completion} of $X$ along $M^+$ (glue $M^-\times (0,1]$ for
a negative component). From a set theoretic point of view this completion
is unambiguous: $\wdt X := X \sqcup M^+\times (1,\infty)$. We now want to
give details on the structures involved. First we describe the completion
from the differential viewpoint more precisely. Let $f^+:(1-\eps,1)\to
(1-\eps,\infty)$ be a diffeomorphism that coincides with the identity near
$1-\eps$ (respectively, $f^-:(1,1+\eps)\to (0,1+\eps)$ is the identity near $1+\eps$). 
We get a bijection $\Phi:X\priv M^+\to \wdt X$ given by
$\Phi=\id$ outside $M^+\times(1-\eps,1)$ (or $M^-\times (1,1+\eps)$)
and defined in this region by $\Phi(x,r)=(x,f^+(r))$. Since $f^+$ is a
diffeomorphism, we get a natural structure of smooth manifold on $\wdt X$
which makes $\Phi$ a diffeomorphism. Then $\wdt X$ can be endowed with
the symplectic form $\wdt \om := \Phi_*\om$, {\it which we do} (the symplectic forms which
arise this way are all isomorphic).
We insist on this last point: a perhaps more common way to define a symplectic
form on the completion is to take $d(r\alpha)$ on the ends,
which enlarges the symplectic structure.
For instance the completion  of $(\overline{\cw(L,g,1)}, d\lambda)$ along its boundary in this way
is $(T^*L,d\lambda)$ and has infinite volume. 
We choose a different path here.
For us in this paper, the
completion is endowed with a symplectic structure which makes $\wdt X$
symplectomorphic to $X\priv M^+$. On the other hand, on $\wdt X$ we now
have dilations on the end $M^+\times (1,\infty)$ (contractions if we
complete a negative end), and we say that an almost complex structure
on $\wdt X$ is cylindrical at infinity if it is compatible with $\wdt\om$
on $\wdt X$ and compatible with $\alpha_+$ at infinity, 
in a similar way as already defined for the cotangent bundle: outside a compact set
of $\wdt X$
it sends $r\frac\partial {\partial r}$ to the Reeb vector field of
$\alpha_+$ and is invariant by the dilation on the positive end (by the
contractions on the negative end, respectively). In other terms, the completion $\wdt X$ coincides with
$X\priv M^+$ from a symplectic point of view, but it has very specific
almost complex structures, which will usually be denoted with a tilde.
This completion can be made along several boundaries at a time, and
by {\it the completion} of $X$ we mean its completion along all boundary
components. We say that an almost complex structure on $X\priv M^+$ is cylindrical
at infinity if it coincides with the pull-back of a cylindrical almost
complex structure on the completion $\wdt X$ of $X$. Similarly, a
holomorphic curve in $X\priv M^+$ is said to be asymptotic to some Reeb orbit
of a boundary component if it is when pushed forward to the completion
$\wdt X$. The following example will be useful in the sequel.
\begin{example}\label{ex:switch}
Since the completion of $\overline{\cw(L,g,r)}$ is $T^*L$ from the almost complex point of view,
the statement of theorem \ref{thm:existence}
for almost complex structures $J \in \cj^{\infty}_{cyl, g}$ is equivalent to the existence
of punctured holomorphic discs in $\cw(L,g,r)$
with boundary on $L$ and asymptotic to the lift of some minimal geodesic
in class $\beta$ to $\partial \cw(L,g,r)$ for
almost complex structures on $\cw(L,g,r)$ cylindrical at infinity.
We will freely switch from
one setting to the other in the following.
\end{example}

\paragraph{Stretching the neck.} Let now $X$ be a symplectic cobordism
and $M\subset \mathrm{int}\, X$ be a (separating) closed contact type hypersurface. As
previously, a collar \nbd of $M$ in $\mathrm{int}\, X$ is symplectomorphic to $(M\times
(1-\eps,1+\eps),d(r\alpha))$. We start with an almost complex structure
$J$ that is compatible with $\om$, cylindrical at infinity on $\mathrm{int}\, X$ and
compatible with $\alpha$ in our collar \nbd. For small $\eps > 0$ we
define a collection of manifolds $X_\eps$ by replacing our collar
\nbd in $\mathrm{int}\, X$ with a very large collar $M\times (\eps,1/\eps)$, by a procedure
analogous to the one described in the previous paragraph.
We extend $J$ to an almost complex structure $J_{\eps}$ on
$X_\eps$ by choosing $J_{\eps}$ to be dilation invariant in the
collar. There is a diffeomorphism between $X_\eps$ and $\mathrm{int}\,X$ 
and $X_\eps$ can be endowed with the pull-back symplectic form. 
Again, $X_\eps$ coincides with $\mathrm{int}\,X$ from the symplectic point of view, but
the almost complex structures compatible with $\alpha$ in this larger
collar are different. In our context, neck stretching consists in
letting the parameter $\eps$ go to zero and studying the behaviour
of sequences of $J_{\eps}$-holomorphic curves in $X_\eps$. The 
compactness
theorem of \cite{boelhowize} related to this splitting can be summarized in an imprecise way as follows:
\begin{thm*}[Bourgeois et al.]
Let $J_n:=J_{\eps_n}$ be a sequence of almost complex structures on a symplectic
cobordism, cylindrical and fixed at infinity, that stretch the neck
of a closed contact type hypersurface $M$. Assume that $M$ splits $X$ into $X_+\cup X_-$,
with $\partial X_+=M^+\cup M$ and $\partial X_-=M^-\cup M$. Let
$u_n:\Sigma\to X$ be a sequence of $J_n$-holomorphic curves
which satisfy some energy bound.
Then, after extraction, $u_n$
converges to a holomorphic building with finite energy: a collection of punctured
holomorphic curves in $M^+\times (0,+\infty)$, $\wdt X_+$, $M\times (0,\infty)$,
$\wdt X_-$, $M^-\times (0,\infty)$, asymptotic at their punctures to Reeb orbits and 
organized in levels (see the more
detailed descriptions in our specific situations below).
\end{thm*}

\begin{rems}\label{rk:energy}
\begin{itemize}
\its The statement above makes sense, and still holds, when $M=\emptyset$.
Then, the sequence $J_n$ does not stretch any neck, and the
components of the limit building take values in $M^+\times (0,+\infty)$,
$\wdt X$ and $M^-\times (0,+\infty)$.  This is exactly the compactness 
statement of theorem \ref{thm:sft_compactness}.\vspace{-,2cm}
\its We do not wish to discuss here the notion of energy of a curve or 
a holomorphic building in a cobordism. What we need to know for the 
rest of the paper is as follows and can also be found in \cite{boelhowize}:
in $T^*L$ this notion coincides with the energy that we defined 
in \S \ref{sec:energy}. The buildings that arise as a limit of curves with 
finite energy have finite energy. Then any component of such a building 
has itself finite energy. And finally, punctured holomorphic curves with finite energy 
are asymptotic in a $\cc^l$-sense to trivial cylinders over closed Reeb orbits.
\end{itemize}
\end{rems}

\paragraph{Symplectic area of a building in a split manifold.}
In the setting defined above, there is a natural concept of
symplectic area for a holomorphic building $B$ with finite energy 
in a split manifold, denoted henceforth $\ca(B)$,
which is very close to the notion of symplectic energy \cite{boelhowize}.
This is simply the sum of the symplectic areas of the components in
the different pieces of the split manifolds, where
\begin{itemize}
\its the symplectic area of a component in a completion ($\wdt X_+,
\wdt X_-$ in the previous statement) is computed with the symplectic
form. We recall that int$\,X$ and $\wdt X$ are symplectomorphic.
This quantity is a positive number.
\its the symplectic area of a component in a cylindrical piece
$M\times (0,\infty)$ is computed by integration of $\pi^*d\alpha$ 
(where $\pi:M\times (0,\infty)\to M$ is the projection). Notice that
$\pi^*d\alpha$ is not a symplectic form on $M\times (0,\infty)$. This 
is however a non-negative number when evaluated on a holomorphic 
curve that vanishes if and only if the component is a trivial cylinder 
over a Reeb orbit.
\end{itemize}
The name symplectic area is justified by the following result,
proved in \cite[Proposition 9.4]{boelhowize} in the language
of symplectic energy.
\begin{prop}\label{prop:limitarea}
Let $u_n$ be a sequence of $J_n$-holomorphic curves for a sequence
of almost complex structures that stretch the neck of a closed contact 
type hypersurface (possibly empty). Then, the symplectic area of the limit 
building is the limit of the symplectic area of the curves $u_n$.
\end{prop}
For a subbuilding $B' \subset B$ the symplectic area $\ca(B')$ has
an easy geometric interpretation, which we will not use:
it simply represents, up to a small error that goes to $0$ with
$n \to \infty$, the symplectic area of the restriction of $u_n$ to a
subset that converges to this subbuilding (see \cite{boelhowize} for 
the precise definition of this convergence). We will use however
the following fact: in an exact setting (to be made precise in the
following statement) the symplectic area of a building
depends only on its asymptotic Reeb orbits.

\begin{lemma}\label{le:areasubb}
Let $(X,\om)$ be a symplectic cobordism between $(M^+,\alpha_+)$ and $(M^-,\alpha_-)$
such that $\om = d\lambda$ is exact. Let $(M,\alpha)$ be a closed contact
type hypersurface that splits $X$ in $X_+$ and $X_-$. Let $B$ be a finite energy 
holomorphic building with components in the floors $M^-\times (0,\infty)$, 
$M^+\times (0,\infty)$, $M\times (0,\infty)$,  $\wdt X_-$ and $\wdt X_+$.
We call $\{\gamma_i^+\}$ and $\{\gamma_i^-\}$ the positive and negative Reeb 
orbits to which $B$ is asymptotic at its positive and negative punctures.
Then the symplectic area of $B$ is
$$
\ca(B)=\sum_i\int_{\gamma_i^+}\lambda - \sum_j \int_{\gamma_j^-}\lambda.
$$
\end{lemma}
\noindent {\it Proof:} The proof is almost immediate. The formula is additive with 
respect to the decomposition of the building in different floors, hence we only need
to verify the formula for each floor. 

For the components in $\wdt X_+$ or $\wdt X_-$ we can apply Stokes' theorem. 
For a component in $M\times (0,\infty)$ 
given by $u: \Sigma\priv Z \to M\times (0,\infty)$, 
where $Z := \{z_i\}$ is a set of punctures of the closed Riemann surface $\Sigma$ and 
$\Gamma := \{\gamma_i^+\} \cup \{\gamma_j^-\}$ their corresponding asymptotic closed 
Reeb orbits, we proceed as follows. 
Let $f:(0,1)\to (0,\infty)$ be an increasing diffeomorphism and $F:=\id\times f: 
M\times(0,1)\to M\times (0,\infty)$ the induced diffeomorphism. By the asymptotic 
properties of the holomorphic buildings with finite energy, the pull-back 
$F^*u: \Sigma\priv Z \to M\times (0,1)$ can be considered as a smooth map of the 
compact surface $\hat \Sigma$ obtained by blowing-up the punctures to $M\times [0,1]$. 
This map is therefore amenable to Stokes theorem. Moreover, $F^*(\pi^*d\alpha) = 
(\pi\circ F)^*d\alpha=\pi^*d\alpha=\pi^*d\lambda$ because $d\alpha_{|M}=\om_{|M} 
= d\lambda_{|M}$. Thus,
$$
\begin{array}{ll}
\ds \int_{\Sigma\priv Z} u^*\pi^*d\alpha & \ds =\int_{\textrm{int}\,\hat \Sigma} (F^*u)^*F^*\pi^*d\alpha 
= \int_{\textrm{int}\,\hat \Sigma} (F^*u)^*\pi^*d\lambda=\int_{\hat \Sigma} 
(F^*u)^*\pi^*d\lambda=\int_{\partial \hat \Sigma} (\pi\circ u)^*\lambda \\
&\ds = \sum_{\gamma\in \Gamma}\int_{\gamma}\lambda=\sum_i\int_{\gamma_i^+}\lambda - \sum_j \int_{\gamma_j^-}\lambda.
\end{array}
$$


\subsubsection{Compactness for punctured holomorphic disks II}\label{sec:sft_compactnessII}

We recall the setting: $g$ is a Riemannian metric on $L$ and $M=\{\|p\|_g=1\}$.
Let  $\cw'$ be a smoothly bounded open neighborhood of the zero section 
such that $(M' := \partial \cw', \alpha' )$ is a closed hypersurface of contact type.
Assume that the projection $\pi: T^*L \to L$ induces an isomorphism
$\pi_*: H_1(\cw') \to H_1(L)$. We consider a sequence $J_n \in \mathcal{J}_{\mathrm{cyl}, g}^{\infty}$ 
of almost complex structures which stretch the neck near $M'$, in particular these almost complex 
structures are cylindrical and fixed in the complement of a compact set of $T^*L$.
For a chosen class $\beta \in H_1(L)$
and ${\tilde\gamma}$ a lift of a minimal geodesic $\gamma$ representing $\beta$,
let $u_n: D \backslash \{ 0 \} \to T^*L$ be a sequence of
$J_n$-holomorphic punctured disks such that $u_n(\partial D) \subset L$,
$[u_n(\partial D)] = \beta$ and $u_n$ is asymptotic to $\tilde \gamma$ at 0. We write $\widetilde{T^*L \backslash \cw'}$ for
the completion of $T^*L \backslash \cw'$ with a negative cylindrical end
$(0,1) \times M'$ and $\widetilde{\cw'}$ for the completion of $\cw'$ with
a positive cylindrical end $(1,\infty) \times M'$.
Note that these completions are symplectomorphic to the underlying spaces.
Recall also that the splitting process provides  almost
complex structures, each denoted $\wdt J$, on $(0,\infty)\times M$, $\widetilde{T^*L \backslash \cw'}$, 
$(0,\infty)\times M'$ and $\widetilde{\cw'}$. These structures are cylindrical at
infinity in $\widetilde{T^*L\priv \cw'}$ and $\widetilde{\cw'}$ for the contact form $\alpha'$, and 
cylindrical in $(0,\infty)\times M'$ (meaning that it is compatible with $\alpha'$ and $\R_*^+$-invariant).

\begin{theorem}\label{thm:neckstretchcompact}
The sequence $(u_n)_{n \in \N}$ converges to a holomorphic building with the following 
properties: \vspace*{-,2cm}
\begin{itemize}
\its it has no component in $(0,\infty)\times M$ (in other terms, the top floor of the holomorphic
building is $\widetilde{T^*L \backslash \cw'}$), \vspace*{-,2cm}
\its it contains, as a subbuilding, a disc with one positive puncture with boundary on $L$ and
asymptotic at the puncture to a Reeb orbit $\tilde \gamma'$ in $M'$ that projects
  to a representative of $\beta\in H_1(L)$, \vspace*{-,2cm}
 \its if $\widetilde{T^*L\priv \cw'}$ is exact (i.e. $\alpha'=\lambda_{|M'}$) we can be more precise: 
the bottom level contains
a $\widetilde{J}$-holomorphic map $v_{0}: (D \backslash \{0 \}, \partial D) \to
(\widetilde{\cw'}, L)$ that is asymptotic to a Reeb orbit ${\tilde \gamma'}$ in $M'$
that projects to  to a representative of $\beta\in H_1(L)$).
\end{itemize}
\end{theorem}

\noindent{\it Proof:}
Since all $u_n$ are asymptotic to $\tilde\gamma$ at $0$, by corollary \ref{cor:energy_bound}
the energy $\mathcal{E}(u_n)$ is uniformly bounded. The SFT compactness theorem for
split symplectic manifolds \cite{boelhowize, abbas} implies that our sequence
$u_n: (D \backslash \{0\}, \partial D, j) \to (T^*L, L, J_n)$  converges to a
stable holomorphic building. In our situation this is given by the following data.

\begin{enumerate}
\item[(i)] $v_0: (S_0 \backslash Z_0, \partial S_0, j_0) \to (\widetilde{\cw'}, L,
\widetilde{J})$ is a proper $\widetilde{J}$-holomorphic map from a compact Riemann
surface $S_0$ with boundary $\partial S_0$ with a finite set of punctures $Z_0 \subset
S_0 \backslash \partial S_0$ to the almost complex manifold $(\widetilde{\cw'},
\widetilde{J})$ with finite energy. Moreover, although $S_0$ may have several
components, $\partial S_0$ is a unique circle, $v_0(\partial S_0)\subset L$, and
since $(u_n)_*[\partial D] = \beta$, $v_0(\partial S_0)$ represents the class $\beta$ in $H_1(L)$.
The almost complex structure $\widetilde J$ is cylindrical at infinity on
$\widetilde{\cw'}$ (for the contact form $\alpha'$).
\item[(ii)]
For $k = 1, \ldots, {p-1}$ we have holomorphic maps
\[
v_k: (S_k \backslash Z_k, j_k) \longrightarrow ((0,\infty) \times M', \widetilde J),
\]
from closed Riemann surfaces $S_k$ with a finite set of punctures $Z_k \subset S_k$
to the symplectization $(0,\infty) \times M'$. The almost complex structure $\widetilde J$
is cylindrical with respect to the contact form $\alpha'$. As in section
\ref{sec:compactness1}, there are decoration maps $\Phi_k$ for $k = 1, \ldots, p-1$ that
glue all of the negative punctures of $Z_k$ to the positive punctures of $Z_{k-1}$
in the corresponding oriented blow-ups $\overline{S}_{k}$ and $\overline{S}_{k-1}$.
\item[(iii)] $v_p: (S_{p} \backslash Z_{p}, j_{p})\to
(\widetilde{T^*L \priv \cw'}, \widetilde{J})$ is a $\widetilde{J}$-holomorphic
map from a closed Riemann surface $S_p$ with a finite set of
punctures $Z_{p} \subset S_{p}$ to the almost complex manifold $(\widetilde{T^*L
\priv \cw'}, \widetilde{J})$ such that $v_{p}$ has finite energy. The almost
complex structure $\widetilde{J}$ is cylindrical at infinity at both ends of
$\widetilde{T^*L\priv \cw'}$. Furthermore, similar to (ii), there is a decoration map
$\Phi_{p}$ that glues all of the negative punctures of $Z_{p}$ to the positive
punctures of $Z_{p-1}$ in the corresponding oriented blow-ups $\overline{S}_{p}$ and
$\overline{S}_{p-1}$.
\item[(iv)] For $k = {p+1, \ldots, q}$ we have holomorphic maps
\[
v_k: (S_k \backslash Z_k, j_k) \longrightarrow ((0,\infty) \times M, \widetilde J)
\]
from closed Riemann surfaces $S_k$ with a finite set of punctures $Z_k \subset
S_k$ to the symplectization $(0,\infty) \times M$. As before, $\widetilde J$ is cylindrical
on $(0,\infty)\times M$ with respect to the contact form $\alpha$. Similar to (ii), there are
decoration maps $\Phi_k$ for $k = p+1, \ldots, q$ that glue all of the negative
punctures of $Z_k$ to the positive punctures of $Z_{k-1}$ in the corresponding
oriented blow-ups $\overline{S}_k$ and $\overline{S}_{k-1}$
\item[(v)] Denote by
\[
\overline{S} := \overline{S}_0 \cup_{\Phi_1} \overline{S}_1 \cup \ldots \cup_{\Phi_{p}}
\overline{S}_{p} \cup_{\Phi_{p+1}} \ldots \cup_{\Phi_q} \overline{S}_q
\]
the piecewise smooth surface obtained by gluing together
all blow ups $\overline{S}_k$ at their punctures via the decoration maps
$\Phi_k$. Then $\overline{S}$ has no nodal point and is homeomorphic to
$\widehat{D_0} := \R_{\geq 0} \times S^1 \cup \{ \infty \} \times S^1$,
which is homeomorphic to the oriented blow-up of $D$ at $0$. Again, this is
due to the exactness of the manifold (not to be confused with exactness of the cobordism),
which prevents the formation of sphere or disc bubbles by lemma \ref{le:areasubb}.
\end{enumerate}
\begin{figure}[h!]
\begin{center}
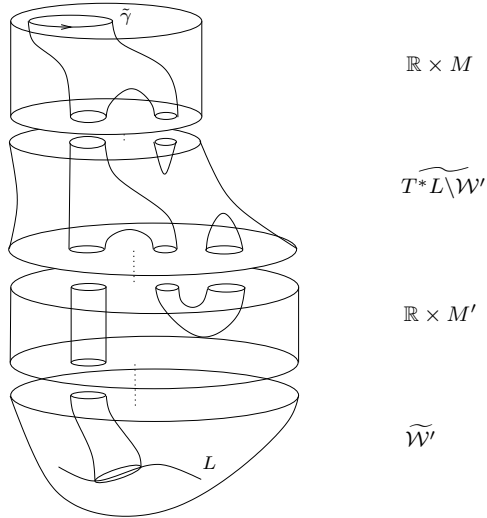
\end{center}
\caption{A limit holomorphic building when stretching the neck of $M'$}
\label{fig:sft2}
\end{figure}

We now analyze the limit holomorphic building. Notice first that the argument 
for the top floor in the proof of theorem \ref{thm:sft_compactness} applies 
word for word and proves the first assertion of our theorem: there is no component 
of the limit holomorphic building in  $(0,\infty)\times M$.

Now by a subbuilding we mean a collection of connected components $(S_{i}^j)$
of the $S_i$ together with the maps $v_i^j:=v_{i|S_i^j}$. We say that a
subbuilding is connected if the surface obtained by gluing the $\overline{S}_i^j$
{\it via} the corresponding decoration maps is connected.  We concentrate first on the lowest level of the
limit building, i.e. we consider only the map $v_0: (S_0 \backslash Z_0,
\partial S_0, j_0) \to (\widetilde{\cw'}, L, \widetilde{J})$ and restrict to
the component of $S_0$ containing the boundary $\partial S_0$. We label this
component by $S_0^{\partial}$ and write $v_0^{\partial}$ for the restricted map.
Since the domain of our sequence of maps is given by $D \backslash \{ 0 \}$,
we see that $S_0^{\partial}\backslash Z_0^{{\partial}} = D \backslash
\{ z_1, \ldots, z_l \}$ and $v_0^{\partial}$ is asymptotic to a $\tilde \gamma_i'$
at the puncture $z_i$.
For each puncture $z_i$, let $B(z_i)$ denote the maximal connected subbuilding
which has $\tilde \gamma_i'$ as a single  negative puncture.
Since $\overline{S}$ is homeomorphic to an annulus, exactly one of the $B(z_i)$,
say $B(z_1)$, is a topological annulus, while all the other ones are topological
discs. The subbuilding composed of $S_0^\partial$ and the $B(z_i)$,
$i=2,\dots,n$ is the required subbuilding: it is indeed a topological disc with one
positive puncture at $z_1$ and with boundary on $L$. Moreover, its projection to
$L$ - well-defined because each floor of the decomposition of $T^*L$ naturally
projects to $L$ - provides a homotopy between $\pi(\tilde{\gamma_1}')$ and
$v_0(\partial S_0)$, which represents the class $\beta$ in $H_1(L)$.

Finally, when $\wdt{T^*L\priv \cw'}$ is an exact cobordism we  show that the $B(z_i)$,
$i\geq 2$, simply do not appear. Indeed, for each $i\geq 2$, $B(z_i)$ has a
maximal floor, whose components therein have only negative punctures.
By the maximal principle these maximal floors cannot lie in $(0,\infty)\times M$
nor $(0,\infty)\times M'$, so they have to be contained in $\widetilde{T^*L \backslash
\cw'}$. Now $\widetilde{T^*L \backslash \cw'}$ is an exact
cobordism by assumption and the
almost complex structure $\widetilde J$
tames the exact symplectic form. By Stokes' Theorem components with
only negative punctures do not exist in $\widetilde{T^*L \backslash \cw'}$,
since they would have non-positive symplectic area. Hence we may rule out the
existence of such holomorphic subbuildings and we see that $v_{0}^{\partial}$
has only one (positive) puncture, whose asymptotic we call $\tilde \gamma'
\subset M'$. Proving that $[\tilde \gamma']=\beta$
is similar to the non-exact case above. \cqfd


\subsection{Proof of theorem \ref{thm:existence}}\label{sec:proofexistence}

Let $g$ be a Riemannian metric on $L$ satisfying the assumptions of theorem
\ref{thm:existence} and let $\eps > 0$. We divide the proof into two steps.
We first prove our theorem for the Riemannian metric $g':=g_{\eps,\beta}$
provided by proposition \ref{prop:goodmetric}. In particular it is $\eps$-close
to $g$ in the $\cc^0$-norm. We consider the metric $g$ itself in the second step.

\paragraph{Step 1: proof of theorem \ref{thm:existence} for those $J$ that coincide
with $J_{g'}$ at infinity.}
Let $J_{g'}\in \cj_{\cyl,g'}^\infty$ be the almost complex structure as
defined in section \ref{sec:jpart}. Let $\cj_{g'}$ denote in this paragraph the set of
almost complex structures on $T^*L$ that coincide with $J_{g'}$ outside
a compact subset. We first show that the set $\cm(J,\beta)$
(defined on p. \pageref{thm:existence}) contains at least one element
for all $J \in \cj_{g'}$. Recall that by construction $g'$ has a unique minimizing
geodesic $\gamma(\beta)$ of the form $\gamma(\beta')^k$ for some primitive geodesic
$\gamma(\beta')$, where $k\geq 1$ and $\beta = k \beta'$.
Lemma \ref{le:injpoints} then implies  that the space $\cm(J,\beta')$ and its subspace 
consisting of somewhere injective curves coincide for all $J \in \cj_{\cyl,
g'}^\infty$.

The proof relies on the compactness theorem \ref{thm:sft_compactness} and the
Fredholm theory for punctured holomorphic discs with boundary on a Lagrangian
as summarized in the following lemma. Although it may be considered folklore
we explain this lemma in appendix \ref{sec:fredholm}.
\begin{lemma}\label{le:fredholmproof}
There is a subset $\cj_{g'}^{\reg}\subset \cj_{g'}$, dense in the
$\cc^\infty$-topology, such that for all $J \in \cj_{g'}^{\reg}$ the set
$\cm(J,{\beta'})$ is a manifold of dimension $0$. Moreover, if $J_0,J_1\in
\cj_{g'}^{\reg}$,
for a generic smooth path $\{J_t\}_{t\in [0,1]}$ in $\cj_{g'}$
that interpolates between $J_0$ and $J_1$, the space
$\cup_{t\in[0,1]}\cm(J_t,{\beta'})$ is a smooth manifold of dimension $1$.
\end{lemma}

Proposition \ref{prop:goodmetric} shows that $\cm(J_{g'},{\beta'})$ consists of exactly
one element and $J_{g'}\in \cj_{g'}^{\reg}$. Let $J\in \cj_{g'}^{\reg}$ and $\{J_t\}$
be a regular path between $J_{g'}$ and $J$ consisting of almost complex structures which
coincide with $J_{g'}$ outside a compact set, as defined in lemma \ref{le:fredholmproof}.
Then $\cup_{t\in [0,1]}\cm(J_t, {\beta'})$ is a one-dimensional cobordism between
$\cm(J_{g'},{\beta'})$ and $\cm(J,{\beta'})$ by lemma \ref{le:fredholmproof} and is compact
by theorem \ref{thm:sft_compactness}. Both ends of the cobordism therefore have the
same - odd - parity, so $\cm(J,{\beta'})$ is non-empty when $J$ is regular. Finally,
if $J\in \cj_{g'}$ does not belong to $\cj_{g'}^{\reg}$, let $J_n\in \cj_{g'}^{\reg}$
be a sequence of regular almost complex structures that converge to $J$ in the
$\cc^\infty$-topology. The previous argument shows that there exist elements
$u_n\in \cm(J_n,{\beta'})$ for all $n$. By theorem \ref{thm:sft_compactness} we
can extract from $u_n$ a subsequence that converges to an element $u \in
\cm(J,{\beta'})$, thus $\cm(J,{\beta'})$ is non-empty. {Finally, if $u\in \cm(J,\beta')$,
 its $k$-cover $u(z^k)\in \cm(J,\beta)$, so $\cm(J,\beta)$ itself is indeed non-empty. }This completes
the first step of the proof.

\paragraph{} Before we prove theorem \ref{thm:existence} for $g$ itself, let us
recall that by example \ref{ex:switch}, establishing the existence of a
punctured holomorphic disc asymptotic to a lift of a geodesic $\gamma$ at
the puncture for all structures in $\cj^\infty_{\cyl, h}(T^*L)$ is equivalent
to establishing their existence for all elements of $\cj_{\cyl,h}^\infty
(\cw(L,h,r))$ (the almost complex structures that are cylindrical at infinity
in $\cw(L,h,r)$). We will freely switch from one problem to the other in the
remaining of this paper. In particular the previous theorems (compactness, neck
stretching) are applicable in this new setting.

\paragraph{Step 2: from $g'$ to $g$.} Let now $J\in \cj_{\cyl,g}^\infty(\cw(L,g,1))$
Since $g'$ is $\eps$-close to $g$ in the $\cc^0$-topology,
$\cw(L,g',(1-\eps)^{-\nf 12}):=\{\Vert  p\Vert_{g'}<(1-\eps)^{-\nf 12}\}\Supset
\cw(L,g,1)$. We can therefore consider a sequence of almost complex structures
$J_n$ that are cylindrical at infinity in
$\cw(L,g',(1-\eps)^{-\nf 12})$, stretch the neck along $\partial \cw(L,g,1)$, and
converge to $J$ in $\cc^{\infty}_{\mathrm{loc}}$ on $\cw(L,g,1)$.
By the discussion above, the first step of our proof guarantees that for every
$n$, there exists a punctured $J_n$-holomorphic  disc $u_n:(D\priv \{0\},\partial D)
\to (\cw(L,g',(1-\eps)^{-\nf 12}), L)$ asymptotic to  $\tilde\gamma'$ (the lift
of the unique $g'$-minimizing curve in class $\beta$).
Applying theorem \ref{thm:neckstretchcompact} to the sequence of  punctured
holomorphic discs $u_n$, we get a limit holomorphic building
$(v_0,\dots,v_p,\dots,v_q)$ as described in the section \ref{sec:sft_compactnessII}.
Since we have an exact symplectic cobordism, the domain of $v_0$ contains a punctured
disc $D\priv \{0\}$ and $v_0^\partial:=v_{0| D\priv\{0\}}$
is asymptotic at the puncture to a Reeb orbit $\tilde \gamma\subset \partial \cw(L,g,1)$,
whose projection to $L$ is a geodesic $\gamma$ that represents the class $\beta$.
Gluing all the components of the building together, except for $(v_0^\partial, D\priv\{0\})$,
we get a subbuilding which is a punctured $2$-sphere with one positive puncture
asymptotic to $\tilde \gamma'$ and one negative puncture asymptotic to
$\tilde \gamma$. By lemma \ref{le:areasubb} (we are indeed in the exact setting),
the symplectic area of this subbuilding is therefore
$$
0 < (1-\eps)^{-\nf 12}\ell_{g'}(\gamma')-\ell_g(\gamma),
$$
so $\ell_g(\gamma)<(1-\eps)^{-\nf 12}\ell_{g'}(\gamma')$. On the other hand,
since $g'$ is $\eps$-close to $g$ in the $\cc^0$-topology, we get
$$
\ell_g^{\min}(\beta) \leq \ell_g(\gamma) < \frac{1}{(1-\eps)^{\nf 12}} \ell_{g'}^{\min}(\beta)
\leq \sqrt{\frac{1+\eps}{1-\eps}}\ell_g^{\min}(\beta).
$$
Since $g$ has a discrete length spectrum in the class $\beta$, we can take
$\eps$ much smaller than the gap between the length of the $\beta$-minimizing
$g$-geodesic and the other elements in this spectrum. In view of the estimation
above, we see that $\gamma$ must then be a $\beta$-minimizing geodesic.

Notice also that, as stated in remark \ref{rk:existence}, if $g$ does not
have a discrete length spectrum in class $\beta$, the proof above shows
the existence of a punctured holomorphic disc with boundary on $L$, now
not necessarily asymptotic to the minimal geodesic in class $\beta$, but
still asymptotic to a geodesic in class $\beta$ with length $\eps$-close
to the minimal length.
\cqfd


\section{$\cc^0$-rigidity of the area homomorphism and the Maslov class}\label{sec:rigspec}

\subsection{The Maslov class of a Lagrangian}\label{sec:defmaslov}

We first define the Maslov index $\mu^\tau_L$ of a Lagrangian submanifold $L$. Let $u:(D,\partial D)\lra (M,L)$ 
be a disc with boundary on a Lagrangian submanifold $L\subset M$ and let $\tau$ be a 
Lagrangian distribution of $u^*TM$. Then $\mu^\tau_L(u)$ is the Maslov index of the loop 
of Lagrangian subspaces $u_{}^*TL\subset u^*TM$ relative to $\tau_{|\partial D}$. Since 
the second homotopy group of the Lagrangian Grassmannian vanishes, this index does not 
depend on the choice of $\tau$ nor the representative of $[u]$ in $\pi_2(M,L)$. We therefore get a 
map $\mu_L:\pi_2(M,L)\to \Z$ that assigns to each class this index.

When a symplectic manifold has a globally defined Lagrangian distribution $\tau$ - e.g. 
cotangent bundles with their vertical distributions - we can define a map $\mu^\tau_L : 
H_1(L)\to \Z$ by simply assigning to each class $\beta \in H_1(L)$ the Maslov index of 
the loop $t\mapsto T_{\gamma(t)}L$ relative to $\tau$ for an arbitrary choice of
representative $\gamma$ of $\beta$. Notice that, in this case, the Maslov class 
depends on the choice of $\tau$. If $\tau'$ is another choice of Lagrangian distribution 
on $M$, one can compute a class $\mu(\tau,\tau'):H_1(M)\to \Z$ by computing the relative 
Maslov index of loops $\tau_{\gamma(t)}$ relative to $\tau'$. Then it is easy to check that
for every $\beta \in H_1(L)$,
\begin{equation}\label{eq:transfermaslov}
\mu^{\tau'}_L(\beta) = \mu_L^{\tau}(\beta) + \mu(\tau,\tau')(\iota_*\beta), 
\end{equation}
where $\iota: L \hra M$. In the case of a Lagrangian submanifold in a cotangent bundle
we take $\tau$ to be the vertical distribution and we sometimes simply write $\mu_L$.

\subsection{Area homomorphisms and Maslov classes of closeby Lagrangian submanifolds}

We now aim at explaining the proofs of theorems \ref{thm:rigspec} and \ref{thm:closeclose}.
We recall the context. Given a closed Riemannian manifold $(L,g)$ and a class $\beta\in H_1(L)$,
we denote by $\ell_g^{\min}(\beta)$ the minimal length of loops that
represent the class $\beta$. The Riemannian structure endows
the zero section of the cotangent bundle with a basis of \nbds denoted
$\cw(L,g,\eps)$. By the Weinstein \nbd theorem, whenever $L$ has a Lagrangian
embedding into a symplectic manifold $M$, this embedding can be extended to
a symplectic embedding of $\cw(L,g,\eps)$. We simply denote such data by
$L\subset \cw(L,g,\eps)\subset M$. We recall theorem \ref{thm:closeclose} and then
prove theorem \ref{thm:rigspec}.

\setcounter{thm-repeat}{2}
\begin{thm-repeat}
Let $(L,g)$ and $(L',g')$ be two closed Riemannian manifolds and $\iota:
L'\hra  (T^*L,d\lambda)$ a Lagrangian embedding such that $\pi_* \circ \iota_*:
H_1(L')\to H_1(L)$ is an isomorphism. Assume that $\iota$ extends to a symplectic
embedding $\ci$ of a neighborhood of $L'$ such that
$$
L\subset  \ci\big(\cw(L',g',r')\big)\subset\cw(L,g,r)\subset T^*L
$$
for some $r,r'>0$. Then for all $\beta' \in H_1(L';\Z)$ we have \vspace{-,3cm}
\begin{itemize}
\item[a)] $\left| \iota^*\lambda (\beta')\right|\leq  r\ell_g^{\min}(\pi_*\circ \iota_*\beta')$,\vspace{-,2cm}
\item[b)] $\mu_{\iota(L')}(\iota_*\beta')=0$.
\end{itemize}

\end{thm-repeat}
\noindent{\it Proof of theorem \ref{thm:rigspec}:}
Let $h:(M,\om)\to (M',\om')$  be a symplectic homeomorphism that takes a
closed Lagrangian submanifold $L\subset M$ to a smooth, hence Lagrangian,
submanifold $L'\subset M'$. Let $g,g'$ be Riemannian metrics on $L,L'$,
respectively. By the Weinstein \nbd theorem there exists $\eps_0,\eps_0'$
such that $\cw(L,g,\eps_0)\subset M$  and $\cw(L',g',\eps'_0)\subset M'$
and we can even assume (by decreasing $\eps_0$ if necessary) that
$h\big(\cw(L,g,\eps_0)\big)\Subset \cw(L',g',\eps_0')$. Fix $\eps<\eps_0$
and choose a sequence of symplectic diffeomorphisms $f_n$ that approximate
$h$. For $n$ sufficiently large we have
$$
f_n\big(\cw(L,g,\eps)\big) \,\Subset\, \cw(L',g',\eps_0').
$$
Since moreover $h(L)=L'$, there exists $0 < \eps'' < \eps' \leq \eps_0'$ such
that $\cw(L',g',\eps'') \Subset h(\cw(L,g,\eps)) \Subset \cw(L',g',\eps')$
and thus $\cw(L',g',\eps'') \Subset f_n(\cw(L,g,\eps)) \Subset \cw (L',g',\eps')$
for all $n$ large enough. Since $\{\cw(L,g,\eps)\}_{\eps>0}$ is a basis of
\nbds of $L$, $\eps'$ can be chosen to tend to $0$ when $\eps$ goes to $0$.
Putting all this together we therefore get
$$
L' \; \subset \; f_n(\cw(L,g,\eps)) \; \Subset \; \cw(L',g',\eps')
$$
for all $n$ sufficiently large.

We now prove theorem \ref{thm:rigspec}.a). Let $\sigma\in H_2(M,L;\Z)$ be 
represented by a smooth surface $\Sigma$
with boundary on $L$ and $\beta := \partial \sigma\in H_1(L;\Z)$ be represented
by $\partial \Sigma$. Then $f_n(\Sigma)$ is a smooth surface with boundary
on $f_n(L)$ and
$$
\ca_{\om'}(f_n(\Sigma))=\ca_\om(\Sigma)=\ca_\om^L(\sigma),
$$
because $f_n$ is a symplectic diffeomorphism. Moreover, since $f_n(L)\subset
\cw(L',g',\eps') \subset (T^*L',d\lambda')$,
for $n$ large enough we have
$$
\ca_{\om'}^{L'}(h_*\sigma) = \ca_{\om'}(f_n(\Sigma)) + f_n^*\lambda'(\partial \sigma).
$$
This can be seen as follows: since $f_n\to h$ in the $\cc^0$-norm, for
$\pi': T^*L' \to L'$ we see that $\pi'\circ f_n : L\to L'$ is $\cc^0$-close
to $h_{|L}$. So $\pi'\circ f_{n|L}$ induces an isomorphism in homology and
$\pi'_*\circ f_{n\,*}= h_{*} : H_1(L)\to H_1(L')$. Then we can extend
$f_n(\Sigma)$ by a straight cylinder (in the coordinates provided by
$\cw(L',g',\eps_0')\supset f_n(L)$) that connects $f_n(\partial \Sigma)
\subset f_n(L)$ to $\pi'\big(f_n(\partial \Sigma)\big) \subset L'$
(see \cite[Lemma 5.1]{buop} for the details). The $\om'$-area of this
cylinder is given by $f_n^*\lambda'(\partial \sigma)$. Rewriting we thus obtain
$$
\ca_{\om'}^{L'}(h_{*}\sigma)=\ca_\om^L(\sigma)+f_n^*\lambda'(\beta).
$$
By theorem \ref{thm:closeclose}.a),
$$
|f_n^*\lambda'(\beta)| \;\leq\; \eps'\ell_{g'}^{\min}(\pi'_* \circ {f_n}_* \beta),
$$
and therefore
$$
\big|\ca_\om^L(\sigma) - \ca_{\om'}^{L'}(h_{*}\sigma)\big|
\leq \eps' \ell_{g'}^{\min}(h_*\beta).
$$
Since this holds for all $\eps'$ sufficiently small, in the limit we
get $\ca_\om^L(\sigma)=\ca_{\om'}^{L'}(h_*\sigma)$.

We prove theorem \ref{thm:rigspec}.b). Let $D \in \pi_2(M,L)$, 
$[\gamma]:=\partial D \in \pi_1(L)$ and $\beta := [\gamma] \in H_1(L)$. As 
in the previous paragraph we consider $f_n$ also as an embedding into a 
Weinstein \nbd of $L'$ for $n$ large enough. An easy computation 
based on \eqref{eq:transfermaslov} shows that
$$
\mu_L(D,M) = \mu_{f_n(L)}(f_{n*}D,M') = \mu_{L'}(h_*D,M') 
- \mu_{f_n(L)}^{\tau'}(f_{n*}\beta,T^*L'),
$$
where $\tau'$ is the vertical fiber distribution on $T^*L'$ and in our notation
we have specified in which manifolds the different Maslov indices are 
computed. By theorem \ref{thm:closeclose}.b)  we have 
$\mu_{f_n(L)}^{\tau'}(f_{n*}\beta, T^*L') = 0$, so $\mu_L(D)=\mu_{L'}(h_*D)$. \cqfd


\subsection{Proof of theorem \ref{thm:closeclose}.a)}\label{sec:proofa}

Let $\iota : L' \hra T^*L$ be a Lagrangian embedding which extends to a
symplectic embedding $\ci : \cw(L',g',r') \hra T^*L$ such that
$$
L,\iota(L') \;\subset\; \ci(\cw(L',g',r')) \;\subset\; \cw(L,g,r)
\;\subset\; T^*L.
$$
We fix $\beta'\in H_1(L';\Z)$ and write $\beta := \pi_* \circ \iota_*\beta'$.
We aim at proving that $| \iota^*\lambda (\beta')|\leq  r\ell_g^{\min}
(\beta)$. We can obviously slightly perturb the metric $g$ to prove
the theorem, so we assume henceforth that there is a unique minimizing geodesic $\gamma$ in the class $\beta$.
Notice that $\iota^*\lambda (\beta')$ is the symplectic area of the cylinder
\fonction{C_{\iota\circ\gamma'}}{[0,1]\times \R/\ell\Z}{\cw(L, g, r)}{(s,t)}{s\,\iota\circ\gamma'(t),}
where $\gamma' : \R/\ell\Z \to L'$ is any smooth curve that represents the class
$\beta'$.

Let $J_n$ be a sequence of almost complex structures on $\cw(L,g,r)$ that are
cylindrical at infinity (near $\partial \cw(L,g,r)$), and that stretch
the neck along the contact type hypersurface $\ci(\partial \cw(L',g',r')) \subset \cw(L,g,r)$. By theorem
\ref{thm:existence} there exists for each $n$ a $J_n$-holomorphic map
$$
u_n:(D\priv \{0\},\partial D) \lra (\cw(L,g,r), L)
$$
which is asymptotic to the lift $\tilde \gamma$ of the minimal geodesic $\gamma$
at $0$.
An easy computation shows that the symplectic area $\ca(u_n)= r \ell_g^{\min}(\beta)$.
By theorem \ref{thm:neckstretchcompact} a subsequence of $(u_n)$ converges to
a holomorphic building that contains, as a subbuilding $B$,
a punctured disc with boundary on $L$ that represents the class $\beta$ and one
positive puncture asymptotic to $\ci(\tilde \gamma')$ in $\ci(\partial\cw(L',g',r'))$, the lift of a geodesic
$\gamma': \R/\ell' \Z \to L'$ that represents the class $\beta'$.
By proposition \ref{prop:limitarea}
the total symplectic area of the building is $r\ell_{g}^{\min}(\beta)$ and by
non-negativity of the area of each subbuilding we get $\ca(B) \leq
r\ell_g^{\min}(\beta)$.

Now $\cw(L,g,r)$ is an exact symplectic manifold with
$L$ as an exact Lagrangian submanifold, so by lemma \ref{le:areasubb} we can compute the area of $B$
by means of any piecewise smooth cylinder 
that connects $\ci(\tilde \gamma')$ to a curve on $L$ in the class $\beta$ in $\cw(L,g,r)$.
Such a cylinder can be obtained by concatenation of two cylinders: one
is given by
\fonction{\ci\circ C_{\tilde{\gamma}'}}{[0,1] \times \R/\ell'\Z}{ \cw(L,g,r)}{(s,t)}{
\ci\big(s\tilde \gamma'(t)\big),}
and the other one by $C_{\iota\circ\gamma'}$.
Notice that $\ca(\ci\circ{C}_{\tilde{\gamma}'}) \geq 0$ because it is the image by the symplectomorphism $\ci$
of a trivial cylinder on a  Reeb orbit of $\partial \cw(L' , g', r')$. The area of the second
cylinder is $\iota^*\lambda([\gamma']) = \iota^*\lambda(\beta')$,  as already
noticed.
\begin{center}
\begin{figure}[h!]
\hspace*{4cm}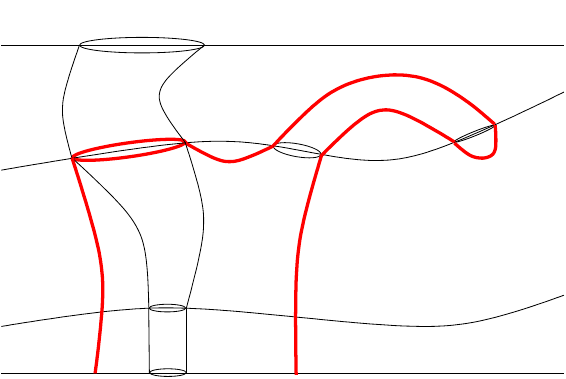
\caption{Estimating $|\iota^*\lambda([\gamma'])|$: $r\ell_{g}(\gamma)>\ca(B) = \ca\big(\ci(C_{\tilde \gamma'})\big) 
+ \ca\big(C_{\iota\circ \gamma'}\big)\geq 0+\iota^*\lambda([\gamma'])$.}
\label{fig:break-cyl}
\end{figure}
\end{center}
We therefore obtain
$$
r \ell_g^{\min}(\beta) \geq \ca(B) = \ca(\ci\circ C_{\tilde{\gamma}'})
+ \iota^*\lambda(\beta') \geq \iota^*\lambda(\beta').
$$
Considering the class $-\beta'$ instead of $\beta'$ in the last inequality,
we have $\ell_g^{\min}(-\beta) = \ell_g^{\min}(\beta)$, thus
only the sign of $\iota^*\lambda(\beta')$ changes and we get
$$
\big| \iota^*\lambda(\beta') \big| \leq r\ell_g^{\min}(\pi_*\circ \iota_*\beta').\hspace*{2cm} 
$$
\cqfd


\subsection{Proof of theorem \ref{thm:closeclose}.b)}

We recall the assumptions of the theorem:
$\iota : L' \hra T^*L$ is a Lagrangian embedding which extends to a
symplectic embedding $\ci : \cw(L',g',r') \hra T^*L$ such that
$$
L,\iota(L') \;\subset\;\cw':= \ci(\cw(L',g',r')) \;\subset\; \cw := \cw(L,g,r)\subset T^*L.
$$
Without loss of generality we assume in the following that we have perturbed $g$ so 
that there is a unique minimizing geodesic $\gamma$ in class $\beta$. Again $\beta'\in H_1(L';\Z)$, 
$\beta := \pi_* \circ \iota_*\beta'$ and we aim at proving that $\mu_{\iota(L')}(\iota_*\beta')=0$. 
As before, we consider a sequence of almost complex structures $J_n$ on $\cw$ that are
cylindrical at infinity and that stretch the neck along the contact type hypersurface 
$\partial \cw' \subset T^*L$. By theorem \ref{thm:existence} there exists for each $n$ a 
$J_n$-holomorphic map $u_n:(D\priv \{0\},\partial D) \lra (\cw, L)$ which is asymptotic to the 
lift $\tilde \gamma$ of $\gamma$ at $0$. 
Letting $n$ to infinity, we apply theorem \ref{thm:neckstretchcompact} and get a limit 
holomorphic building. We focus henceforth on the unique component in the top floor 
$\cw\priv \cw'$ with a positive puncture. Since $L\subset \cw'$ and $\pi:\cw'\to L$ induces 
an isomorphism in homology, theorem \ref{thm:neckstretchcompact} ensures that this 
component is a map $v:S^2\priv Z\to \cw\priv \cw'$, where $Z \subset S^2$ is a finite
set of punctures. Here $v$ is asymptotic at the unique positive puncture $z^+\in Z$ 
to the curve $\tilde \gamma$ and at the negative punctures 
$z_i^- \in Z$ to curves $\ci(\tilde \gamma_i')$. The important point here is that $\tilde \gamma_i'$ 
are lifts of geodesics of $L'$ and that $\sum [\gamma_i']=\beta'$ 
(this is seen as usual by considering $\pi\circ v$ as a cobordism between $\pi(\tilde \gamma)=\gamma$ 
and $\sum \pi\circ \iota (\gamma_i')$). In particular, using lemma \ref{le:areasubb} the area 
of $v$ is given by $\ca(v)  = 
r\ell_{g}^{\min}(\beta) - \sum\limits_i \lambda(\ci(\tilde \gamma_i'))$. Moreover we have 
$$
\begin{array}{ll}
\lambda\big(\ci(\tilde \gamma_i')\big)& =[\lambda-\ci_*\lambda']\big(\ci(\tilde \gamma_i')\big)+\ci_*\lambda'\big(\ci(\tilde \gamma_i')\big)\\
& = [\ci^*\lambda-\lambda'](\tilde \gamma_i')+\lambda'(\tilde \gamma_i')\\
& =[\ci^*\lambda-\lambda'](\gamma_i')+r'\ell_{g'}(\gamma_i') \text{ \hspace{.5cm}(because $\ci^*\lambda-\lambda'$ is closed)}\\
& =\iota^*\lambda([ \gamma_i' ]) + r' \ell_{g'}(\gamma_i').
\end{array}
$$
Thus, 
$$
\begin{array}{ll}
0 < \ca(v)  = 
r\ell_{g}^{\min}(\beta) - \sum\limits_i \lambda(\ci(\tilde \gamma_i')) 
 & = r\ell_{g}^{\min}(\beta) - \sum\limits_i \left( \iota^*\lambda([ \gamma_i' ]) + r' \ell_{g'}(\gamma_i') \right) \\
 & = r \ell_g^{\min}(\beta) - \iota^*\lambda(\beta') - \sum\limits_i r'\ell_{g'}(\gamma_i').
 \end{array}
$$
This gives a bound on the total length of the multi-curve $\{\gamma_i'\}$ that depends 
only on $\beta$. Since this multi-curve represents a given homology class, we see that, 
provided $g'$ is chosen generically, it belongs to some finite set that depends only on $\beta$. 
Thus $v$ belongs to a set of maps whose asymptotics belong to a given finite set. Without 
loss of generality we assume that $\beta$ is a primitive class, so $v$ is somewhere injective. 
Moreover, we have total freedom in the choice of $J$ in some compact subset of $\cw\priv\overline{\cw'}$, 
which all possible $v$ must pass through. By standard transversality arguments we can therefore 
assume that the index of $v$ is non-negative. This index is given by formula \eqref{eq:index-Du} 
in the appendix,
$$
\mathrm{ind}(v) = n\chi(S^2\priv Z)+2c_1^\tau ( v^*T(T^*L) ) + \mu_\cz^\tau(\tilde \gamma ) 
- \sum_i\mu_\cz^\tau(\ci(\tilde\gamma_i')) +\# Z.
$$
The precise definitions of these quantities are recalled in the appendix. The symbol $\tau$ 
denotes an arbitrary choice of a Lagrangian distribution in the fiber bundle $u^*T(T^*L)$ 
and the different quantities are Chern classes and Maslov indices of various 
objects, computed with respect to this choice. Now the vertical fiber distributions provide natural
global Lagrangian distributions $\tau,\tau'$ in $T(T^*L)$ and $T(T^*L')$. One easily sees
that for this choice $c_1^\tau$ automatically vanishes. Moreover, the computation in the proof of
corollary \ref{cor:index} shows that $\mu_\cz^\tau(\tilde\gamma)=0$. This can also be seen 
by recalling that this quantity is the Morse index of the geodesic $\gamma$, which is minimal. 
It is clear that
$$
\mu_\cz^\tau(\ci(\tilde\gamma_i')) = \mu_\cz^{\tau'}(\tilde\gamma_i') + \mu_{\iota(L')}(\iota_*[\gamma_i']),
$$
because the last term of the equality represents the Maslov class of the loop 
$\ci_*\tau'_{|\im(\iota \circ  \gamma_i')}$ 
relative to $\tau$, which coincides with the Maslov class of $\ci_*\tau'_{|\im(\ci \circ \wdt \gamma_i')}$ relative to $\tau$.  
Also, $\mu_\cz^{\tau'}(\tilde\gamma_i')$ is non-negative as the Morse index of a geodesic. Putting 
$\# Z=:k+2$ we therefore get
\begin{align*}
0 \leq \mathrm{ind}(v) & = -nk - \sum_i \mu_\cz^{\tau'}(\tilde\gamma_i') - \sum_i \mu_{\iota(L')}(\iota_*[\gamma_i'])+k+2 \\
 & \leq 2-(n-1)k-\mu_{\iota(L')}(\iota_*\beta').
\end{align*}
Since $n\geq 1$, 
we get $\mu_{\iota(L')}(\iota_*\beta')\leq 2$. Replacing $\beta$ by $-\beta$ (which is still primitive) we see 
that $|\mu_{\iota(L')}(\iota_*\beta')|\leq 2$. In other terms, there is a constant bound for the Maslov class 
of any such Lagrangian embedding. But then an argument by Polterovich shows that this Maslov class 
vanishes \cite{polterovich}. Indeed,  consider a $3$-fold cover  $\psi:\hat L\to L$ and $\psi':\hat L'\to L'$ 
associated to the classes $\beta$ and $\beta'$. The map $\psi$ lifts to a symplectic covering 
$\Psi:T^*\hat L\to T^*L$, and since $(\pi\circ \iota)_*\beta'=\beta$, the embedding $\iota$ lifts to a 
Lagrangian embedding $\hat \iota:\hat L'\hra T^*\hat L$. By the previous analysis we see that
$$2\geq |\mu_{\hat \iota(\hat L')}(\hat \iota_*\hat \beta')|=3|\mu_{\iota(L')}(\iota_*\beta')|\in 3\N.$$
We therefore conclude the vanishing of this Maslov class. \cqfd


\subsection{Proof of proposition \ref{prop:nlctoriglag}}
Let $h:M\to M'$ be a symplectic homeomorphism that takes a Lagrangian submanifold $L$ to a 
 Lagrangian submanifold $L'$. We assume that conjecture \ref{conj:nl-} holds for $T^*L'$ 
(considering $h^{-1}$ instead of $h$, this amounts to assuming that it holds for $T^*L$). 
Let $f_n:M\to M'$ be a sequence of symplectic diffeomorphisms that approximate $h$ and put 
$L_n:=f_n(L)$. As explained in section \ref{sec:proofa}, we have 
$$
L' \subset f_n(\cw(L,g,\eps))\Subset \cw(L',g',\eps')\Subset \cw(L',g',1)\subset M'
$$ 
for any $\eps'\ll1$, $\eps$ small enough compared with $\eps'$ and $n$ large enough. We need to find a symplectic diffeomorphism 
of $M'$ that takes $L_n$ to $L'$. We proceed as follows. Denote by $\lambda'$ the natural 
Liouville form on $\cw(L',g',1)\subset M'$ and $\pi:\cw(L',g',1)\to L'$ the projection. 
For $n$ large enough
it induces an isomorphism $\pi_*:H_1(L_n)\to H_1(L')$.
The proof of theorem \ref{thm:closeclose}.a)
shows that the cohomology class $a_n:=[\lambda'_{|L_n}]\circ \pi_*^{-1}\in H^1(L')$ satisfies 
$\Vert a_n\Vert \leq C\eps'$. Moreover, the symplectic areas of the discs $D, h(D)$ and $f_n(D)$ with boundaries on $L, L'$ and $L_n$, respectively, coincide by theorem \ref{thm:rigspec}.a), so if $\delta:H_2(M',L')\to H_1(L')$ 
is the connecting morphism, $a_{n | \mathrm{Im}\delta}$ vanishes.  As a result, together with the exact sequence 
$H^1(M')\overset r\lra H^1(L')\overset {\delta^*}\lra H^2(M',L')$, we see that 
$a_n\in \im(r)$ and hence $a_n$ is the restriction to $H_1(L')$ of a cohomology 
class $A_n\in H^1(M')$. Let now $\theta_n$ be a closed $1$-form on $M'$ that represents 
$A_n$. Since  $\cw(L',g',1)$ retracts to $L'$, we can chose $\theta_n:=\pi^*\eta_n$ in $\cw(L',g',1)$
where $\eta_n$ is a closed $1$-form on $L'$ that represents the class $a_n$. Since $a_n$ is 
$\eps'$-small, we can choose $\eta_n$, and hence $\theta_n$ in $\cw(L',g',1)$ to be 
$\eps'$-small (in the uniform norm). Consider now the symplectic vector field on $M'$ defined by 
$$
\om'(X_n,\cdot)=-\theta_n,
$$ 
and its time $1$-map $\Phi_n$. Since $\theta_n$ is small on $\cw(L',g',1)$, we have 
$\Phi_n^t(\cw(L',g',\eps'))\subset \cw(L',g',1)$ for all $0 \leq t \leq 1$, so 
$\Phi_{n|\cw(L',g',\eps')}$ coincides with the map $(q,p)\mapsto (q,p-\eta_n(q))$. 
One thus verifies without difficulty that $[\lambda'_{|\Phi_n(L_n)}]$  vanishes. Thus, 
$\Phi_n(L_n)$ is an exact Lagrangian submanifold in $\cw(L',g',1)$ when considered in $T^*L'$. 
It is also easy to see that for $\eps_0$ chosen small enough but fixed (so that $\eps'\ll \eps_0$), 
$\Phi_n\circ f_n(\cw(L,g,\eps_0))$ is a Weinstein \nbd of $\Phi_n(L_n)$ that contains the zero section.
If, as we assume in the statement of our proposition, conjecture \ref{conj:nl-} holds for $T^*L'$, 
there exists a Hamiltonian diffeomorphism $\phi$ in $T^*L'$ which takes $\Phi_n(L_n)$ to $L'$. 
A classical argument even shows that our Hamiltonian isotopy can also be modified so as 
to be supported in $\cw(L',g',1)$. As a result $\phi$ can be seen as a Hamiltonian 
diffeomorphism of $M$ with support in $\cw(L',g',1)$, and the symplectomorphism 
$\phi\circ \Phi_n\circ f_n:M\to M'$ takes $L$ to $L'$. \cqfd


\section{Some embedding problems}\label{sec:applications}

\subsection{A symplectic order on Riemannian metrics}
The set of Riemannian metrics $\met(M)$ on any
manifold $M$ is endowed with a partial order defined by $g \leq g'$ if and
only if $g_x(v,v)\leq g_x'(v,v)$ for all $x\in M$ and all $v\in T_xM$. This
amounts to saying that for all $x\in M$, for all $p\in T_x^*M$,
$\Vert p\Vert_{g'}\leq \Vert p\Vert_g$, which in turn means that
$$
\cw_g\subset \cw_{g'} \hspace{,5cm}(\cw_g:=\cw(M,g, 1)=\{\,\Vert p\Vert_g < 1\,\}).
$$
In this perspective it is tempting to define a symplectic order on the
space of Riemannian metrics in the following way.
\begin{definition}
We say that $g\prec_\om g'$
if there exists a symplectic
embedding  $\Phi:\cw_g\hra \cw_{g'}$ with $\Phi(M)=M$ and
such that $\Phi_{|M}$ is homotopic
to the identity, where $M$ is identified with the zero section of the
cotangent bundle $T^*M$.
\end{definition}
This relation is only a preorder because any diffeomorphism  $f$ isotopic to
the identity  lifts to a symplectic diffeomorphism between
$\cw_g$ and $\cw_{f_*g}$  that preserves the zero section, so $f_*g \prec_\om g$ and
$g \prec_\om f_*g$ for all $g$. Thus $\prec_\om$ induces another
preorder $\prec_\om$ on $\met(M)_{/\diff^0(M)}$. Knowing whether this
new relation is now a partial order is related to subtle problems known
as rigidity of metrics (see e.g. \cite{crkl,becoga} or \cite{courte} for
a contact analogue). The persistence of punctured holomorphic discs asymptotic
to lifts of minimizing geodesics immediately implies a rigidity of this
preorder. To state it recall that if $\beta\in H_1(M)$ we have defined
$$
\ell_g^{\min}(\beta) = \min\{ \, \ell_g(\gamma) \;| \; \gamma\in \cc^1(S^1,M),\;
[\gamma]=\beta\in H_1(M) \, \}.
$$
\begin{theorem}\label{thm:symporderg}
Let $M$ be a closed manifold endowed with two Riemannian metrics $g$ and
$g'$ such that $g \prec_\om g'$. Then for all
$\beta\in H_1(M)$ we have
$$
\ell_g^{\min}(\beta)\leq \ell_{g'}^{\min}(\beta).
$$
\end{theorem}
\noindent{\it Proof:}
Let $g,g'$ be two Riemannian metrics on $M$ with $g\prec_\om g'$. Let
$\Phi:\cw_g\hra \cw_{g'}$ be a symplectic embedding such that $\Phi_{|M}$
is homotopic to the identity.
For any $\eps > 0$ the map $\Phi$ embeds $\cw_{g}$ compactly into $\cw_{(1+\eps)g'}$.
Let then $J_n
$ be a sequence of cylindrical
almost complex structures on $\cw_{(1+\eps)g'}$ which stretch the neck of
$\partial \Phi(\cw_g)$. By theorem \ref{thm:existence}, completed by
remark \ref{rk:existence}, there exists a $J_n$-holomorphic map
$u_n:(D\priv \{0\},\partial D)\to ({\cw_{(1+\eps)g'}}, M)$, asymptotic to a lift
$\tilde{\gamma}'$ of a geodesic $\gamma'$ in the class $\beta$ that
satisfies $\ell_{(1+\eps)g'}(\gamma')\leq (1+2\eps)\ell_{g'}^{\min}(\beta)$.
A brief calculation shows that $\ca(u_n) = \ell_{(1+\eps)g'}(\gamma')$.
Since $\Phi_{|M}$
is homotopic to the identity, the projection $\pi: \partial \Phi(\cw_g)
\to M$ induces an isomorphism of homotopy groups, hence in homology,
so we can apply theorem \ref{thm:neckstretchcompact} exactly as in the second
step of the proof of theorem \ref{thm:existence} (\S \ref{sec:proofexistence})
We get a limit holomorphic building $B$ that satisfies $\ca(B) =
\ell_{(1+\eps)g'}(\gamma')$. The holomorphic building contains a subbuilding $B'$
which is a disc with boundary on $M$ and with one positive puncture asymptotic
to a Reeb orbit $\tilde\gamma$ of $\partial\Phi(\cw_g)$,
whose projection to $M$ is a geodesic $\gamma$ that represents the class $\beta$.
Since $\Phi(M)=M$, we can apply lemma \ref{le:areasubb} in the exact setting and
the symplectic area of the subbuilding $B'$ is
\[
\ca(B') = \ell_g(\gamma).
\]
We have $\ca(B) \geq \ca(B')$ and thus
\[
0\leq \ell_{(1+\eps)g'}(\gamma')-\ell_g(\gamma)
\leq (1+2\eps)\ell_{g'}^{\min}(\beta)-\ell_g(\gamma).
\]
Thus, $\ell_g^{\min}(\beta)\leq \ell_g(\gamma)\leq (1+2\eps)\ell_{g'}^{\min}
(\beta)$. Since this estimate holds for all $\eps$ we obtain the desired
inequality. \cqfd

It might be worth mentioning that theorem \ref{thm:symporderg} can also be proved with 
symplectic homology, and might not be new at all (see \cite{absc,sawe} for the computation of the symplectic homology and spectral invariants of $\cw(L,g,r)$). This paragraph is only meant for showing that the holomorphic punctured discs provided by theorem \ref{thm:existence} allow to recover some of the quantitative invariants of $\cw(L,g,r)$ obtained by  spectral invariants associated to its symplectic homology.


\subsection{A Poisson bracket invariant}

In the previous section we saw that theorem \ref{thm:existence} allowed us to
study
relative embeddings of Weinstein \nbds of Lagrangian submanifolds
one into another
(relative means here that the Lagrangian submanifold has to be fixed by the
symplectic embedding).
It is also natural to consider symplectic embeddings of Weinstein \nbds of
the zero section into a general symplectic manifold and ask about the
maximal symplectic size of such Weinstein neighbourhoods.
In this general setting there are no punctured holomorphic discs of the ambient
manifolds that can be exploited (although this may happen in some particular cases).
The aim of this section is to define a monotone symplectic invariant of
a Lagrangian embedding and to compute it explicitly for the zero section in
certain Weinstein \nbds in the cotangent bundle.
This invariant is associated to a pair $(L,a)$ given by an embedded Lagrangian
submanifold $L$ and a primitive integral cohomology class $a \in H^1(L;\Z)$.
It is similar to the one defined in \cite{engame} and is based on the Poisson
bracket invariants of \cite{buenpo}.
By computing this invariant for a Lagrangian submanifold $L$ in a general
symplectic manifold $M$ and using monotonicity, our computations here will provide
information on the symplectic sizes of Weinstein neighborhoods of $L$ in $M$.
Estimating the size of Weinstein \nbds is a natural question and has been already
considered in the literature, see for instance \cite{zehmisch,cimo}.

We start with the definition of our Poisson bracket invariant and then show how
it relates to Lagrangian embeddings.
Let $L \subset (M, \omega)$ be a connected closed Lagrangian submanifold, where
$M$ is not necessarily compact. We associate to each non-zero primitive cohomology
class $a \in H^1(L;\Z)$ a Poisson bracket invariant in the following way. By the
de Rham isomorphism we can represent $a$ by a closed 1-form $\theta$ on $L$.
Now choose a base point $x_0 \in L$. Define a function $\Theta: L \to \R/\Z$ as
follows. For $x \in L$ set
\[
\Theta(x) := \int_{\gamma_x} \theta \, \mod 1,
\]
where $\gamma_x$ is any smooth path in $L$ from $x_0$ to $x$. Since a different
choice of path changes the integral by an integral value, this map is well defined.
Consider the four sets given by
\[
X_0 := \Theta^{-1}([0,\tfrac{1}{4}]),
\quad Y_{0} := \Theta^{-1}([\tfrac{1}{4}, \tfrac{1}{2}]),
\quad X_1 := \Theta^{-1}([\tfrac{1}{2}, \tfrac{3}{4}]),
\quad Y_{1} := \Theta^{-1}([\tfrac{3}{4}, 1]).
\]
Then we have $X_0 \cap X_1 = Y_0 \cap Y_1 = \emptyset$ and $L = X_0 \cup Y_0
\cup X_1 \cup Y_1$. Following the definition in \cite{buenpo}
we consider the set $\mathcal{F}(\theta,x_0)$ of all pairs $(H,K)$, where
$H, K \in C^{\infty}_c(M)$ such that $H|_{\op(X_0)} =0$, $H|_{\op(X_1)} = 1$,
$K|_{\op(Y_0)}= 0$ and $K|_{\op(Y_1)} = 1$. We then set
\[
pb^+(L, \theta,x_0) := \inf\limits_{\mathcal{F}(\theta,x_0)} \max\limits_M
\; \{ H,K \}\,\in\,[0,\infty).
\]
We now define (cf. \cite{engame})
\[
bp(L,a) := \frac{1}{\inf\, \{\, pb^+(L, \theta,x_0) \;|\; [\theta] = a,
\,x_0\in L \,\}} \,\in\,(0,\infty].
\]
We also write $bp(L, a, M)$ when we want to emphasize the ambient symplectic manifold.
When $a$ runs through the set of primitive classes of $H^1(L;\Z)$ the numbers
$bp(L, a)$ provide a set of quantitative invariants for Lagrangian submanifolds. More
precisely, if $\varphi: (M, \omega) \to (M',\omega')$ is a symplectomorphism that maps
a Lagrangian submanifold $L$ to $L' := \varphi(L)$, then $bp(L,\varphi^*a',M) = bp(L',a', M')$.
These invariants are obviously monotone: if two symplectic manifolds $M$ and $M'$
contain Lagrangian submanifolds $L$ and $L'$ and there exists a relative symplectic
embedding $f: (M,L) \hra (M',L')$, then
$$
bp(L,(f_{|L})^*a',M) \leq bp(L',a',M')
$$
for all primitive $a' \in H^1(L'; \Z)$. A version of these invariants was
introduced and computed in several examples in \cite{engame}. There a list of
properties of $bp$ is given and the definition is extended to include
non-primitive $a$, which is also possible in our case. In the rest of this
section we  compute these invariants for the zero section in
Weinstein neighborhoods in cotangent bundles. 
Before we state our result we recall the stable norm of a cohomology class together 
with an estimate of this number using closed curves.
\begin{definition}
For a Riemannian manifold $(L,g)$ we endow its cohomology group $H^1(L;\R)$ with the
stable norm
$$
\Vert a\Vert_\st := \inf\limits_{[\theta]=a} \, \max_{q\in L} \,\Vert \theta(q)\Vert_g.
$$
Then (\cite[Proposition 4.35]{gromov3}),
$$
\Vert a \Vert_\st= \sup\left\{\,\frac{a([\gamma])}{\ell_g(\gamma)} \;\Big|\;
\gamma\in \cc^1(S^1,L),\; a([\gamma]) \,> 0 \right\}.
$$
\end{definition}
Here is the main result of this section.
\begin{theorem}\label{thm:pb4L}
Let $(L,g)$ be a closed Riemannian manifold and $\cw(L,g,r) := \{\, \Vert p \Vert_g
< r\,\} \subset T^*L$. Let also $a \in H^1(L;\Z)$ be a primitive class. Then we
have
$$
bp(L,a,\cw(L,g,r)) = \frac{r}{\Vert a\Vert_\st} 
= r \,\inf \left\{\,\frac{\ell_g(\gamma)}{a([\gamma])} \;\Big|\;
\gamma\in \cc^1(S^1,L),\; a([\gamma]) \, > 0 \right\}.
$$
\end{theorem}

By monotonicity these numerical invariants provide bounds for the symplectic size 
of Weinstein neighborhoods of Lagrangian embeddings. To put it in a general framework, 
consider for a Lagrangian submanifold $L\subset (M,\om)$ with a given Riemannian metric 
$g$ on $L$ the quantity 
$$
c_{(M,L)}(L,g):=\sup\,\{\, r>0 \; | \; \big(\cw(L,g,r),L\big)\overset \om \hra (M,L) \,\}. 
$$ 
This is a relative version of the \emph{embedding capacity} $c^{(M,\om)}(L,g)$ defined 
and studied in \cite[p. 9]{cimo} (the embedding of $L$ for $c^{(M,\om)}(L,g)$
is not fixed). We have the inequality} $c_{(M,L)}(L,g)\leq c^{(M,\om)}(L,g)^{-1}$.

\begin{corollary}\label{cor:pb4L}
Let $L\subset M$ be a Lagrangian submanifold and $g$ a metric on $L$. If for $r > 0$
there is a relative symplectic embedding $(\cw(L,g,r), L) \hra (M,L)$, then for all
primitive classes $a\in H^1(L;\Z)\priv\{0\}$ we have
$$
r \;\leq\; bp(L,a,M)\cdot \Vert a\Vert_\st.
$$
\end{corollary}
\noindent In other terms,
\begin{equation}\label{eq:low-bound-emb-cap}
c_{(M, L)}(L,g)\leq \inf\limits_{a \in H^1(L) \backslash \{0\}}\;\;  bp(L, a, M) \cdot\Vert a\Vert_\st.
\end{equation}
This corollary relies on the lower bound for $bp(L,a, \cw(L,g,r))$ which, as will be
clear from the proof, only requires soft techniques. The hard part in using this corollary
therefore really lies in obtaining an upper bound for $bp(L,a,M)$. 

As an illustration of equation \eqref{eq:low-bound-emb-cap} we show that $c_{(\C \P^n, L_\cliff)}(\T^n,g) = \frac 1{\sqrt{n}(n+1)}
$ for the Clifford torus  $L_\cliff\subset (\C \P^n, \omega)$, where $\omega$ is the Fubini-Study symplectic form normalized 
such that the class $[\C \P^1]$ has area 1. We endow $\T^n = \R^n/{\Z^n}$ with the flat metric $g$ and consider 
the parametrization of $L_\cliff$ by $\T^n$ given by $(t_1,\dots,t_n)\mapsto [e^{2i\pi t_1}:\dots:e^{2i\pi t_n}:1]$.
The standard Hamiltonian $\T^n$-action on $\C \P^n$ gives us a moment map $\Phi: \C \P^n \to (\R^n)^*$ whose image 
is the simplex
\[
\Delta := \{\, (x_1, \ldots, x_n) \in (\R^n)^* \; | \; x_1, \ldots, x_n \geq 0,\; 0 \leq x_1 + \ldots + x_n \leq 1  \,\}, 
\]
and $L_{\cliff} = \Phi^{-1}(\mathbf{x})$ for $\mathbf{x}:=\big(\frac 1{n+1},\dots,\frac 1{n+1}\big)$. 
Note that we have a canonical isomorphism $(\Z^n)^*
\simeq H^1(\T^n)$  that we identify with $H^1(L_\cliff)$ via our specific embedding. 
Then, for the class $a := (1, \ldots, 1) \in H^1(\T^n)$ one can show that $bp(L_\cliff, a, \C\P^n) \leq 1/(n(n+1))$ by 
applying an adapted version of theorem 1.4 
together with the computation in theorem 2.16 of \cite{engame}.
We also easily see that for the flat metric we have $\Vert a\Vert_\st = \sqrt{n}$, so 
\[
c_{(\C \P^n, L_\cliff)}(\T^2,g)\leq \frac 1{\sqrt{n}(n+1)}.
\]

Now note that  $\Phi^{-1}(\mathrm{int}\,\Delta)$ is symplectomorphic to 
$\T^n \times \mathrm{int}\,\Delta \subset T^*\T^n$, and that there is an explicit symplectic embedding 
of $\cw(\T^n, g, r_{\max})$ into $\T^n\times \mathrm{int}\,\Delta$  relative to $\T^n\times\{\mathbf{x}\}$, with 
$r_{\max}:=\frac 1{\sqrt n(n+1)}=d_{\text{eucl}}(\mathbf{x},\partial \Delta)$. This provides a relative 
symplectic embedding of $\big(\cw(\T^n,g,r_{\max}),0_{\T^n}\big)$ into $(\C\P^n,L)$, thus we obtain
the lower bound $c_{(\C \P^n, L)}(\T^n,g) \geq  \frac1{\sqrt{n}(n+1))}$. 

\medskip
\noindent{\it Proof of theorem \ref{thm:pb4L}:} As is common with Poisson bracket
invariants (cf. \cite{buenpo, engame}), we obtain an upper bound from the persistence
of holomorphic punctured discs under a deformation of the almost complex structure
(theorem \ref{thm:existence}). For the lower bound we provide an explicit
construction. \newline\newline
\noindent\textbf{Step I: upper bound.} 
Consider first a curve $\gamma$ in a class
$\beta\in H_1(L)$ that satisfies $a(\beta) > 0$. We can freely perturb
$g$ in the $\cc^0$-topology for this proof, so we can assume that $g$ has a unique
minimal geodesic $\gamma(\beta)$ in the class $\beta$ and a unique $J_g$-holomorphic
map $u_{\gamma,g}$, as defined on page \pageref{page:defG}, which is Fredholm regular.
Fix a closed $1$-form $\theta$ that represents the class $a$ and a base point $x_0 \in L$.
This choice fixes the sets $X_0, Y_0, X_1, Y_1$ as above. Let $(H, K) \in \mathcal{F}
(\theta,x_0)$ be a pair of smooth functions supported in $\cw(L,g,r')$ for some
$r' < r$. Since $H, K$ are constant in neighborhoods of $X_i$ and $Y_i$ we have
$dH \wedge dK \equiv 0$ on $L$ and hence $H dK$ is a closed 1-form on $L$.
Following the calculation in \cite[Theorem 3.4]{engame} we see that
$[H dK|_L] = a \in H^1(L)$ and the exact form
$$
\om_s:=\om + s dH\wedge dK
$$
is symplectic for all $s \in I := [0, 1/ \max_M \{ H, K \})$. Note that $\omega_s$
coincides with the symplectic form $\omega = d\lambda$ near the boundary. Thus there
exists a generic smooth family of almost complex structures $\{J_s\}_{s \in I}$
in $\cw(L,g,r)$ compatible with $\om_s$, $g$-cylindrical at infinity and
starting at $J_0 = \Phi^*J_g$ (where $\Phi: \cw(L,g,r) \to T^*L$ is the map defined in 
section \ref{sec:framesplit}). Since there is a unique $J_0$-holomorphic punctured disc
asymptotic to $\tilde \gamma(\beta)$ and it is regular, there is a family $\{u_s\}$ of
$J_s$-holomorphic punctured discs asymptotic to $\tilde \gamma(\beta)$ for small
$s\geq 0$. The proof of theorem \ref{thm:existence} (section \ref{sec:proofexistence})
shows that this family persists for all $s\in I$. Indeed, the main ingredient is the compactness theorem \ref{thm:sft_compactness},
which only relies on the fact that $\tilde \gamma(\beta)$ has least action among the
lifts of the curves in class $\beta$ and this still holds for $\om_s$.
Thus, for all $s \in I$ there exists a $J_s$-holomorphic map $u_s:(D\priv \{0\},
\partial D) \to (\cw(L,g,r),L)$ which is asymptotic to $\tilde \gamma(\beta)$.
Then using Stokes' Theorem we have
$$
0 < \int u_s^*\om_s = \int u_s^* d\lambda + s\int u_s^*dH\wedge dK
= r\ell_g(\gamma(\beta)) - s a(\beta).
$$
Thus we have $s < \tfrac{r\ell_g(\gamma(\beta))}{a(\beta)} = \tfrac{r\ell_g^{\min}
(\beta)}{a(\beta)}$. This holds for all $s \in I$, hence
\[
\frac{1}{\max_M \{ H, K \}} \leq \frac{r\ell_g^{\min}(\beta)}{a(\beta)}.
\]
Varying $(H, K) \in \mathcal{F}(\theta,x_0)$ as well as $\theta$ in the class of $a$
and the choice of base point $x_0 \in L$, we see that
$$
bp(L, a, \cw(L,g,r)) \leq r\frac{\ell_g^{\min}(\beta)}{a(\beta)}
\leq r \frac {\ell_g(\gamma)}{a([\gamma])}.
$$

\noindent\textbf{Step II: lower bound.} It is enough to find good functions.
We first prove the weaker inequality
$$
bp(L,a,\cw(L,g,r))\geq \frac{r}{2 \Vert a \Vert_\st},
$$
because the proof is more visual. We then get rid of the constant $\tfrac{1}{2}$.
Figure \ref{fig:functions} represents the different functions that appear in the
proof.
\begin{figure}[h]
\begin{center}
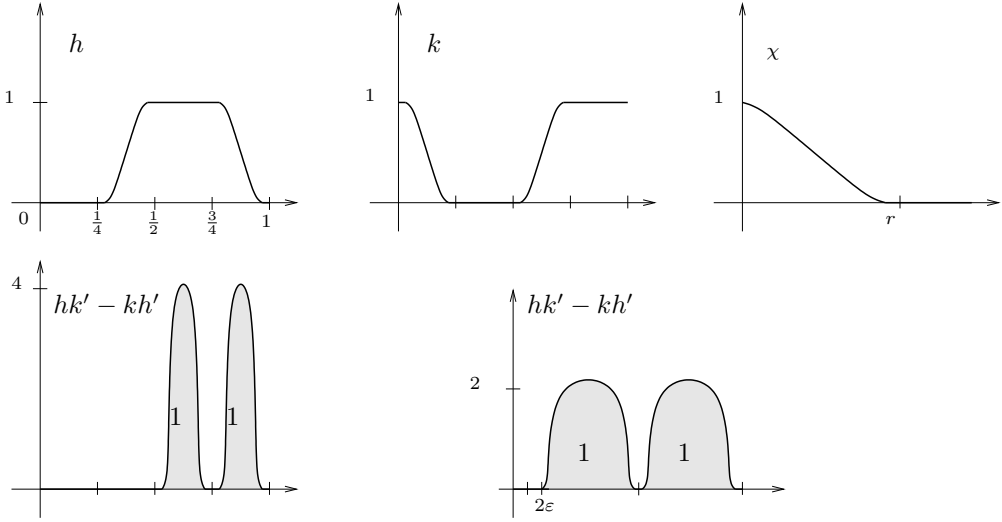
\caption{The almost optimal functions for the $bp$-invariant.}
\label{fig:functions}
\end{center}
\end{figure}

Let $\theta$ be a closed $1$-form that represents the class $a$ with
$\Vert \theta \Vert_g < \Vert a \Vert_\st + \eps$ and $\Theta$ its
$\R{/\Z}$-valued primitive as defined above. Let $h, k : \R{/\Z}\to [0,1]$
be smooth functions such that $h$ equals $0$ on $[0,\tfrac{1}{4}]$ and
$1$ on $[\tfrac{1}{2},\tfrac{3}{4}]$ and $k$ equals $0$ on
$[\tfrac{1}{4},\tfrac{1}{2}]$ and $1$ on $[\tfrac{3}{4},1]$. Let also $\chi:
[0,r) \to [0,1]$ be a smooth, compactly supported function that equals
$1$ near $0$. We define the compactly supported functions $H, K: \cw(L,g,r)
\to [0,1]$ as follows. For canonical coordinates $(q,p)\in T^*L$ we set
\begin{align*}
H(q,p) & := \chi(\Vert p\Vert_g)\, h\circ \Theta(q) \\
K(q,p) & := \chi(\Vert p\Vert_g)\, k\circ \Theta(q).
\end{align*}
Then
$$
\{ H, K\} = \chi(\Vert p\Vert_g)\, h \circ \Theta(q) \{ \chi(\Vert p\Vert_g),
k \circ \Theta(q) \} + \chi(\Vert p\Vert_g)\, k \circ \Theta(q)
\{h\circ \Theta(q),\chi(\Vert p\Vert_g)\}.
$$
Now for any function $f(q)$ we have
$$
\{ f(q), \chi(\Vert p\Vert_g)\} = \sum\limits_{i = 1}^{n} \frac {\partial f(q)}{\partial q_i}
\frac{\partial \chi(\Vert p\Vert_g)}{\partial p_i} =
\chi'(\Vert p\Vert_g) \sum\limits_{i = 1}^{n} \frac{\partial f(q)}{\partial q_i}
\frac{\partial \Vert p\Vert_g}{\partial p_i}.
$$
For a point $q_0 \in L$ we choose coordinates in a neighborhood such that the metric $g$
satisfies $g_{ij}(q_0) = \delta_{ij}$. Then at $q_0$ have $\tfrac{\partial \Vert p\Vert_g}{\partial p_i}
= \tfrac{p_i}{\Vert p\Vert_g}$. Applying this to our case we get at $q_0$
\begin{align*}
\{ H, K \}(q_0)  & = \chi(\Vert p\Vert_g) \chi'(\Vert p\Vert_g)(h' k - h k')\circ \Theta(q_0)
\sum\limits_{i = 1}^{n} \frac{\partial \Theta(q_0)}{\partial q_i}\frac{p_i}{\Vert p\Vert_g}
\end{align*}
Note that $\frac{\partial \Theta(q)}{\partial q_i} = \theta_q(\partial_{q_i})$
and $\sum_i \tfrac{p_i}{\Vert p \Vert_g}\partial_{q_i} := v$ has norm $\Vert v \Vert_g = 1$
for $p \neq 0$. By our choice of $\theta$ we have $|\theta_{q_0}(v)| < \| a \|_{\mathrm{st}}
+ \eps$. We therefore get
\begin{align*}
\{ H, K \}(q_0) & \leq \Big|\Big(\tfrac{\chi^2}{2}\Big)'\Big| \cdot|(h'k - h \, k')\circ  \Theta(q_0)|
 \cdot \,(\Vert a\Vert_\st+\eps).
\end{align*}
Now notice that because of the properties of $h,k$ (see also figure \ref{fig:functions}),
$$
h'k - h \,k' =
\left\{
\begin{array}{ll}
h' & \text{ on } [\frac 34, 1],\\
-k' & \text{ on } [\frac 12,\frac 34], \\
0 & \text{ else. } \\
\end{array}\right.
$$
Taking into account that the only constraint on $h', k'$ is that their integral equals $\pm 1$
on $[\frac 34,1]$ and $[\frac 12,\frac 34]$, respectively, we see that for
a good choice of $h, k$ we have $\Vert h'k - h \,k' \Vert_{\cc^0} \leq 4+\eps$. Also, since
the only constraint on $\chi$ is that $\chi$ varies from $1$ to $0$ within the interval
$[0, r)$, we see that for a good choice of $\chi$ (for which $\chi^2$ is almost linear) we
can ensure that $\Vert \big(\nf {\chi^2}2\big)' \Vert_{\cc^0}\leq \frac 1{2r} + \eps$.
Altogether we get
\begin{equation}\label{eq:pblowbound}
\{H,K\}(q_0) \leq \left(\frac 1{2r} + \eps\right)(4+\eps)\,(\Vert a\Vert_\st + \eps)
\leq \frac{2 \,\Vert a\Vert_\st}{r} + C\eps
\end{equation}
for $\eps\ll 1$ and a constant $C$. The choice of $q_0 \in L$ was arbitrary and
$\eps > 0$ can be chosen arbitrarily small. Thus,
$$
bp(L,a,\cw(L,g,r))\geq \frac{r}{2 \,\Vert a\Vert_\st}.
$$

Finally, in order to get rid of the constant $\frac{1}{2}$ we use an equivalent
definition of $bp(L, a, M)$. Namely for the function $\Theta$ one can choose the four
sets
\[
X_0' := \Theta^{-1}([0,\eps]),\quad Y_0' := \Theta^{-1}([\eps,2\eps]),
\quad X_1' := \Theta^{-1}([2\eps, \tfrac{1}{2}]),\quad Y_1' := \Theta^{-1}([\tfrac{1}{2},1]),
\]
and choose the functions $H, K \in C^{\infty}_c(M)$ to satisfy $H_{|\op(X_0')} = 0$,
$H_{|\op(X_1')} = 1$, $K_{|\op(Y_0')} = 0$ and $K_{|\op(Y_1')} = 1$ in the definition of
$bp(L, a, M)$. Indeed, if $\rho: \R/\Z \to \R/\Z$ is a diffeomorphism that equals
the identity near $0$ and sends the intervals
$[0,\eps]$, $[\eps,2\eps]$, $[2\eps,\tfrac{1}{2}]$ and $[\tfrac{1}{2},1]$ to
$[0,\tfrac{1}{4}]$, $[\tfrac{1}{4},\tfrac{1}{2}]$, $[\tfrac{1}{2},\tfrac{3}{4}]$
and $[\tfrac{3}{4},1]$, then
as in \cite[Section 3.2]{engame} one can show that $\rho \circ \Theta =: \Theta'
= \int\theta'$ for a closed 1-from $\theta'$ representing the class $a$. Thus pairs
of functions in $\cf(\theta', x_0)$ are equivalent to pairs $(H,K)$ with the
corresponding values on $X_i'$ and $Y_j'$.
Since for $bp(L, a, M)$ the infimum is taken over all $[\theta] = a$, we see that
this definition is equivalent to the original one.

For this extremalization problem the result is slightly better, since now $h'k - h k'$
is supported on $[2\eps,\tfrac{1}{2}] \cup [\tfrac{1}{2}, 1]$ and we can choose
the slopes of $h$ and $k$ to be close to $\pm \tfrac{1}{2}$ on these intervals.
For a suitable choice of $h$ and $k$ we therefore have $\Vert h'k - h k'\Vert_{\cc^0}
\leq 2+\eps$, and the inequality \eqref{eq:pblowbound} becomes
\begin{equation}
\{H,K \} \leq \left( \frac 1{2r} + \eps \right)(2+\eps)(\Vert a\Vert_\st + \eps)
\leq \frac{\Vert a\Vert_\st}{r} + C\eps
\end{equation}
Since $\eps > 0$ can be chosen arbitrarily small, we indeed conclude that
$bp(L,a,\cw(L,g,r))\geq r \Vert a\Vert_{\st}^{-1}$.
\cqfd


\appendix
\section{Transversality for punctured holomorphic disks}\label{sec:fredholm}
In this section we prove generic transversality for moduli spaces of
punctured pseudoholomorphic curves with boundary in symplectic cobordisms.
This result is essentially folklore and the statement follows mostly from
standard results in the literature where parts have already been done
(see \cite{dragnev,mcsa,wendl,howize2}). Nevertheless, as far as we are aware,
transversality has not been explicitly written for punctured surfaces with totally real
boundary conditions in symplectic cobordisms of arbitrary dimension, so we provide a
proof for the sake of completeness. This is mostly a matter of compiling the different
sources cited above. For the case without boundary we refer the reader to the excellent
exposition by Wendl \cite{wendl2}.

We fix notation. Let $(X, \omega)$ be a symplectic cobordism. Thus $(X,
\omega)$ is a symplectic manifold containing a compact domain $K \subset
X$ with contact boundary such that $(X \backslash \mathrm{int}(K), \omega)$
is symplectomorphic to the union $([0,\infty) \times M_+, d
(e^r \alpha_+)) \sqcup ((-\infty, 0] \times M_-, d (e^r \alpha_-))$
for suitable contact manifolds $(M_+, \alpha_+)$ and $(M_-, \alpha_-)$.
Here $r$ denotes the $\mathbb{R}$-coordinate. The contact manifolds
$M_{\pm}$ may contain several components and one, but not both, may be
empty. Let $L \subset (X, \omega)$ be a compact Lagrangian submanifold
contained in $\mathrm{int}(K)$. We fix a Riemannian metric $g$
that is translation-invariant outside a compact set containing $K$ and
such that $L$ is totally geodesic.
Let $\gamma$ be a non-degenerate Reeb orbit of a contact manifold $(M, \alpha)$
with period $T > 0$. Define $Z_+ := [0, \infty) \times S^1$ and $Z_- := (-\infty,0]
\times S^1$ with the coordinates $(s,t)$ and conformal structure $j \partial_s
= \partial_t$. We say that a map $v = (a, u): Z_+ \to \R \times M$ is
positively asymptotic to $\gamma$, if $\lim_{s \to \infty} a(s,t) = \infty$
and $\lim_{s \to \infty} u(s,t) = \gamma( T t)$. We say a map $v' = (a', u')
: Z_- \to \R \times M$ is negatively asymptotic to $\gamma$, if
$\lim_{s \to -\infty} a'(s,t) = -\infty$ and $\lim_{s \to -\infty} u'(s,t)
= \gamma(T t)$.
Here we assume uniform convergence of the limits. For a fixed Riemann surface
$(\Sigma, j)$ (possibly with boundary) let $\Gamma = \{ z_1, \ldots, z_{l+m} \}
\subset \mathrm{int}(\Sigma)$ be a set of punctures. Given a family of non-degenerate
Reeb orbits $\mathcal{O}_{\Gamma} = \{ \gamma_1^+, \ldots, \gamma_l^+,
\gamma_1^-, \ldots, \gamma_m^- \}$ in $M_{\pm}$, we say that a map
$\phi: \Sigma \backslash \Gamma \to X$ is asymptotic to the family
$\mathcal{O}_{\Gamma}$, if for all $i$ there exists a disk neighborhood
$\mathcal{U}_{i}$ of $z_i$ in $\Sigma$, a biholomorphism $\psi_i: Z_\pm
\to \mathcal{U}_{i}\backslash \{ z_i \}$ such that $\phi\circ \psi_i :
Z_\pm\to X$ is asymptotic to $\gamma_i^{\pm}$.

Let us fix an almost complex structure $J_\cyl$ on $X$ that is
compatible with $\alpha_\pm$ outside a compact \nbd of $K$. Denote by
$\cj^\infty_\cyl(J_\cyl)$ the space of smooth almost complex structures
on $X$ that are compatible with $\om$ and coincide with $J_\cyl$ outside
this fixed compact \nbd of $K$. We endow this space with the
$C^{\infty}$-topology. Then $\cj^\infty_\cyl(J_\cyl)$ is a contractible
space. In order to lighten notation we omit the $J_\cyl$ and only write
$\cj^\infty_\cyl$. In this appendix, this notation will {\it always} refer
to $\cj_\cyl^\infty(J_\cyl)$. For a chosen $J \in \cj^\infty_\cyl$
and family $\mathcal{O}_{\Gamma}$ we define the moduli space of
punctured Riemann surfaces with boundary on $L$,
\begin{equation*}
\mathcal{M}(\mathcal{O}_{\Gamma}, J) := \left\{ u: (\Sigma \backslash
\Gamma, \partial \Sigma; j) \longrightarrow (X, L; J) \;\; \bigg| \;\;
 \begin{array}{c}
      du + J \circ du \circ j = 0, \; u(\partial \Sigma) \subset L,  \\
      \quad u  \text{ is asymptotic to } \mathcal{O}_{\Gamma}
    \end{array}
\right\}.
\end{equation*}
We also consider the subset $\mathcal{M}^*(\mathcal{O}_{\Gamma}, J)$
consisting of curves in $\mathcal{M}(\mathcal{O}_{\Gamma}, J)$ that
are somewhere injective in $\mathrm{int}(K)$, i.e. elements $u$ for
which there exists a $z \in \mathrm{int}(\Sigma \backslash \Gamma)$
such that $du(z) \neq 0$, $u(z) \in \mathrm{int}(K)$ and $u^{-1}(u(z))
= \{ z \}$. The main result of this section is as follows.
\begin{theorem}\label{thm:transversality_disks}
There exists a subset $\mathcal{J}^{\mathrm{reg}} \subset
\cj^\infty_\cyl$ such that the following holds:
\begin{enumerate}
\item If $J \in \mathcal{J}^{\mathrm{reg}}$, then $\mathcal{M}^*(
\mathcal{O}_{\Gamma}, J)$ is a smooth manifold of dimension
\begin{align}\label{eq:dimension_punctured_surfaces}
\mathrm{dim}(\mathcal{M}^*(\mathcal{O}_{\Gamma}, J)) & =
n\chi(\dot{\Sigma}) + 2c_1^{\tau} (u^*TX) + \mu^{\tau}(u^*TX, u^*TL) \nonumber \\
& \quad + \sum\limits_{i = 1}^l \mu_{\mathrm{CZ}}^{\tau}(\gamma_i^+)
- \sum\limits_{j = 1}^m \mu_{\mathrm{CZ}}^{\tau}(\gamma_j^-) + \# \Gamma,
\end{align}
locally around the element $u \in \mathcal{M}^*(\mathcal{O}_{\Gamma}, J)$.
Here $\tau$ denotes a trivialization of $u^*TX$ on the cylindrical ends
and boundary.
\item The subset $\mathcal{J}^{\mathrm{reg}}$ is of second category in
$\cj^\infty_\cyl$.
\end{enumerate}
\end{theorem}
We refer the reader to the end of section \ref{sec:crtypeop} for the
definitions of the topological invariants $c_1^\tau$, $\mu^\tau$ and
$\mu_\textrm{CZ}^\tau$. We call the elements $J \in \mathcal{J}^{\mathrm{reg}}$
\emph{regular}. For our applications we need to discuss the
dependence of $\mathcal{M}^*(\mathcal{O}_{\Gamma}, J)$ under variations
of $J \in \mathcal{J}^{\mathrm{reg}}$. For a smooth path $\{J_t\}_{t
\in [0,1]} \subset \cj^\infty_\cyl$ define
\[
\mathcal{W}^*(\mathcal{O}_{\Gamma}, \{J_t \} ) = \{\, (t, u)
\ \big| \ t \in [0,1], u \in \mathcal{M}^*( \mathcal{O}_{\Gamma},
J_t) \,\}.
\]
For two regular $J_0, J_1 \in \mathcal{J}^{\mathrm{reg}}$ we denote by
$\mathcal{J}(J_0, J_1)$ the space of all smooth paths in $\cj^\infty_\cyl$
connecting $J_0$ to $J_1$.

\begin{theorem}\label{thm:homotopy_transversality_disks}
There exists a subset $\mathcal{J}^{\mathrm{reg}}(J_0, J_1)
\subset \mathcal{J}(J_0, J_1)$ such that the following holds:
\begin{enumerate}
\item If $\{J_t \}_{t \in [0,1]} \in \mathcal{J}^{\mathrm{reg}
}(J_0, J_1)$, then $\mathcal{W}^*(\mathcal{O}_{\Gamma},
\{J_t \} )$ is a smooth oriented manifold with boundary
\[
\partial \mathcal{W}^*(\mathcal{O}_{\Gamma}, \{J_t \} ) =
\mathcal{M}^*( \mathcal{O}_{\Gamma}, J_0) \cup
\mathcal{M}^*( \mathcal{O}_{\Gamma}, J_1).
\]
\item The set $\mathcal{J}^{\mathrm{reg}}(J_0, J_1)$ is of
second category in $\mathcal{J}(J_0, J_1)$.
\end{enumerate}
\end{theorem}

The proof of theorem~\ref{thm:transversality_disks} follows the standard
line of arguments in the literature, see e.g. \cite{mcsa}. We first recall
various notions from the general theory of Cauchy-Riemann type operators
on punctured Riemann surfaces. In the functional analytical setup we explain
the Banach manifolds and bundles involved in our setting. We then introduce
the universal moduli space as a Banach submanifold of the aforementioned
Banach manifolds and use this construction to prove
theorem~\ref{thm:transversality_disks}. Adaptations of these arguments
then provide a proof of theorem~\ref{thm:homotopy_transversality_disks}


\subsection{CR type operators on Hermitian bundles over punctured surfaces}\label{sec:crtypeop}

Let $(\Sigma, j)$ be a compact Riemann surface of genus $g$ with
$m \geq 0$ boundary components. Choose a non-empty finite set $\Gamma
\subset \mathrm{int}(\Sigma)$ of positive and negative interior punctures,
we write $\Gamma = \Gamma^+ \cup \Gamma^-$. The punctured surface is then
denoted $\dot{\Sigma} = \Sigma \backslash \Gamma$. Now for every
puncture $z \in \Gamma^{\pm}$ we choose a closed neighborhood $\mathcal{U}_z
\subset \Sigma$ of $z$ together with a biholomorphic map $\varphi_z:
(\dot{\mathcal{U}}_z, j) \to (Z_{\pm}, i)$, where $\dot{\mathcal{U}}_z :=
\mathcal{U}_z\backslash \{z \}$ is the punctured neighborhood and $Z_+ =
[0, \infty)\times S^1$, $Z_- = (-\infty, 0]\times S^1$ are complex
cylinders. Note that the assumption that $\phi_z$ is a biholomorphism
implies that $\phi_z(w)\to \pm \infty$ when $w\to z$.
The union of punctured neighborhoods $\dot{\mathcal{U}}_z$ will be
called the cylindrical ends of $\dot{\Sigma}$.

Let $(E, \omega, J) \to (\dot{\Sigma}, j)$ be a smooth Hermitian vector
bundle of rank $n$ over the punctured surface. By Hermitian structure
we mean that $(E, \omega)$ is a symplectic vector bundle and $J$ is an
$\omega$-compatible almost complex structure. The inner product is then
given by
\[
\langle \cdot, \cdot \rangle_E := \omega(\cdot, J \cdot) + i \omega( \cdot,
\cdot).
\]
We call a trivialization $\Phi$ of $E$ near $z \in \Gamma^{\pm}$
\emph{admissible}, if $\Phi: E|_{\dot{\mathcal{U}_z}} \to Z_{\pm} \times
\mathbb{C}^n$ is a unitary bundle isomorphism which projects to the
biholomorphism $\varphi_z$. Here $\mathbb{C}^n$ is identified with the
standard Hermitian vector space. Note that every Hermitian vector bundle
$E$ over $\dot{\Sigma}$ has admissible trivializations over the
cylindrical ends. Furthermore, there exist admissible trivializations
that extend over $\dot{\Sigma}$, since $\Gamma \neq \emptyset$ and
thus $\dot{\Sigma}$ has the homotopy type of a 1-dimensional
cellular complex. For a choice of admissible trivialization $\Phi$
near all $z\in \Gamma$ we fix a volume form $dvol$ on $\dot{\Sigma}$
such that $dvol$ is equal to $ds \wedge dt$ in the cylindrical
coordinates induced by $\Phi$.

Let $(E', \omega', J')$ and $(E, \omega, J)$ be two Hermitian vector
bundles over $\dot{\Sigma}$. We denote by $\mathrm{Hom}^{1,0}(E', E)$ and
$\mathrm{Hom}^{0,1}(E', E)$ the corresponding complex vector bundles
consisting of complex linear and antilinear bundle maps $E' \to E$.
Sometimes we include the complex structures $J$, $J'$ in the notation
to provide more clarity. In the following we will often consider the bundle
$F := \mathrm{Hom}^{0,1}(T\dot{\Sigma}, E)$ (and sections thereof). In
particular this bundle inherits a Hermitian structure
$(F,\overline{\omega}, \overline{J})$ by setting
\[
\overline{J}\eta = J\circ \eta, \qquad \overline{\omega}_p(\eta, \eta')
= \frac{\omega_p(\eta(v), \eta'(v))}{dvol(v, jv)},
\]
for $p \in \dot{\Sigma}$, $\eta, \eta' \in F$ and any non-zero choice
of $v \in T_p\dot{\Sigma}$.
Furthermore, any admissible trivialization $\Phi$ of $(E, \omega, J)$ induces an
admissible trivialization of $(F, \overline{\omega}, \overline{J})$ via
\[
F|_{\dot{\mathcal{U}}_z} \to
Z_{\pm} \times \mathbb{C}^n, \qquad \eta \longmapsto \Phi(\eta(
\partial_s)),
\]
where $\partial_s$ is the vector field on $T\dot{\Sigma}|_{\dot{
\mathcal{U}}_z}$ arising from $\phi_z$. On $E$ and $F$ we define the
$L^2$-inner product to be
\begin{align*}
\langle \eta, \xi \rangle_{L^2(E)} &:= \int\limits_{\dot{\Sigma}}
\omega( \eta, J \xi ) \, {dvol} \quad \text{ for }  \eta, \xi \in C^{\infty}_0(E), \\
\langle \nu, \rho \rangle_{L^2(F)} &:= \int\limits_{\dot{\Sigma}}
\overline{\omega}( \nu, \overline{J} \rho ) \, {dvol} \quad \text{ for }
 \nu, \rho \in C^{\infty}_0(F).
\end{align*}

We recall the theory of Cauchy-Riemann type operators on Hermitian
bundles over punctured surfaces. We begin with the notion of an
asymptotic operator.

Let $(S^1 \times \mathbb{R}^{2n}, \omega_0, J_0) \to S^1$ be the
standard Hermitian vector bundle of complex rank $n$. An
\emph{asymptotic operator} $\mathbf{A}$ on $S^1 \times \mathbb{R}^{2n}$ is
any real linear differential operator
\[
\mathbf{A}:  C^{\infty}(S^1, \mathbb{R}^{2n}) \longrightarrow
C^{\infty}(S^1, \mathbb{R}^{2n}), \quad \eta \longmapsto
- J_0 \partial_t \eta - S(t)\eta,
\]
where $S: S^1 \to \mathrm{End}(\mathbb{R}^{2n})$ is {any}
smooth loop of symmetric matrices. Equivalently, an asymptotic
operator $\mathbf{A}$
is any operator of the form $-J_0 \nabla_t$, where $\nabla$ is {a}
symplectic connection on $(S^1 \times \mathbb{R}^{2n}, \omega_0)$.
For the real $L^2$-bundle metric given by
\[
\langle \eta, \eta' \rangle_{L^2} := \int\limits_{S^1} \omega_0
(\eta(t), J_0 \eta'(t)) dt,
\]
the operator $\mathbf{A}$ is symmetric. If we consider
$\mathbf{A}$ as a bounded
linear operator $H^1(S^1, \mathbb{R}^{2n}) \to L^2(S^1,
\mathbb{R}^{2n})$, then $\mathbf{A}$ is Fredholm with
index equal to $0$ (see e.g. \cite{wendl2}) 
We say that $\mathbf{A}$ is \emph{non-degenerate}, if
its spectrum $\sigma(\mathbf{A})$ does not contain 0. In this case
the operator $\mathbf{A}: H^1(S^1, \mathbb{R}^{2n}) \to L^2(S^1,
\mathbb{R}^{2n})$ induces an isomorphism.

We recall the standard $\overline{\partial}$ operator for smooth
functions on $(\dot{\Sigma},j)$,
\[
\overline{\partial}: C^{\infty}(\dot{\Sigma}, \mathbb{C})
\longrightarrow \Omega^{0,1}(\dot{\Sigma}, \mathbb{C}), \qquad
f \longmapsto df + i \circ df \circ j.
\]

\begin{definition}
Let $(E, \omega, J)$ be a smooth Hermitian vector bundle over $(\dot{
\Sigma},j)$. A real linear first order differential operator
$\mathbf{D}: \Gamma(E) \to \Gamma(\mathrm{Hom}^{0,1}(T\dot{\Sigma}, E))$
is called of \emph{Cauchy-Riemann type}, if it satisfies
\[
\mathbf{D}(f\eta) = \overline{\partial}(f) \eta + f \mathbf{D}\eta
\]
for every $\eta \in \Gamma(E)$ and $f \in C^{\infty}(\dot{\Sigma},
\mathbb{C})$. For a puncture $z \in \Gamma$ we say that $\mathbf{D}$ is
asymptotic to an asymptotic operator $\mathbf{A}_z$ at $z$, if in
an admissible trivialization near $z$ the operator $\mathbf{D}$
takes the form
\[
\mathbf{D}\eta(s,t){\tfrac{\partial}{\partial s}}
= \partial_s \eta(s,t) + J_0 \partial_t \eta(s,t) + S(s, t) \eta(s,t),
\]
where $S(s,t)$ is a smooth family in $\mathrm{End}_{\mathbb{R}}
(\mathbb{C}^n)$ such that $S(s,t)$ converges uniformly as $s \to
\pm \infty$ to a smooth loop of symmetric matrices $S_z(t)$, where
$-J_0 \partial_t - S_z(t)$ is the expression of $\mathbf{A}_z$
in the trivialization.
\end{definition}

Let $k \geq 1$ and $p > 2$. For a smooth Hermitian vector bundle
$(E, \omega, J) \to (\dot{\Sigma},j)$ with fixed admissible trivialization
at the cylindrical ends we consider the topological vector space
$W^{k,p}_{\mathrm{loc}}(E)$ of $W^{k,p}_{\mathrm{loc}}$-sections
of $E$. We define the Banach space
\[
W^{k,p}(E) := \{ \eta \in W^{k,p}_{\mathrm{loc}}(E)\;|\; \eta_z \in
W^{k,p}(\mathrm{int}(Z_{\pm}), \mathbb{C}^n) \;\, \forall\, z \in \Gamma^{
\pm} \},
\]
where $\eta_z$ is $\eta|_{\dot{\mathcal{U}}_z}$ in coordinates induced
by the admissible trivialization and the area form $dvol$ is
used to define the $W^{k,p}$-norm of $\eta_z$. By choosing a compact
set $C$ that contains $\dot{\Sigma} \backslash \bigcup \dot{\mathcal{U}}_z$
we obtain a Banach space norm by summing the corresponding $W^{k,p}$-norms
over $C$ and the cylindrical ends. {Norms} that arise in this way are
equivalent. Now for a smooth totally real subbundle
$\Lambda \subset E|_{\partial\Sigma}$ we consider the Banach subspace
\[
W^{k,p}_{\Lambda}(E) := \{ \eta \in W^{k,p}(E) \;|\; \eta(\partial
\Sigma) \subset \Lambda \}.
\]
We recall the following statement on Cauchy-Riemann type operators, see
\cite{schwarz-thesis, wendl}.

\begin{theorem}\label{thm:cauchy-riemann_op_fredholm}
Let $\mathbf{A}_z$ be asymptotic operators for each $z \in \Gamma$ and let
\[
\mathbf{D}: W^{k,p}_{\Lambda}(E) \longrightarrow W^{k-1,p}(
\mathrm{Hom}^{0,1}(T\dot{\Sigma},E))
\]
be a Cauchy-Riemann type operator asymptotic to $\mathbf{A}_z$ for each $z$.
Then $\mathbf{D}$ is a Fredholm operator if all $\mathbf{A}_z$ are
non-degenerate. Moreover, $\mathrm{ind}(\mathbf{D})$ and $\ker(\mathbf{D})$
are independent of $k$ and $p$.
\end{theorem}
In order to compute the index of the Fredholm operator $\mathbf{D}$
we briefly recall certain topological invariants.

Let $(E, J) \to (\Sigma, j)$ be a complex vector bundle over a compact
Riemann surface $\Sigma$ with boundary $\partial \Sigma$. Let $\tau
: E|_{\partial \Sigma} \to \partial \Sigma \times \mathbb{C}^n$ be a
trivialization over the boundary. We define the \emph{relative first
Chern class} with respect to the trivialization $\tau$ as follows.
If $(E,J)$ is a line bundle, then $c_1^{\tau}(E)$ counts the zeros
(with signs) of a generic smooth section that is non-zero and constant
over $\partial \Sigma$ with respect to $\tau$. For two complex vector
bundles $(E_1, J_1)$ and $(E_2, J_2)$ with trivializations $\tau_1$ and
$\tau_2$ over $\partial \Sigma$ we set $c_1^{\tau_1 \oplus \tau_2}(E_1
\oplus E_2) = c_1^{\tau_1}(E_1) + c_1^{\tau_2}(E_2)$. Since every
complex vector bundle over a Riemann surface splits into a sum of line bundles this uniquely
determines $c_1^{\tau}(E)$ for all complex vector bundles.
The definition then extends to complex vector bundles over punctured Riemann surfaces
where $\tau$ equals a chosen admissible trivialization on the cylindrical
ends. For a totally real subbundle $\Lambda
\subset E|_{\partial \Sigma}$ the trivialization $\tau$ also defines a
\emph{Maslov index} $\mu^{\tau}(E, \Lambda) \in \mathbb{Z}$, see
\cite{mcsa}.

Now let $\mathbf{A}_z$ be an asymptotic operator for $z \in \Gamma$.
In the admissible trivialization we have $\mathbf{A}_z = - J_0
\partial_t - S(t)$. Let $\Psi(t) \in \mathrm{Sp}(n)$ be a smooth loop
of symplectic matrices satisfying the differential equation $\dot{\Psi}
(t) = J_0 S(t) \Psi(t)$. If $\mathbf{A}_z$ is non-degenerate, then
$\Psi(1)$ does not have 1 as an eigenvalue and we can consider its
Conley-Zehnder index $\mu_{\mathrm{CZ}}^{\tau}(\Psi(t))$, see
\cite{salamon}. We then set the Conley-Zehnder index of
$\mathbf{A}_z$ to be $\mu_{\mathrm{CZ}}^{\tau}(\mathbf{A}_z) := \mu_{
\mathrm{CZ}}^{\tau}(\Psi(t))$.

We recall the following statement from \cite{wendl}.

\begin{theorem}\label{thm:cauchy-riemann_op_index}
Let $\mathbf{A}_z$ be non-degenerate asymptotic operators for each
$z \in \Gamma$ and let
\[
\mathbf{D}: W^{k,p}_{\Lambda}(E) \longrightarrow W^{k-1,p}(
\mathrm{Hom}^{0,1}(T\dot{\Sigma},E))
\]
be a Cauchy-Riemann type operator asymptotic to $\mathbf{A}_z$ for each $z$.
Then the Fredholm index of $\mathbf{D}$ is
\begin{equation}\label{eq:index-general-D}
\mathrm{ind}(\mathbf{D}) = n \chi(\dot{\Sigma}) + 2 c_1^{\tau}(E) +
\mu^{\tau}(E, \Lambda) + \sum\limits_{z \in \Gamma^+}
\mu_{\mathrm{CZ}}^{\tau}(\mathbf{A}_z) - \sum\limits_{z \in \Gamma^-}
\mu_{\mathrm{CZ}}^{\tau}(\mathbf{A}_z),
\end{equation}
where $n = \mathrm{rank}_{\mathbb{C}}(E)$.
\end{theorem}

Note that the dependence of equation~\eqref{eq:index-general-D}
on the choice of trivialization $\tau$ cancels out.
In subsequent sections we will consider Banach spaces of sections
of $E$ with exponential weights at the punctures. Let $(E, \omega,
J) \to (\dot{\Sigma}, j)$ be a Hermitian vector bundle over a
punctured Riemann surface equipped with an admissible trivialisation.
For $k \in \mathbb{N}, p > 2$ and $\delta \in \mathbb{R}$ we define the
Banach space
\[
W^{k,p,\delta}_{\Lambda}(E) \subset W^{k,p}_{\mathrm{loc}}(E)
\]
to be the space of sections $\eta \in W^{k,p}_{\mathrm{loc}}(E)$
whose restriction to the neighborhood of a puncture $z \in \Gamma^{\pm}$,
$\eta_{z}: Z_{\pm} \to \C^n$, satisfies
\[
\| e^{\pm \delta s}\eta_{z} \|_{W^{k,p}(Z_{\pm})} < \infty,
\]
and $\eta(\partial \Sigma) \subset \Lambda$.
Thus for $\delta = 0$ we recover the Banach space $W^{k,p}_{\Lambda}(E)$
from before and for $\delta > 0$ the sections in $W^{k,p,\delta}_{\Lambda}
(E)$ are guaranteed to have exponential decay at infinity. Note that
for $k = 0$, the dual space $(L^{p, \delta}(E))^*$ is isomorphic to
$L^{q, -\delta}(E)$ for $1/p + 1/q = 1$ via the $L^2$-inner product.


\subsection{Functional analytic setup}\label{app:func_an_setup}

We describe the Banach spaces used in the proof of
Theorems~\ref{thm:transversality_disks} and~\ref{thm:homotopy_transversality_disks}.
Let $k \in \N$, $p > 2$ and $\delta > 0$ and let $\mathcal{O}_{\Gamma} =
\{\gamma_z \}_{z \in \Gamma}$ be a collection of Reeb orbits of $M_{\pm}$,
one for every puncture $z \in \Gamma = \Gamma^{+} \cup \Gamma^{-}$.
We define the space \label{p:bkpd}
\[
\mathcal{B}^{k, p, \delta} := \mathcal{B}^{k, p, \delta}(\dot{\Sigma},
\partial \Sigma; X, L; \mathcal{O}_{\Gamma})
\]
to consist of maps $u: \dot{\Sigma} \to X$ of class $W^{k,p}_{\mathrm{loc}}$
which satisfy $u(\partial \Sigma) \subset L$ and have asymptotically
cylindrical behavior approaching $\gamma_z$ at the puncture $z \in
\Gamma^{\pm}$. Basically, one can view $u$ as a decreasing perturbation
of $y(s,t):=(Ts,\gamma_z(Tt))$ up to a shift in the cylindrical ends. 
To be precise, this means that in cylindrical coordinates $(s,t) \in
Z_{\pm}$ near $z$, there exist constants $s_0$ and $t_0$ such that for
sufficiently large $|s|$ we have
\[
u(s + s_0, t+ t_0) = \exp_{y(s,t)}h(s,t),
\]
where $h \in W^{k,p,\delta}(y^*T H_{\pm})$ and the exponential map is
defined with respect to any $\R$-invariant metric on $H_{\pm} :=
\R\times M_{\pm}$. The condition $kp > 2$ implies via the
Sobolev embedding theorem that $\mathcal{B}^{k, p, \delta}
\hookrightarrow C^0(\dot{\Sigma}, X)$.
Even though $\dot{\Sigma}$ is non-compact, we can still give the space
$\mathcal{B}^{k, p,\delta}$ the structure of a smooth, separable and metrizable
Banach manifold by generalizing the results of \cite{eliasson}.
\[
T_u\mathcal{B}^{k, p, \delta} = W^{k,p, \delta}_{\Lambda}(u^*TX) \oplus V,
\]
where the summands are defined as follows. $\Lambda$ is the Lagrangian
subbundle
\[
\Lambda := (u|_{\partial \Sigma})^*TL \longrightarrow \partial \Sigma,
\]
so that sections $v \in W^{k,p, \delta}_{\Lambda}(u^*TW)$ are required to
decay exponentially near the puncture and satisfy $v(\partial \Sigma)
\subset \Lambda$. $V$ is a $2(\#\Gamma)$-dimensional real vector space with basis
given by non-canonical choices of two sections $\dot{\Sigma} \to u^*TW$
supported in $\mathcal{U}_z$ for every puncture $z \in \Gamma$ and
asymptotic to the vector fields $\partial_r, R_{\alpha_{\pm}}$ of
$T(\R \times M_{\pm})$. In particular, $V$ contains vector fields that
are asymptotically parallel to orbit cylinders $y(s, t) = (Ts, \gamma_z(Tt))$
in the cylindrical ends.

Let $J \in \mathcal{J}^{\infty}_\mathrm{cyl}$. Note that by \cite[Theorem 1.5]{howize},
for a finite family of non-degenerate
orbits $\mathcal{O}_{\Gamma}$ there exists an open set $I \subset \R$
containing 0, such that for all $\delta \in I$ maps satisfying $du +
J\circ du \circ i = 0$ and the asymptotic condition at $\mathcal{O}_{\Gamma}$
automatically satisfy the exponential decay estimate near each puncture.
Hence for $\delta$ sufficiently small the space $\mathcal{B}^{k, p, \delta}$
contains all pseudoholomorphic maps asymptotic to a chosen family of Reeb
orbits.

For $l \geq k$ consider now the space $\mathcal{J}^l$
of almost complex
structures of class $C^l$ on $X$ that are cylindrical at infinity. Since
$\mathcal{J}^l$ consists of almost complex structures that are translation
invariant outside of a compact set, we can give $\mathcal{J}^l$ the
structure of a smooth separable Banach manifold by using the translation
invariant metric. For $J \in \mathcal{J}^l$ the tangent space
$T_J\mathcal{J}^l$ consists of compactly supported $C^l$-sections $Y$ of the smooth bundle
$\mathrm{End}(TX, J, \omega) \to X$ that satisfy
\[
YJ + JY = 0, \qquad \omega(Yv, w) + \omega(v, Yw) = 0.
\]
The first equation is derived from the condition $J^2 = -\mathrm{Id}$ and
the second equation comes from the compatibility of $J$ with $\omega$. The
space of such sections is a Banach space and provides a local chart
containing $J$ via the mapping $Y \mapsto J \exp (-JY)$.

We now consider the bundle $\mathcal{E}^{k-1, p, \delta} \to \mathcal{B
}^{k, p,\delta} \times \mathcal{J}^l$, whose fiber over $(u, J)$ is given by
\[
\mathcal{E}^{k-1, p, \delta}_{(u,J)} := W^{k-1, p, \delta}\left(\mathrm{Hom}^{0,1}
\left((T\dot{\Sigma}, j), (u^*TX, J) \right)\right),
\]
the space of complex-antilinear bundle maps. One can show that $\mathcal{E}^{
k-1, p, \delta}$ has the structure of a $C^{l-k}$ Banach space bundle, see
the proof of Proposition 3.2.1 in \cite{mcsa}. The map $\mathcal{F}: \mathcal{B}^{k, p,\delta}\times
\mathcal{J}^l \to \mathcal{E}^{k-1, p, \delta}$ given by
\begin{equation}\label{eq:def-section}
\mathcal{F}(u,J) := du + J\circ du \circ j
\end{equation}
then defines a $C^{l-k}$-section of the bundle,
because $\mathcal{F}(u,J) \in \mathcal{E}^{k-1, p, \delta}$
satisfies the exponential weighting condition when
$J$ is translation invariant on the ends. The zeros of this section is given by
the union of the moduli spaces $\mathcal{M}(\mathcal{O}_{\Gamma}, J)$ for
$J \in \mathcal{J}^l$. For $(u, J)$ such that $\mathcal{F}(u,J) = 0$
the vertical differential $D\mathcal{F}(u, J): T_u \mathcal{B}^{k, p,
\delta}\times T_J\mathcal{J}^l \to \mathcal{E}^{k-1, p, \delta}_{(u, J)}$
of the section $\mathcal{F}$ is given by  \cite[Section 3.2]{mcsa}:
\[
D\mathcal{F}(u, J)(\xi, Y) := \nabla\xi + J\circ \nabla \xi \circ j +
(\nabla_{\xi} J) \circ du \circ j + Y(u) \circ du \circ j,
\]
where $\nabla$ is any symmetric connection on $X$. For fixed $J \in \mathcal{J}^l$
we also consider the restriction of $\mathcal{F}$ to $\mathcal{B}^{k, p, \delta}$
and the associated vertical differential $\mathbf{D}_u: T_u \mathcal{B}^{k, p, \delta} \to
\mathcal{E}^{k-1, p, \delta}_{(u, J)}$ at the zero section,
\[
\mathbf{D}_u(\xi) =  \nabla\xi + J\circ \nabla \xi \circ j + (\nabla_{\xi} J)
\circ du \circ j.
\]
The smooth manifold structure of $\mathcal{M}^*(\mathcal{O}_{\Gamma}, J)$
depends on the properties of the map $\mathbf{D}_u$ as a Fredholm operator.
We continue with the analysis of this operator. Let $\tau$ be an
admissible trivialization of $u^*TX$ on the cylindrical ends and boundary
of $\dot{\Sigma}$.

\begin{prop}\label{prop:Du_fredholm}
For $\delta > 0$ sufficiently small the operator $\mathbf{D}_u:
T_u \mathcal{B}^{k, p,\delta} \to \mathcal{E}^{k-1, p, \delta}_{(u, J)}$
is Fredholm and has index
\begin{equation}\label{eq:index-Du}
\mathrm{ind}(\mathbf{D}_u) = n\chi(\dot{\Sigma}) + 2c_1^{\tau}
(u^*TX) + \mu^{\tau}(u^*TX, \Lambda) + \sum\limits_{i = 1}^m
\mu_{\mathrm{CZ}}^{\tau}(\gamma_i^+) - \sum\limits_{j = 1}^n
\mu_{\mathrm{CZ}}^{\tau}(\gamma_j^-) + \# \Gamma.
\end{equation}
\end{prop}
\noindent{\it Proof:}
Since $T_u \mathcal{B}^{k, p, \delta} = W_{\Lambda}^{k,p, \delta}(u^*TX) \oplus
V$ and $V$ is finite dimensional, it suffices to prove that $\mathbf{D}_u$
restricted to $W_{\Lambda}^{k,p, \delta}(u^*TX)$ is Fredholm. By abuse of
notation let $\mathbf{D}_u$ be the restricted operator. We proceed by
showing that $\mathbf{D}_u$ is conjugate to a Cauchy-Riemann type operator
with non-degenerate asymptotic operators at the punctures.

Let $z \in \Gamma^{\pm}$ be a puncture and suppose 
$\mathbf{D}_u(\cdot)\tfrac{\partial}{\partial s}$ 
takes the form $\overline{\partial} + S(s,t)$ in the coordinates of the admissible
trivialization on the cylindrical end of $z$, where $S(s,t)$ converges
uniformly as $s \to \pm\infty$ to a smooth loop of symmetric matrices
$S_{z}(t)$. A computation reveals (see the proof of Theorem 3.6 of \cite{dragnev})
that $S_{z}(t) = 0_{2 \times 2} \oplus S_{\gamma_z}(t)$, where we use
the splitting in the trivialization arising from $T(\mathbb{R} \times
M_{\pm}) \simeq \beta \oplus \xi_{\pm}$. Here $\beta$ is the complex
line bundle generated by $\partial_r$ and the Reeb vectorfield $R_{
\alpha_{\pm}}$ and $\xi_{\pm}$ is the pullback of the contact
distribution on $M_{\pm}$. The matrix $S_{\gamma_z}(t)$ is
derived from the linearization of the Reeb flow at $\gamma_z$
and is symmetric and non-singular since $\gamma_z$ is non-degenerate.
Now since $S_z(t)$ is zero on the complex line bundle $\beta$, this implies
that the associated asymptotic operator $\mathbf{A}_z := -J_0 \partial_t
- S_z(t)$ is degenerate. In particular the operator (for $\delta = 0$)
\[
\mathbf{D}_u: W_{\Lambda}^{k,p}(u^*TX) \longrightarrow
W^{k-1, p}(\mathrm{Hom}^{0,1}(T\dot{\Sigma}, u^*TX))
\]
is not Fredholm. This issue is resolved by considering spaces with
exponential weights.

Let $(E, \omega, J)$ be a Hermitian vector bundle over $\dot{\Sigma}$.
For $\delta \in \R$ pick a smooth function $f: \dot{\Sigma}
\to \mathbb{R}$ such that $f(\pm s, t) = {\mp \delta s}$ on the
cylindrical ends. Then we obtain Banach space isomorphisms
\[
\Phi_{\delta}: W^{k,p}(E) \to W^{k,p, \delta}(E), \;\; \eta \mapsto e^f \eta, \qquad
\Psi_{\delta}: W^{k-1,p, \delta}(E) \to W^{k-1,p}(E), \;\; \theta \mapsto e^{f} \theta.
\]
Returning to our setting we consider the bounded linear map
\[
\mathbf{D}_u' := \Psi_{\delta}^{-1} \mathbf{D}_u \Phi_{\delta}: W^{k,p}_{\Lambda}
(u^*TX) \to W^{k-1, p}(\mathrm{Hom}^{0,1}(T\dot{\Sigma}, u^*TX)).
\]
We see that $\mathbf{D}_u'$ is a linear Cauchy-Riemann type operator on
$\dot{\Sigma}$. Moreover a short calculation reveals that at a puncture
$z \in \Gamma^{\pm}$ the operator $\mathbf{D}_u'$ takes the form
\[
\mathbf{D}_u' \eta = \overline{\partial} \eta + (S(s,t) \mp \delta \,
\mathrm{Id}_{2n\times 2n}) \eta
\]
in the trivialization and is therefore asymptotic to the operator
\[
\mathbf{A}'_z := -J_0 \partial_t - S_z(t) \pm \delta \, \mathrm{Id}_{2n\times 2n}
= \mathbf{A}_z \pm \delta \, \mathrm{Id}_{2n\times 2n}.
\]
The spectrum of $\mathbf{A}_z$ is discrete.
Thus for $\delta > 0$ chosen small enough, we can assume
that $\ker(\mathbf{A}'_z)$ remains trivial for all $z \in \Gamma^{\pm}$.
The Conley-Zehnder index for $z \in \Gamma^{\pm}$ then computes to
(see \cite{wendl2})
\[
\mu^{\tau}_{\mathrm{CZ}}(\mathbf{A}'_z) = \mp 1 + \mu^{\tau}_{\mathrm{CZ}}
(\gamma_z). 
\]
Applying Theorem~\ref{thm:cauchy-riemann_op_fredholm}
and Theorem~\ref{thm:cauchy-riemann_op_index} we see that $\mathbf{D}_u'$
is Fredholm with index equal to the right hand side of
equation~\eqref{eq:index-Du} $- \ 2\# \Gamma$. We then include the
dimension of $V$ to obtain the index given by \eqref{eq:index-Du}.\cqfd

\begin{corollary}\label{cor:index}
Let $(L,g)$ be a Riemanian manifold and $(X,\om) := (T^*L,d\lambda)$,
seen as a symplectic cobordism with one positive end $M \times [0,\infty)$,
where $M=\partial \cw_g$ is the unit sphere bundle associated to the metric
$g$. Let $J_g$ be the almost complex structure defined in section \ref{sec:jpart},
$\gamma$ a closed geodesic of minimal length in its homology class, and
$u_{\gamma,g}$ the unique element of $\cm(J_g,\beta)$ asymptotic to
$\tilde \gamma$ modulo reparametrization. Then
$$
\mathrm{ind} (\mathbf{D}_{u_{\gamma,g}}) = 1.
$$
\end{corollary}

\noindent{\it Proof:} Recall that  $\im u_{\gamma,g} = \{\,s\tilde \gamma(t) \;|\;
(s,t)\in [0,\infty) \times \R/\ell \Z \,\}$, where $\tilde \gamma(t) =
(\gamma(t),\dot \gamma(t)^\sharp)$ is the natural lift of $\gamma$ to $M$.
Choose an orthonormal basis $(\dot \gamma(t),e_2(t),\dots, e_n(t))$ along
$\gamma$ and lift these vectors to vectors 
$(\dot{\tilde\gamma}(t),E_1(s,t),\dots,E_n(s,t))$, where $E_j(s,t)$ is a
horizontal vector in $T_{u_{\gamma,g}(s,t)} T^*L$. Since $H$ is totally real
for the structure $J_{g}$, we get a complex splitting of $T^*L$ along
$u_{\gamma,g}$ given by
$$
u_{\gamma,g}^*T(T^*L)(s,t) = \langle R,\tfrac\partial{\partial r} \rangle
\oplus \langle E_1(s,t),J_g E_1(s,t) \rangle \oplus \dots \oplus
\langle E_n(s,t), J_g E_n(s,t)\rangle.
$$
Let $\tau$ be the symplectic trivialization of $u_{\gamma,g}^*T(T^*L)$
induced by this decomposition. Then, $c_1^\tau(u_{\gamma,g}^*T(T^*L))$
and $\mu^\tau(u_{\gamma,g}^*T(T^*L),TL)$ obviously vanish. Since
$\chi(D\priv\{0\})=0$, we get
$$
\mathrm{ind}(\mathbf{D}_{u_{\gamma,g}}) = \mu_{\text{CZ}}^\tau
(\tilde \gamma) + 1.
$$
It remains to show that $\mu_\text{CZ}^\tau(\tilde \gamma)=0$. We recall
that this is the Maslov index of the path of symplectic matrices given by
the linearization of the Reeb flow (hence the cogeodesic flow) along $\gamma$.
Notice that any deformation $g_\eps$ of the metric that leaves $\gamma$ a
geodesic of fixed length induces a continuous deformation of
$\mathbf{D}_{u_{\gamma,g}}$ among Fredholm operators defined on the same
Banach space (because $u_{\gamma,g_\eps} = u_{\gamma,g}$ for all $\eps$
under this assumption). These
Fredholm operators thus all have the same index. By proposition \ref{prop:goodmetric},
completed by remark \ref{rk:goodmetric}, we can therefore assume for our
computation of $\mathrm{ind}(\mathbf{D}_{u_{\gamma,g}})$ that $g$ has the
very particular form achieved by proposition \ref{prop:goodmetric}: in
Fermi coordinates near $\gamma$ we have $g_{ij}=(1+k\|x'\|^2)\delta_{ij}$.
Now for this special metric the derivative of the geodesic flow can be
computed explicitly. It preserves the horizontal and the vertical
distributions. In the basis $(\dot \gamma(t),e_2(t),\dots,e_n(t))$
of the horizontal distribution it has the form
$$
\left(\begin{matrix}
1& 0\\
0&k\id
\end{matrix}\right).
$$
The vanishing of the Conley-Zehnder index readily follows.
\cqfd

\subsection{A universal moduli space}
Choose an integer $l \geq 2$ and a real number $p > 2$, let $k \in \{
1, \ldots, l \}$. We define the universal moduli space
\begin{equation*}
\mathcal{M}^*(\mathcal{O}_{\Gamma}, \mathcal{J}^l) := \left\{ (u, J)
\;\; \Bigg| \;\;
 \begin{array}{c}
      J \in \mathcal{J}^l, \; u \in \mathcal{M}(\mathcal{O}_{\Gamma}, J),  \\
      u \text{ has an injective point mapped to } \mathrm{int}(K)
    \end{array}
\right\}.
\end{equation*}
\begin{prop}\label{prop:universal_moduli_banach}
The universal moduli space $\mathcal{M}^*(\mathcal{O}_{\Gamma},
\mathcal{J}^l)$ is a $C^{l-k}$-Banach submanifold of $\mathcal{B}^{
k, p,\delta} \times \mathcal{J}^l$.
\end{prop}

\begin{proof}
The universal moduli space $\mathcal{M}^*(\mathcal{O}_{\Gamma},
\mathcal{J}^l)$ is a subset of the zero set of the $C^{l-k}$-section
$\mathcal{F}$ (\ref{eq:def-section}) of the Banach space bundle
$\mathcal{E}^{k-1, p, \delta}$. By showing that the vertical
differential $D\mathcal{F}$ is surjective on $\mathcal{M}^*(
\mathcal{O}_{\Gamma}, \mathcal{J}^l)$, this provides us with
the structure of a $C^{l-k}$-Banach submanifold via the infinite
dimensional implicit function theorem.

First note that, as in the proof of Proposition 3.2.1 in \cite{mcsa2}, we can give 
$\mathcal{E}^{k-1, p, \delta}$ the structure of a Banach space bundle of class $C^{l-k}$.
The map $\mathcal{F}$ defines a $C^{l-k}$-section
of the bundle $\mathcal{E}^{k-1, p, \delta}$ such that the zero set of
$\mathcal{F}$ contains $\mathcal{M}^*(\mathcal{O}_{\Gamma}, \mathcal{J
}^l)$. We now show that the operator $D\mathcal{F}(u, J)$ is surjective
for every pair $(u, J) \in \mathcal{M}^*(\mathcal{O}_{\Gamma}, \mathcal{
J}^l)$.

We prove surjectivity first in the case $k = 1$, we consider the
operator $D\mathcal{F}(u, J): T_u \mathcal{B}^{1, p,\delta} \times
T_J \mathcal{J}^l \to L^{p, \delta}(\mathrm{Hom}^{0,1}(T\dot{\Sigma},
u^*TX))$, where $D\mathcal{F}(u, J)(\xi, Y) = \mathbf{D}_u(\xi) +
Y \circ du \circ j$. By Proposition~\ref{prop:Du_fredholm} the map
$\mathbf{D}_u$ is Fredholm and thus has a closed range, by a standard
result in functional analysis this implies that the range of
$D\mathcal{F}(u, J)$ is closed as well. Thus we prove that the image of
$D\mathcal{F}(u, J)$ is dense. If the image of $D\mathcal{F}(u, J)$ is
not dense, then by the Hahn-Banach theorem, there exists a non-zero
section $\eta \in L^{q, -\delta}(\mathrm{Hom}^{0,1}(T\dot{\Sigma},
u^*TX))$ for $1/p + 1/q = 1$ which annihilates the image of
$D\mathcal{F}(u, J)$. In particular this implies
\[
\langle D\mathcal{F}(u, J)(\xi, Y), \eta \rangle_{L^2} = 0,
\]
for every $(\xi, Y) \in T_u \mathcal{B}^{1, p,\delta}\times T_J
\mathcal{J}^l$, which in turn implies
\begin{equation}\label{eq:image_vert_diff}
\langle \mathbf{D}_u(\xi), \eta \rangle_{L^2} = 0, \quad \langle
Y \circ du \circ i, \eta \rangle_{L^2} = 0,
\end{equation}
for all elements in the domain. In particular this implies that
$\eta$ is a weak solution to the formal adjoint equation
$\mathbf{D}_u^*(\eta) = 0$. By elliptic regularity  $\eta$ is continuous
and applying the Carleman similarity principle \cite{mcsa} we
see that $\eta$ has isolated zeros. Now the set of injective points
of $u$ is open, so there
exists a $z_0 \in \dot{\Sigma}$ such that $z_0$ is an injective point of
$u$ that maps to $\mathrm{int}(K)$ and $\eta(z_0) \neq 0$. By Lemma 3.2.2
of \cite{mcsa} there exists a $Y \in T_J \mathcal{J}^l$ such that
$Y \circ du \circ j(z_0) = \eta(z_0)$, so the inner product $\langle
Y \circ du \circ i, \eta \rangle_{L^2} > 0$ in some neighborhood of $z_0$.
Now multiply $Y$ with a non-negative bump function $\beta$ on $X$
which has support in the aforementioned neighborhood. Then we have
$\langle (\beta Y) \circ du \circ i, \eta \rangle_{L^2} > 0$ since $z_0$ is
an injective point, which violates
(\ref{eq:image_vert_diff}) and hence surjectivity in the case $k = 1$ holds.

In the general case assume $\theta \in W^{k-1,p, \delta}(\mathrm{Hom}^{0,1}
(T\dot{\Sigma}, u^*TX))$ for $k \geq 2$. Then $\theta \in L^{p, \delta}$ and
by surjectivity for $k = 1$ there exists $(\xi, Y) \in T_u \mathcal{B}^{1,
p,\delta}\times T_J \mathcal{J}^l$ such that $\mathbf{D}_u(\xi) + Y \circ du \circ i
= \theta$. $Y$ is of class $C^l$, so
$Y \circ du \circ i$ is of class $W^{k-1, p, \delta}$ and thus $\mathbf{D}_u
(\xi) = \theta - Y \circ du \circ i$ is of class $W^{k-1, p, \delta}$. Elliptic
regularity then implies that $\xi$ is of class $W^{k, p,\delta}$. This proves
that $D\mathcal{F}(u, J)$ is surjective for general $k$.

Now, since $\mathbf{D}_u$ is a Fredholm operator, by Lemma A.3.6 of
\cite{mcsa} $D\mathcal{F}(u, J)$ has a right inverse. The
infinite dimensional implicit function theorem then implies
that $\mathcal{M}^*(\mathcal{O}_{\Gamma}, \mathcal{J}^l)$ is
a $C^{l-k}$-Banach submanifold of $\mathcal{B}^{k,
p,\delta} \times \mathcal{J}^l$.
\end{proof}

\begin{definition}
An almost complex structure $J \in \mathcal{J}^{\infty}_{\mathrm{Cyl}}$
is called \emph{regular}, if the operator $\mathbf{D}_u$ is onto for every
$u \in \mathcal{M}^*(\mathcal{O}_{\Gamma}, J)$. We denote by $\mathcal{
J}^{\mathrm{reg}} \subset \mathcal{J}^{\infty}_{\mathrm{Cyl}}$ the subset
of all regular almost complex structures on $X$.
\end{definition}

\begin{proof}[Proof of Theorem \ref{thm:transversality_disks} (i)]
We prove that $\mathcal{M}^*(\mathcal{O}_{\Gamma}, J)$ is a
smooth manifold of the prescribed dimension around $u$.
Let $J \in \mathcal{J}^{\mathrm{reg}}$ and $u \in \mathcal{M}^*(\mathcal{O
}_{\Gamma}, J)$. By Proposition \ref{prop:elliptic_regularity}
$u$ is smooth. For an integer $k \geq 1$
and $p > 2$ we consider the section $\mathcal{F}$ for $J$ fixed in a
trivialization of $\mathcal{E}^{k-1,p, \delta}$ over a neighborhood
of $\mathcal{N}(u)$ of $u \in \mathcal{B}^{k,p, \delta}$.
We can identify $\mathcal{N}(u)$ with a neighborhood $U$ of $0$ in
$T_u \mathcal{B}^{k,p, \delta}$ and consider the restricted map
\[
\mathcal{F}_u := \mathcal{F}_{|U} : U \longrightarrow
W^{k-1,p, \delta}(\mathrm{Hom}^{0,1}(T\dot{\Sigma}, (u^*TX, J))).
\]
Then $\mathcal{F}_u$ is a smooth map between Banach spaces
and the differential $d\mathcal{F}_u(0) = \mathbf{D}_u$ is surjective
by assumption. By the infinite dimensional implicit function
theorem \cite{smale2}, $\mathcal{F}_u^{-1}(0)$
intersects a sufficiently small neighborhood of $0$ in a smooth
finite dimensional submanifold of dimension
$\mathrm{ind}(\mathbf{D}_u)$. The image of this
submanifold under the map $\xi \mapsto \exp_u(\xi)$ is a smooth
submanifold of $\mathcal{B}^{k,p, \delta}$ that agrees with a
neighborhood of $u \in \mathcal{M}^*(\mathcal{O}_{\Gamma}, J)$.
Hence $\mathcal{M}^*(\mathcal{O}_{\Gamma}, J)$ is a smooth
submanifold of $\mathcal{B}^{k,p, \delta}$ of dimension
$\mathrm{ind}(\mathbf{D}_u)$ locally around $u$.
\end{proof}

\begin{prop}\label{prop:proj_fredholm}
The projection $\mathcal{M}^*(\mathcal{O}_{\Gamma}, \mathcal{J}^l)
\to \mathcal{J}^l$ is a nonlinear Fredholm map and for $l$ large
enough the set of regular values in $\mathcal{J}^l$ is dense.
\end{prop}

\begin{proof}
The projection $\pi: \mathcal{M}^*(\mathcal{O}_{\Gamma}, \mathcal{J}^l)
\to \mathcal{J}^l$ is, by Proposition \ref{prop:universal_moduli_banach}
for $k = 1$, a $C^{l-1}$-map
between separable $C^{l-1}$-Banach manifolds. The tangent space at $(u, J)$
\[
T_{(u, J)}\mathcal{M}^*(\mathcal{O}_{\Gamma}, \mathcal{J}^l) \subset
T_u \mathcal{B}^{1, p, \delta} \times T_J \mathcal{J}^l,
\]
consists of all pairs $(\xi, Y)$ such that $\mathbf{D}_u (\xi) + Y(u) \circ du
\circ j = 0$. The derivative $d\pi (u, J)$ is just the projection
$(\xi, Y) \mapsto Y$. Thus the kernel of $d\pi(u, J)$ is isomorphic
to the kernel of $\mathbf{D}_u$. By a standard result in functional analysis
\cite[Lemma A.3.6]{mcsa} the cokernel of $d\pi(u, J)$ is
also isomorphic to the cokernel of $\mathbf{D}_u$. It follows that $d\pi(u, J)$
is a Fredholm operator with the same index as $\mathbf{D}_u$. Moreover the
operator $d\pi(u, J)$ is onto precisely when $\mathbf{D}_u$ is onto. This implies
that a regular value $J$ of $\pi$ is an almost complex structure with
the property that $\mathbf{D}_u$ is onto for every somewhere injective
curve $u \in \mathcal{M}^*(\mathcal{O}_{\Gamma}, J) = \pi^{-1}(J)$. In
other words,
\begin{equation}\label{def:reg-acs-of-regularity-l}
\mathcal{J}^{\mathrm{reg}, l} := \{ J \in \mathcal{J}^l \;|\;
\mathbf{D}_u \text{ is onto for all } u \in \mathcal{M}^*(\mathcal{O}_{\Gamma}, J) \},
\end{equation}
the set of regular almost complex structures of class $C^l$ is the
set of regular values of $\pi$. By the Sard-Smale theorem
\cite{smale2}, this set
is of second category in the sense of Baire. Here we use the fact
that $\mathcal{M}^*(\mathcal{O}_{\Gamma}, \mathcal{J}^l)$ and the projection $\pi$
are of class $C^{l-1}$ and we can apply Sard-Smale whenever $l - 2 \geq
\mathrm{ind}(\mathbf{D}_u)$. Thus the set  $\mathcal{J}^{\mathrm{reg},l}$
is dense in $\mathcal{J}^l$ with respect to the $C^l$-topology for
$l$ sufficiently large.
\end{proof}

\begin{proof}[Proof of Theorem \ref{thm:transversality_disks} (ii)]
We must show that the set $\mathcal{J}^{\mathrm{reg}}$ is of second
category in $\mathcal{J}^{\infty}_{\mathrm{Cyl}}$.

We fix metrics on $X$ and $\dot{\Sigma}$ that are translation invariant
on the cylindrical ends and we denote by $\mathrm{dist}(\cdot, \cdot)$
the induced distance functions. For $N \in \N$ consider the set
\[
\mathcal{J}^{\mathrm{reg}}_N := \mathcal{J}^{\mathrm{reg}}_N
(\mathcal{O}_{\Gamma}) \subset \mathcal{J}^{\infty}_{\mathrm{Cyl}}
\]
of all smooth almost complex structures $J$ such that the operator
$\mathbf{D}_u$ is onto for every $J$-holomorphic curve
$u \in \mathcal{M}^*(\mathcal{O}_{\Gamma}, J)$ that satisfies
\begin{enumerate}
\item[(i)] $\sup\limits_{z \in \dot{\Sigma}} | du(z) | \leq N$;
\item[(ii)] there exists a $z_0 \in \dot{\Sigma}$ such that
\[
\mathrm{dist}(u(z_0), X \backslash \mathrm{int}(K)) \geq \frac{1}{N},
\quad |du(z_0)| \geq \frac{1}{N}, \quad
\inf\limits_{z \in \dot{\Sigma} \backslash \{ z_0 \}} \frac{\mathrm{dist}
(u(z_0), u(z))}{\mathrm{dist}(z_0, z)} \geq \frac{1}{N}.
\]
\end{enumerate}
Note that the set of such $u$ has a somewhere injective point
mapped to $\mathrm{int}(K)$. Furthermore, every asymptotically
cylindrical $J$-holomorphic curve with an injective point mapped to
$\mathrm{int}(K)$ satisfies these conditions for some value of
$N \in \N$.

We claim that $\mathcal{J}^{\mathrm{reg}}_N$ is open and dense
in $\mathcal{J}^{\infty}_{\mathrm{Cyl}}$. We first show that this
set is open, which is equivalent to the complement being closed.
Assume we have a sequence $J_{\nu} \notin \mathcal{J}^{\mathrm{reg}}_N$,
$J_{\nu} \to J$ in $C^{\infty}$. This means that for every $\nu$
there exists a $J_{\nu}$-holomorphic $u_{\nu}$ and a $z_{\nu} \in
\dot{\Sigma}$ that satisfy conditions (i) and (ii) and such that
$\mathbf{D}_{u_{\nu}}$ is not surjective.
Since the first derivatives of $u_{\nu}$ are uniformly bounded
and (ii) is a closed condition, by a standard elliptic
bootstrapping argument \cite[Theorem B.4.2]{mcsa}
there exists a subsequence $u_{\nu_i}$ that converges uniformly
with all derivatives to a smooth $J$-holomorphic curve $u$ that
satifisfies conditions (i) and (ii). Since the operators
$\mathbf{D}_{u_{\nu_i}}$ are not surjective,
it follows that $\mathbf{D}_u$ is not surjective either. This shows that
$J \notin \mathcal{J}^{\mathrm{reg}}_N$ and thus
$\mathcal{J}^{\mathrm{reg}}_N$ is open in the $C^{\infty}$-topology.
We now prove that $\mathcal{J}^{\mathrm{reg}}_N$ is dense in $\mathcal{J
}^{\infty}_{\mathrm{Cyl}}$. Let $\mathcal{J}^{\mathrm{reg}, l}_N \subset
\mathcal{J}^l$ be the set of all $J \in \mathcal{J}^l$ such that
the operator $\mathbf{D}_u$ is onto for every $u \in \mathcal{M}^*(
\mathcal{O}_{\Gamma}, J)$ that satisfies conditions (i) and (ii).
Now note that
\[
\mathcal{J}^{\mathrm{reg}}_N = \mathcal{J}^{\mathrm{reg}, l}_N \cap
\mathcal{J}^{\infty}_{\mathrm{Cyl}}.
\]
Let $J \in \mathcal{J}^{\infty}_{\mathrm{Cyl}} \subset \mathcal{J}^l$.
Now by Proposition~\ref{prop:proj_fredholm} the set
$\mathcal{J}^{\mathrm{reg}, l}$ of \eqref{def:reg-acs-of-regularity-l} is
dense in $\mathcal{J}^l$ for large $l$, so there exists a
sequence $J_l \in \mathcal{J}^{\mathrm{reg}, l}$ such that
\[
\| J_l - J \|_{C^l} \leq 2^{-l}.
\]
Since $J_l \in \mathcal{J}^{\mathrm{reg}, l}_N$ and $\mathcal{J}^{
\mathrm{reg}, l}_N$ is open in the $C^l$-topology,
there exists an $\varepsilon_l$ such that for every $J' \in \mathcal{J}^l$,
\[
\| J_l - J' \|_{C^l} < \varepsilon_l  \Longrightarrow J' \in
\mathcal{J}^{\mathrm{reg}, l}_N.
\]
Choose $J_l' \in \mathcal{J}^{\infty}_{\mathrm{Cyl}}$ to be any smooth
element such that
\[
\| J_l' - J_l \|_{C^l} \leq \mathrm{min}\{ \varepsilon_l, 2^{-l} \}.
\]
Then $J_l' \in \mathcal{J}^{\mathrm{reg}, l}_N \cap \mathcal{J}^{\infty
}_{\mathrm{Cyl}} = \mathcal{J}^{\mathrm{reg}}_N$ converges to $J$
in the $C^{\infty}$-topology. This shows that the set of $\mathcal{J}^{
\mathrm{reg}}_N$ is dense in $\mathcal{J}^{\infty}_{\mathrm{Cyl}}$
as claimed. Thus $\mathcal{J}^{\mathrm{reg}}$ is the intersection of
a countable number of open dense sets $\mathcal{J}^{\mathrm{reg}}_N$,
$N \in \N$, and is so of second category.
\end{proof}

\begin{proof}[Proof of Theorem \ref{thm:homotopy_transversality_disks}]
Applying the ideas of the proof of Theorem~\ref{thm:transversality_disks}
(i) to the bundle $\mathcal{E}^{k-1, p, \delta} \to \mathcal{B}^{k,p, \delta}
\times [0,1]$ gives us the manifold structure of $\mathcal{W}^*(
\mathcal{O}_{\Gamma}, \{ J_t \})$.
For (ii) we proceed analogous to the proof of
Theorem~\ref{thm:transversality_disks}. Define $\mathcal{J}^l(J_0,
J_1)$ to be the space of $C^l$-homotopies in $\mathcal{J}^l$ from
$J_0$ to $J_1$. Consider the universal moduli space
$\mathcal{W}^*(\mathcal{O}_{\Gamma}, \mathcal{J}^l(J_0,J_1))$.
One can show that this space is a $C^{l-1}$-Banach manifold and that
the projection map to $\mathcal{J}^l(J_0,J_1)$ is a $C^{l-1}$
Fredholm map. By Sard-Smale, the regular values of this map is
dense for $l$ sufficiently large and produces the desired subset
of regular homotopies. The conclusion to the case of smooth
homotopies is then as in the proof of
Theorem~\ref{thm:transversality_disks}.
\end{proof}


{\footnotesize
\bibliographystyle{alpha}
\bibliography{biblio2}
}

{\small
\begin{minipage}{0.45\linewidth}
\noindent Cedric Membrez\\
School of Mathematical Sciences\\
Tel Aviv University\\
Tel Aviv 69978, Israel\\
\texttt{ckmembrez@gmail.com}\\
\end{minipage}
\hfill
\begin{minipage}{0.6\linewidth}
\noindent Emmanuel Opshtein,\\
Institut de Recherche Math\'ematique Avanc\'ee\\
UMR 7501, Universit\'e de Strasbourg et CNRS\\
7 rue Ren\'e Descartes\\
67000 Strasbourg, France\\
\texttt{opshtein@unistra.fr}
\end{minipage}
}

\end{document}

%% file: sft1.pdf_t
\begin{picture}(0,0)%
\includegraphics{sft1.pdf}%
\end{picture}%
\setlength{\unitlength}{2072sp}%
\begingroup\makeatletter\ifx\SetFigFont\undefined%
\gdef\SetFigFont#1#2#3#4#5{%
  \reset@font\fontsize{#1}{#2pt}%
  \fontfamily{#3}\fontseries{#4}\fontshape{#5}%
  \selectfont}%
\fi\endgroup%
\begin{picture}(2459,3222)(979,-4086)
\put(3423,-3396){\makebox(0,0)[lb]{\smash{{\SetFigFont{8}{9.6}{\rmdefault}{\mddefault}{\updefault}{\color[rgb]{0,0,0}$T^*L$}%
}}}}
\put(3333,-2331){\makebox(0,0)[lb]{\smash{{\SetFigFont{8}{9.6}{\rmdefault}{\mddefault}{\updefault}{\color[rgb]{0,0,0}$\R\times M$}%
}}}}
\put(3334,-1364){\makebox(0,0)[lb]{\smash{{\SetFigFont{8}{9.6}{\rmdefault}{\mddefault}{\updefault}{\color[rgb]{0,0,0}$\R\times M$}%
}}}}
\put(2185,-1104){\makebox(0,0)[lb]{\smash{{\SetFigFont{7}{8.4}{\rmdefault}{\mddefault}{\updefault}{\color[rgb]{0,0,0}$\tilde \gamma$}%
}}}}
\end{picture}%

%% file: sft2.pdf_t
\begin{picture}(0,0)%
\includegraphics{sft2.pdf}%
\end{picture}%
\setlength{\unitlength}{2072sp}%
\begingroup\makeatletter\ifx\SetFigFont\undefined%
\gdef\SetFigFont#1#2#3#4#5{%
  \reset@font\fontsize{#1}{#2pt}%
  \fontfamily{#3}\fontseries{#4}\fontshape{#5}%
  \selectfont}%
\fi\endgroup%
\begin{picture}(4783,6160)(978,-6798)
\put(2315,-858){\makebox(0,0)[lb]{\smash{{\SetFigFont{7}{8.4}{\rmdefault}{\mddefault}{\updefault}{\color[rgb]{0,0,0}$\tilde \gamma$}%
}}}}
\put(3300,-6223){\makebox(0,0)[lb]{\smash{{\SetFigFont{7}{8.4}{\rmdefault}{\mddefault}{\updefault}{\color[rgb]{0,0,0}$L$}%
}}}}
\put(5712,-4446){\makebox(0,0)[lb]{\smash{{\SetFigFont{8}{9.6}{\rmdefault}{\mddefault}{\updefault}{\color[rgb]{0,0,0}$\R\times M'$}%
}}}}
\put(5746,-1452){\makebox(0,0)[lb]{\smash{{\SetFigFont{8}{9.6}{\rmdefault}{\mddefault}{\updefault}{\color[rgb]{0,0,0}$\R\times M$}%
}}}}
\put(5681,-2912){\makebox(0,0)[lb]{\smash{{\SetFigFont{8}{9.6}{\rmdefault}{\mddefault}{\updefault}{\color[rgb]{0,0,0}$\widetilde{T^*L\priv \cw'}$}%
}}}}
\put(5733,-5946){\makebox(0,0)[lb]{\smash{{\SetFigFont{8}{9.6}{\rmdefault}{\mddefault}{\updefault}{\color[rgb]{0,0,0}$\widetilde{\cw'}$}%
}}}}
\end{picture}%

%% file: break-cylinder2.pdf_t
\begin{picture}(0,0)%
\includegraphics{break-cylinder2.pdf}%
\end{picture}%
\setlength{\unitlength}{2072sp}%
\begingroup\makeatletter\ifx\SetFigFont\undefined%
\gdef\SetFigFont#1#2#3#4#5{%
  \reset@font\fontsize{#1}{#2pt}%
  \fontfamily{#3}\fontseries{#4}\fontshape{#5}%
  \selectfont}%
\fi\endgroup%
\begin{picture}(5167,3434)(1989,-5209)
\put(3948,-3863){\makebox(0,0)[lb]{\smash{{\SetFigFont{8}{9.6}{\rmdefault}{\mddefault}{\updefault}{\color[rgb]{0,0,0}$\ci(C_{\tilde \gamma'})$}%
}}}}
\put(2929,-2918){\makebox(0,0)[lb]{\smash{{\SetFigFont{8}{9.6}{\rmdefault}{\mddefault}{\updefault}{\color[rgb]{0,0,0}$\ci(\tilde \gamma')$}%
}}}}
\put(3430,-4454){\makebox(0,0)[lb]{\smash{{\SetFigFont{8}{9.6}{\rmdefault}{\mddefault}{\updefault}{\color[rgb]{0,0,0}$\iota(\gamma')$}%
}}}}
\put(4994,-4004){\makebox(0,0)[lb]{\smash{{\SetFigFont{9}{10.8}{\rmdefault}{\mddefault}{\updefault}{\color[rgb]{1,0,0}$B$}%
}}}}
\put(3101,-1982){\makebox(0,0)[lb]{\smash{{\SetFigFont{8}{9.6}{\rmdefault}{\mddefault}{\updefault}{\color[rgb]{0,0,0}$\tilde \gamma$}%
}}}}
\put(3760,-4944){\makebox(0,0)[lb]{\smash{{\SetFigFont{8}{9.6}{\rmdefault}{\mddefault}{\updefault}{\color[rgb]{0,0,0}$C_{\iota\circ \gamma'}$}%
}}}}
\end{picture}%

%% file: functions.pdf_t
\begin{picture}(0,0)%
\includegraphics{functions.pdf}%
\end{picture}%
\setlength{\unitlength}{2486sp}%
\begingroup\makeatletter\ifx\SetFigFont\undefined%
\gdef\SetFigFont#1#2#3#4#5{%
  \reset@font\fontsize{#1}{#2pt}%
  \fontfamily{#3}\fontseries{#4}\fontshape{#5}%
  \selectfont}%
\fi\endgroup%
\begin{picture}(9956,5168)(486,-6031)
\put(6216,-5447){\makebox(0,0)[lb]{\smash{{\SetFigFont{10}{12.0}{\rmdefault}{\mddefault}{\updefault}{\color[rgb]{0,0,0}$1$}%
}}}}
\put(7216,-5447){\makebox(0,0)[lb]{\smash{{\SetFigFont{10}{12.0}{\rmdefault}{\mddefault}{\updefault}{\color[rgb]{0,0,0}$1$}%
}}}}
\put(643,-3090){\makebox(0,0)[lb]{\smash{{\SetFigFont{7}{8.4}{\rmdefault}{\mddefault}{\updefault}{\color[rgb]{0,0,0}$0$}%
}}}}
\put(501,-1875){\makebox(0,0)[lb]{\smash{{\SetFigFont{7}{8.4}{\rmdefault}{\mddefault}{\updefault}{\color[rgb]{0,0,0}$1$}%
}}}}
\put(7573,-1875){\makebox(0,0)[lb]{\smash{{\SetFigFont{7}{8.4}{\rmdefault}{\mddefault}{\updefault}{\color[rgb]{0,0,0}$1$}%
}}}}
\put(9287,-3090){\makebox(0,0)[lb]{\smash{{\SetFigFont{7}{8.4}{\rmdefault}{\mddefault}{\updefault}{\color[rgb]{0,0,0}$r$}%
}}}}
\put(2144,-5090){\makebox(0,0)[lb]{\smash{{\SetFigFont{10}{12.0}{\rmdefault}{\mddefault}{\updefault}{\color[rgb]{0,0,0}$1$}%
}}}}
\put(2715,-5090){\makebox(0,0)[lb]{\smash{{\SetFigFont{10}{12.0}{\rmdefault}{\mddefault}{\updefault}{\color[rgb]{0,0,0}$1$}%
}}}}
\put(5144,-4733){\makebox(0,0)[lb]{\smash{{\SetFigFont{7}{8.4}{\rmdefault}{\mddefault}{\updefault}{\color[rgb]{0,0,0}$2$}%
}}}}
\put(572,-3733){\makebox(0,0)[lb]{\smash{{\SetFigFont{7}{8.4}{\rmdefault}{\mddefault}{\updefault}{\color[rgb]{0,0,0}$4$}%
}}}}
\put(5787,-5947){\makebox(0,0)[lb]{\smash{{\SetFigFont{7}{8.4}{\rmdefault}{\mddefault}{\updefault}{\color[rgb]{0,0,0}$2\eps$}%
}}}}
\put(8101,-1411){\makebox(0,0)[lb]{\smash{{\SetFigFont{7}{8.4}{\rmdefault}{\mddefault}{\updefault}{\color[rgb]{0,0,0}$\chi$}%
}}}}
\put(3061,-3121){\makebox(0,0)[lb]{\smash{{\SetFigFont{7}{8.4}{\rmdefault}{\mddefault}{\updefault}{\color[rgb]{0,0,0}$1$}%
}}}}
\put(1914,-3121){\makebox(0,0)[lb]{\smash{{\SetFigFont{7}{8.4}{\rmdefault}{\mddefault}{\updefault}{\color[rgb]{0,0,0}$\frac {1}{2}$}%
}}}}
\put(1351,-3121){\makebox(0,0)[lb]{\smash{{\SetFigFont{7}{8.4}{\rmdefault}{\mddefault}{\updefault}{\color[rgb]{0,0,0}$\frac {1}{4}$}%
}}}}
\put(2499,-3121){\makebox(0,0)[lb]{\smash{{\SetFigFont{7}{8.4}{\rmdefault}{\mddefault}{\updefault}{\color[rgb]{0,0,0}$\frac 34$}%
}}}}
\put(4096,-1861){\makebox(0,0)[lb]{\smash{{\SetFigFont{7}{8.4}{\rmdefault}{\mddefault}{\updefault}{\color[rgb]{0,0,0}$1$}%
}}}}
\put(1144,-1375){\makebox(0,0)[lb]{\smash{{\SetFigFont{10}{12.0}{\rmdefault}{\mddefault}{\updefault}{\color[rgb]{0,0,0}$h$}%
}}}}
\put(4715,-1375){\makebox(0,0)[lb]{\smash{{\SetFigFont{10}{12.0}{\rmdefault}{\mddefault}{\updefault}{\color[rgb]{0,0,0}$k$}%
}}}}
\put(991,-3976){\makebox(0,0)[lb]{\smash{{\SetFigFont{10}{12.0}{\rmdefault}{\mddefault}{\updefault}{\color[rgb]{0,0,0}$hk'-kh'$}%
}}}}
\put(5716,-3976){\makebox(0,0)[lb]{\smash{{\SetFigFont{10}{12.0}{\rmdefault}{\mddefault}{\updefault}{\color[rgb]{0,0,0}$hk'-kh'$}%
}}}}
\end{picture}%

%% file: rigspec.bbl
\newcommand{\etalchar}[1]{$^{#1}$}
\def\cprime{$'$}
\begin{thebibliography}{{Wen}16}

\bibitem[Abb14]{abbas}
Casim Abbas.
\newblock {\em An introduction to compactness results in symplectic field
  theory}.
\newblock Springer, Heidelberg, 2014.

\bibitem[Abo11]{abouzaid}
Mohammed Abouzaid.
\newblock A cotangent fibre generates the {F}ukaya category.
\newblock {\em Adv. Math.}, 228(2):894--939, 2011.

\bibitem[AOO16]{amohol}
L.~{Amorim}, Y.-G. {Oh}, and J.~{Oliveira dos Santos}.
\newblock {Exact Lagrangian submanifolds, Lagrangian spectral invariants and
  Aubry-Mather theory}.
\newblock {\em ArXiv e-prints}, March 2016.

\bibitem[AS06]{absc}
Alberto Abbondandolo and Matthias Schwarz.
\newblock On the {F}loer homology of cotangent bundles.
\newblock {\em Comm. Pure Appl. Math.}, 59(2):254--316, 2006.

\bibitem[BCG95]{becoga}
G.~Besson, G.~Courtois, and S.~Gallot.
\newblock Entropies et rigidit\'es des espaces localement sym\'etriques de
  courbure strictement n\'egative.
\newblock {\em Geom. Funct. Anal.}, 5(5):731--799, 1995.

\bibitem[BEH{\etalchar{+}}03]{boelhowize}
F.~Bourgeois, Y.~Eliashberg, H.~Hofer, K.~Wysocki, and E.~Zehnder.
\newblock Compactness results in symplectic field theory.
\newblock {\em Geom. Topol.}, 7:799--888, 2003.

\bibitem[BEP12]{buenpo}
Lev Buhovsky, Michael Entov, and Leonid Polterovich.
\newblock Poisson brackets and symplectic invariants.
\newblock {\em Selecta Math. (N.S.)}, 18(1):89--157, 2012.

\bibitem[BHS16]{buhuse}
L.~{Buhovsky}, V.~{Humili{\`e}re}, and S.~{Seyfaddini}.
\newblock {A C\^{}0 counterexample to the Arnold conjecture}.
\newblock {\em ArXiv e-prints}, September 2016.

\bibitem[BO16]{buop}
Lev Buhovsky and Emmanuel Opshtein.
\newblock Some quantitative results in {$\cc^0$} symplectic geometry.
\newblock {\em Invent. Math.}, 205(1):1--56, 2016.

\bibitem[Che98]{chekanov}
Yu.~V. Chekanov.
\newblock Lagrangian intersections, symplectic energy, and areas of holomorphic
  curves.
\newblock {\em Duke Math. J.}, 95(1):213--226, 1998.

\bibitem[CK94]{crkl}
Christopher~B. Croke and Bruce Kleiner.
\newblock Conjugacy and rigidity for manifolds with a parallel vector field.
\newblock {\em J. Differential Geom.}, 39(3):659--680, 1994.

\bibitem[CM14]{cimo}
K.~{Cieliebak} and K.~{Mohnke}.
\newblock {Punctured holomorphic curves and Lagrangian embeddings}.
\newblock {\em ArXiv e-prints}, November 2014.

\bibitem[Cou14]{courte}
Sylvain Courte.
\newblock Contact manifolds with symplectomorphic symplectizations.
\newblock {\em Geom. Topol.}, 18(1):1--15, 2014.

\bibitem[Dra04]{dragnev}
Dragomir~L. Dragnev.
\newblock Fredholm theory and transversality for noncompact pseudoholomorphic
  maps in symplectizations.
\newblock {\em Comm. Pure Appl. Math.}, 57(6):726--763, 2004.

\bibitem[DRGI16]{digoiv}
Georgios Dimitroglou~Rizell, Elizabeth Goodman, and Alexander Ivrii.
\newblock Lagrangian isotopy of tori in {$S^2\times S^2$} and {$\Bbb{C}P^2$}.
\newblock {\em Geom. Funct. Anal.}, 26(5):1297--1358, 2016.

\bibitem[EGM16]{engame}
M.~{Entov}, Y.~{Ganor}, and C.~{Membrez}.
\newblock {Lagrangian isotopies and symplectic function theory}.
\newblock {\em ArXiv e-prints}, October 2016.

\bibitem[El{\u\i}67]{eliasson}
Halld{\'o}r~I. El{\u\i}asson.
\newblock Geometry of manifolds of maps.
\newblock {\em J. Differential Geometry}, 1:169--194, 1967.

\bibitem[Gri98]{grigoryants}
A.~A. Grigor{\cprime}yants.
\newblock On an interpretation of the torsion tensor of a connection.
\newblock {\em Izv. Vyssh. Uchebn. Zaved. Mat.}, (8):22--28, 1998.

\bibitem[Gro99]{gromov3}
Misha Gromov.
\newblock {\em Metric structures for {R}iemannian and non-{R}iemannian spaces},
  volume 152 of {\em Progress in Mathematics}.
\newblock Birkh\"auser Boston, Inc., Boston, MA, 1999.
\newblock Based on the 1981 French original [ MR0682063 (85e:53051)], With
  appendices by M. Katz, P. Pansu and S. Semmes, Translated from the French by
  Sean Michael Bates.

\bibitem[HLS15]{hulese}
Vincent Humili{\`e}re, R{\'e}mi Leclercq, and Sobhan Seyfaddini.
\newblock Coisotropic rigidity and {$C^0$}-symplectic geometry.
\newblock {\em Duke Math. J.}, 164(4):767--799, 2015.

\bibitem[HLS16]{hulese2}
Vincent Humili{\`e}re, R{\'e}mi Leclercq, and Sobhan Seyfaddini.
\newblock Reduction of symplectic homeomorphisms.
\newblock {\em Ann. Sci. \'Ec. Norm. Sup\'er. (4)}, 49(3):633--668, 2016.

\bibitem[HWZ96]{howize}
H.~Hofer, K.~Wysocki, and E.~Zehnder.
\newblock Properties of pseudoholomorphic curves in symplectisations. {I}.
  {A}symptotics.
\newblock {\em Ann. Inst. H. Poincar\'e Anal. Non Lin\'eaire}, 13(3):337--379,
  1996.

\bibitem[HWZ99]{howize2}
H.~Hofer, K.~Wysocki, and E.~Zehnder.
\newblock Properties of pseudoholomorphic curves in symplectizations. {III}.
  {F}redholm theory.
\newblock In {\em Topics in nonlinear analysis}, volume~35 of {\em Progr.
  Nonlinear Differential Equations Appl.}, pages 381--475. Birkh\"auser, Basel,
  1999.

\bibitem[{Kra}11]{kragh}
T.~{Kragh}.
\newblock {Parametrized Ring-Spectra and the Nearby Lagrangian Conjecture}.
\newblock {\em ArXiv e-prints}, July 2011.

\bibitem[LS94]{lasi}
F.~Laudenbach and J.-C. Sikorav.
\newblock Hamiltonian disjunction and limits of {L}agrangian submanifolds.
\newblock {\em Internat. Math. Res. Notices}, (4):161 ff., approx.\ 8 pp.\
  (electronic), 1994.

\bibitem[MS95]{mcsa}
Dusa McDuff and Dietmar Salamon.
\newblock {\em Introduction to symplectic topology}.
\newblock Oxford Mathematical Monographs. The Clarendon Press Oxford University
  Press, New York, 1995.
\newblock Oxford Science Publications.

\bibitem[MS12]{mcsa2}
Dusa McDuff and Dietmar Salamon.
\newblock {\em {$J$}-holomorphic curves and symplectic topology}, volume~52 of
  {\em American Mathematical Society Colloquium Publications}.
\newblock American Mathematical Society, Providence, RI, second edition, 2012.

\bibitem[MW95]{miwh}
Mario~J. Micallef and Brian White.
\newblock The structure of branch points in minimal surfaces and in
  pseudoholomorphic curves.
\newblock {\em Ann. of Math. (2)}, 141(1):35--85, 1995.

\bibitem[Ops09]{opshtein}
E.~Opshtein.
\newblock {$\cc^0$}-rigidity of characteristics in symplectic geometry.
\newblock {\em Ann. Sci. \'Ec. Norm. Sup\'er. (4)}, 42(5):857--864, 2009.

\bibitem[Pol91]{polterovich}
L.~V. Polterovich.
\newblock The {M}aslov class of the {L}agrange surfaces and {G}romov's
  pseudo-holomorphic curves.
\newblock {\em Trans. Amer. Math. Soc.}, 325(1):241--248, 1991.

\bibitem[PPS03]{paposi}
Gabriel~P. Paternain, Leonid Polterovich, and Karl~Friedrich Siburg.
\newblock Boundary rigidity for {L}agrangian submanifolds, non-removable
  intersections, and {A}ubry-{M}ather theory.
\newblock {\em Mosc. Math. J.}, 3(2):593--619, 745, 2003.
\newblock Dedicated to Vladimir I. Arnold on the occasion of his 65th birthday.

\bibitem[Sal99]{salamon}
Dietmar Salamon.
\newblock Lectures on {F}loer homology.
\newblock In {\em Symplectic geometry and topology ({P}ark {C}ity, {UT},
  1997)}, volume~7 of {\em IAS/Park City Math. Ser.}, pages 143--229. Amer.
  Math. Soc., Providence, RI, 1999.

\bibitem[Sch95]{schwarz-thesis}
Matthias Schwarz.
\newblock {\em Cohomology operations from $S^1$-cobordisms in Floer homology}.
\newblock PhD thesis, E.T.H. Zurich, 1995.

\bibitem[Sik89]{sikorav}
Jean-Claude Sikorav.
\newblock Rigidit\'e symplectique dans le cotangent de {$T^n$}.
\newblock {\em Duke Math. J.}, 59(3):759--763, 1989.

\bibitem[Sma65]{smale2}
S.~Smale.
\newblock An infinite dimensional version of {S}ard's theorem.
\newblock {\em Amer. J. Math.}, 87:861--866, 1965.

\bibitem[SW06]{sawe}
D.~A. Salamon and J.~Weber.
\newblock Floer homology and the heat flow.
\newblock {\em Geom. Funct. Anal.}, 16(5):1050--1138, 2006.

\bibitem[Wen10]{wendl}
Chris Wendl.
\newblock Automatic transversality and orbifolds of punctured holomorphic
  curves in dimension four.
\newblock {\em Comment. Math. Helv.}, 85(2):347--407, 2010.

\bibitem[{Wen}16]{wendl2}
C.~{Wendl}.
\newblock {Lectures on Symplectic Field Theory}.
\newblock {\em ArXiv e-prints}, December 2016.

\bibitem[Zeh13]{zehmisch}
Kai Zehmisch.
\newblock The codisc radius capacity.
\newblock {\em Electron. Res. Announc. Math. Sci.}, 20:77--96, 2013.

\end{thebibliography}
